\documentclass[11pt]{amsart}
\usepackage{amssymb,amsmath,mathabx,color,enumerate,hyperref,enumitem}
\usepackage{appendix}
\usepackage[top=1.5in,bottom=1.5in,a4paper]{geometry}

\newcommand{\Rat}{\mathbb{Q}}
\newcommand{\Rea}{\mathbb{R}}
\newcommand{\Nat}{\mathbb{N}}
\newcommand{\Int}{\mathbb{Z}}

\def\DD{\mathcal{D}}

\def\PP{\mathcal P}
\def\KK{\mathcal K}
\def\II{\mathcal I}

\def\BB{\mathcal B}
\def\AA{\mathcal A}

\def\UU{\mathcal{U}}

\def\ZZ{\mathcal{Z}}

\def \Id {\operatorname{Id}} 
\def \Span {\operatorname{span}}

\def \Age {\operatorname{Age}}
\def \BM {\operatorname{BM}}
\def \Aut {\operatorname{Aut}}
\def \Iso {\operatorname{Iso}}

\DeclareMathOperator{\Flim}{Flim}
\DeclareMathOperator{\Mod}{Mod}
\DeclareMathOperator{\BMG}{BMG}
\DeclareMathOperator{\Intg}{Intg}

\DeclareMathOperator{\Norm}{Norm}
\DeclareMathOperator{\PNorm}{PNorm}
\DeclareMathOperator{\Incl}{I}
\DeclareMathOperator{\Proj}{Proj}
\DeclareMathOperator{\Am}{Am}
\DeclareMathOperator{\Fr}{Fr}
\def \Ball {B}
\def \Emb {\operatorname{Emb}} 

\def\restriction{\!\upharpoonright}
\def\setsep{\colon}
\newcommand\isomtrclass[2][]{\langle #2\rangle_{\equiv}^{#1}}

\newcommand\closureOfIsomtrclass[2][]{\overline{\langle #2\rangle_{\equiv}^{#1}}}

\newcommand\closedSpan[1]{\overline{\Span} \{#1\}}
\newcommand\closedAge[1]{\operatorname{Age}_{cl}(#1)}
\newcommand{\sslash}{\mathbin{/\mkern-6mu/}}

\newtheorem{theorem}{Theorem}[section]

\newtheorem{proposition}[theorem]{Proposition}
\newtheorem{corollary}[theorem]{Corollary}

\newtheorem{lemma}[theorem]{Lemma}

\theoremstyle{definition}
\newtheorem{definition}[theorem]{Definition}

\newtheorem*{claim*}{Claim}
\newtheorem{remark}[theorem]{Remark}
\newtheorem{example}[theorem]{Example}

\newcounter{que}
\newtheorem{problem}[que]{Problem}
\newtheorem{question}[que]{Question}

\usepackage{tikz}
\usetikzlibrary{matrix}

\begin{document}

\title{Guarded Fra\"iss\'e Banach spaces}

\author[M. C\' uth]{Marek C\'uth}
\author[N. de Rancourt]{No\'e de Rancourt}
\author[M. Doucha]{Michal Doucha}

\email{cuth@karlin.mff.cuni.cz}
\email{nderancour@univ-lille.fr}
\email{doucha@math.cas.cz}

\address[M.~C\' uth]{Charles University, Faculty of Mathematics and Physics, Department of Mathematical Analysis, Sokolovsk\'a 83, 186 75 Prague 8, Czech Republic}
\address[M.~Doucha]{Institute of Mathematics of the Czech Academy of Sciences, \v{Z}itn\'a 25, 115 67 Prague 1, Czech Republic}
\address[N.~de Rancourt]{Universit\'e de Lille, CNRS, UMR 8524 – Laboratoire Paul Painlev\'e, F-59000 Lille, France}

\subjclass[2020] {46B04, 46B07, 03E15, 54E52}

\keywords{Banach spaces, descriptive set theory, Fra\"iss\'e limit, $\omega$-categorical Banach space}
\thanks{M. C\'uth was supported by the Czech Science Foundation, project no. GA\v{C}R 24-10705S. N. de Rancourt  support from the Labex CEMPI (ANR-11-LABX-0007-01). M. Doucha was supported by the GA\v{C}R project 22-07833K and by the Czech Academy of Sciences (RVO 67985840).
}

\begin{abstract}
   We characterize separable Banach spaces having $G_\delta$ isometry classes in the Polish codings $\PP$, $\PP_\infty$ and $\BB$ introduced by C\'uth--Dole\v{z}al--Doucha--Kurka \cite{CDDK} as those being \textit{guarded Fra\"iss\'e}, a weakening of the notion of Fra\"iss\'e Banach spaces defined by Ferenczi--Lopez-Abad--Mbombo--Todorcevic \cite{FLT}. We prove a Fra\"iss\'e correspondence for those spaces and make links with the notion of $\omega$-categoricity from continuous logic, showing that $\omega$-categorical Banach spaces are a natural source of guarded Fra\"iss\'e Banach spaces. Using those results, we prove that for many values of $(p, q)$, the Banach space $L_p(L_q)$ has a $G_\delta$ isometry class; we precisely characterize those values.
\end{abstract}
\maketitle

\tableofcontents 


\section{Introduction}

The famous Mazur rotation problem asks whether $\ell_2$ is the only separable infinite-dimensional Banach space whose linear isometry group acts transitively on its unit sphere (see e.g. the survey \cite{Mazursurvey}). It remains open up to now although there exist counterexamples in the non-separable setting. In fact, it is well known and easy to check that $\ell_2$ enjoys a much stronger property, namely that its linear isometry group acts transitively on the space $\Emb(E,\ell_2)$ of all isometric embeddings of $E$ into $\ell_2$, where $E\subseteq\ell_2$ is a finite-dimensional subspace. When one relaxes the transitivity condition, for a Banach space $X$, and replaces it by $\varepsilon$-transivity (see Definition \ref{def:FraisseFLMT}), and moreover replaces $\Emb(E,X)$ by the space of $K$-isomorphic embeddings of $E$ into $X$ for $K> 1$ for all finite-dimensional subspaces $E\subseteq X$, it is possible to obtain a seemingly much larger class of Banach spaces that contains the Gurari\u{\i} space and $L_p([0,1])$, for $p\neq 4,6,8,\ldots$. This led Ferenczi, Lopez-Abad, Mbombo, and Todor\v cevi\'c in \cite{FLT} to introduce Fra\"iss\'e Banach spaces, a class of Banach spaces relevant both to the Mazur problem as well, as the name suggests, as to Fra\"iss\'e theory, a theory studying structures presenting a high degree of homogeneity that originated in model theory. Despite the fact there are more examples it might still be a rather limited class of spaces and the authors of \cite{FLT} conjecture that $\ell_2$, the Gurari\u{\i} space and $L_p([0,1])$, for $p\neq 4,6,8,\ldots$ are the only Fra\"iss\'e Banach spaces. There have been recent continuations of that study in \cite{FerLopEnveloppes, FeRV23}.

Coming from a different direction, Dole\v zal, Kurka, the first and the last named author introduced in \cite{CDDK} certain natural Polish spaces of separable Banach spaces in order to investigate the topological complexity of isometry and isomorphism classes of Banach spaces as well as the complexity of other Banach space properties of interest. They proved in \cite{CDDK23} that $\ell_2$ is the unique separable infinite-dimensional Banach space whose isometry class is closed and so it turns out the best thing one might hope for in non-hilbertian spaces are $F_\sigma$ and $G_\delta$ isometry classes. There is no known example, except $\ell_2$ with a closed isometry class, of the former. On the other hand, several Banach spaces are known to have $G_\delta$ isometry classes. This was verified in \cite{CDDK23} for all $L_p([0,1])$, $p\in\left[1,\infty\right)$, and the Gurari\u{\i} space, and no other examples were known.

The goal of this paper is to merge these two research directions which is achieved by the following
\begin{enumerate}
    \item introducing a class of Banach spaces here called \emph{guarded Fra\"iss\'e Banach spaces} which is obtained by further relaxing the defining condition of Fra\"iss\'e Banach spaces (see Definition \ref{def:GuardedFraisseIntroPart1});
    \item showing that being guarded Fra\"iss\'e is equivalent with the condition of having a $G_\delta$ isometry class in the Polish spaces of Banach spaces (see Theorem \ref{thm:BMgameForIsometryClasses});
    \item demonstrating on particular examples that this class is potentially much richer and contains e.g. all $L_p(L_q)$ spaces for $1\leq p\neq q<\infty$ such that $q \neq 2$ and $q\notin (p,2)$ (see Corollary~\ref{cor:LpLq}).
\end{enumerate} 

Several remarks are now in order. This paper is intended both for Banach space theorists as well as to mathematicians interested in Fra\"iss\'e theory of metric structures. The paper is divided into two parts. The first one develops the general theory and generalizes many results that are known, in their particular context, in the discrete setting of countable structures to the setting of Banach spaces. We note that some of these results could be proved in bigger generality in the context of metric structures. Particularly interesting cases could be those of pure metric spaces and separable (unital) $C^*$-algebras. However, it is the second part with examples that gives life to the theory and which demonstrates that the case study of Banach spaces is particularly interesting and important. There we follow the well-developed and studied model theory of Banach spaces which helps us to find new examples.

Let us mention that the beginning of the first part contains a large separate introduction that will help the reader to motivate the results, put them into a context and give intuition into the proofs we give there.

\section{Notation and Preliminaries}

In the paper we use the following notation. Given a set $A$, we denote by $[A]^{<\omega}$ the family of its finite subsets. In this paper all the Banach spaces are over the field of real numbers (let us emphasize that in this paper we deal with finite-dimensional as well as with infinite-dimensional Banach spaces). Let $X$ and $Y$ be Banach spaces. The set of all the linear and continuous operators $T:X\to Y$ is denoted by $\mathcal{L}(X,Y)$. Given $K\geq 1$ and $T\in \mathcal{L}(X,Y)$ we say it is \emph{$K$-isomorphic embedding} if $K^{-1}\|x\|\leq\|Tx\|\leq K\|x\|$, $x\in X$. We say $T$ is a \emph{$(<K)$-isomorphism} if it is $K'$-isomorphic embedding for some $K'\in[1,K)$. The set of all the surjective linear isometries $T:X\to Y$ is denoted by $\Iso(X,Y)$ and we write $\Iso(X)$ instead of $\Iso(X,X)$. The set of all the isometric linear embeddings is denoted by $\Emb(X,Y)$. Further, for $C\geq 0$ we denote
\[\Emb_C(X,Y):=\{T\in \mathcal{L}(X,Y)\setsep e^{-C}\|x\|\leq \|Tx\|\leq e^C\|x\| \text{ for every } x\in X\}\text{\footnotemark},\]\footnotetext{This slightly unusual convention has the advantage that $\Emb_C(Y, Z) \circ \Emb_{C'}(X, Y) \subseteq \Emb_{C+C'}(X, Z)$, which makes some proofs easier.}
and
\[\Emb_{<C}(X,Y):=\bigcup_{0 \leqslant C' < C}\Emb_{C'}(X, Y).\]

Following the notation from \cite{CDDK}, if $K>0$, $x_1,\ldots,x_n$ are linearly independent elements of $X$ and $y_1,\ldots,y_n\in Y$, we say $(x_1,\ldots,x_n)$ and $(y_1,\ldots,y_n)$ are \emph{$K$-equivalent} if the linear operator $T:\Span\{x_1,\ldots,x_n\}\to\Span\{y_1,\ldots,y_n\}$ sending $x_i$ to $y_i$ is a $(<K)$-isomorphic embedding.

In our paper we shall quite often use the following well-known ``perturbation argument'', we refer e.g. to \cite[Lemma 4.3]{CDDK23} for the proof.

\begin{lemma}\label{lem:perturbationArgument}
Given a basis $\{e_1,\ldots,e_n\}$ of a finite-dimensional Banach space $E$, there is $C>0$ and a function $\varphi:[0,C)\to[0,\infty)$ continuous at zero with $\varphi(0)=0$ such that whenever $X$ is a Banach space with $E\subseteq X$ and $\{x_i\setsep i\leq n\}\subseteq X$ are such that $\|x_i-e_i\|<\varepsilon$, $i\leq n$, for some $\varepsilon<C$, then the linear operator $T:E\to X$ given by $T(e_i):=x_i$ is $(1+\varphi(\varepsilon))$-isomorphic embedding and $\|T-Id_E\|\leq \varphi(\varepsilon)$.
\end{lemma}

\subsection{Coding of separable Banach spaces} Following \cite{CDDK}, in this paper we use the coding of separable Banach spaces given by (pseudo)norms on $c_{00}$. Let us recall the most important notation/facts from \cite{CDDK}, further details may be found therein.

By $V$ we denote the vector space over $\Rat$ of all finitely supported sequences of rational numbers and by $\PP$ the space of all the pseudonorms on $V$ endowed with the topology of pointwise convergence, subbasis of this topology is given by sets of the form $U[v,I] :=
\{\mu \in \PP\setsep \mu(v) \in I\}$, where $v \in V$ and $I$ is an open interval. We often identify $\mu \in\PP$ with its extension to the pseudonorm on the space $c_{00}$ (the vector space over $\Rea$ of all finitely supported sequences of real numbers). Every $\mu \in \PP$ should be understood as a ``code'' of a separable Banach space $X_\mu$, which is defined as the completion of the
quotient space $X/N$, where $X = (c_{00},\mu)$ and $N = \{x \in X: \mu(x)=0\}$. We
often identify $V$ with a subset of $X_\mu$, more precisely every $v \in V$ is identified with its equivalence
class $[v]_N \in X_\mu$.

By $\PP_\infty$ we denote the set of those $\mu\in\PP$ for which the Banach space $X_\mu$ is infinite-dimensional, and by $\BB$ we denote the set of those $\mu\in\PP_\infty$ for which the extension of $\mu$ to $c_{00}$ is
an actual norm (that is, the canonical vectors $e_1,e_2,\ldots$ are linearly independent in $X_\mu$). We endow $\PP_\infty$ and $\BB$ with topologies inherited from $\PP$. It is known (and rather easy to check) that the topologies on $\PP$, $\PP_\infty$ and $\BB$ are Polish.

Given a separable Banach space $Z$ and $\mathcal{I}\subseteq \PP$ we denote
$\isomtrclass[\mathcal{I}]{Z}:=\{\mu\in\mathcal{I}\setsep X_\mu\equiv Z\}$.
If $\mathcal{I}$ is clear from the context we write $\isomtrclass{Z}$ instead of $\isomtrclass[\mathcal{I}]{Z}$.
If $\mu\in\PP$ is given, we write $\isomtrclass{\mu}$ instead of $\isomtrclass{X_\mu}$.

For $\II\in\{\PP,\PP_\infty,\BB\}$, $\mu\in\II$, a finite set $F\subseteq V$ and $\varepsilon>0$, we denote by $U_\II[\mu;F;\varepsilon]$ the set $\{\nu\in\II\setsep |\nu(f)-\mu(f)|<\varepsilon,\;f\in F\}$. If $\II$ is clear from the context we write $U[\mu;F;\varepsilon]$ instead of $U_\II[\mu;F;\varepsilon]$. When $\mu$ is fixed, the $U_{\mathcal{I}}[\mu, F, \varepsilon]$'s, for $F \subseteq V$ finite and $\varepsilon> 0$, form a basis of neighborhoods of $\mu$ in $\mathcal{I}$.

The following technical lemma, \cite[Lemma 2.4]{CDDK}, will be used several times in this paper.

\begin{lemma}\label{lem:openEquivalence}
Let $X$ be a Banach space with $\{x_1, \ldots, x_n\} \subseteq X$ linearly independant, and let $v_1, \ldots, v_n \in V$. Then for any $K > 1$, the set
    $$\{\mu \in \PP \setsep (x_1, \ldots, x_n)\in X^n \text{ and } (v_1, \ldots, v_n)\in (X_\mu)^n \text{ are $K$-equivalent}\}$$
    is open in $\PP$.
\end{lemma}

\bigskip

\subsection{Banach-Mazur compactum}

For two finite-dimensional normed spaces $X$ and $Y$ of the same dimension, let:
$$d_{\BM}(E, F) = \inf\left\{\log(\|T\|\cdot\|T^{-1}\|) \setsep T \colon E \to F \text{ is an isomorphism}\right\}.$$
For fixed $n \in \Nat$, this defines a pseudometric on the class of all $n$-dimensional normed spaces, called the \textit{Banach-Mazur pseudometric}. An easy compactness argument shows that the infimum in the above definition is always attained; in particular, $X$ and $Y$ are isometric if and only if $d_{\BM}(X, Y) = 0$.

\smallskip

For $n \in \Nat$, let $\widetilde{\BM}_n$ be the set of all norms on $\Rea^n$, and let $\BM_n$ be the quotient $\widetilde{\BM}_n/\equiv$, where we abusively write $\mu \equiv \nu$ if the spaces $(\Rea^n, \mu)$ and $(\Rea^n, \nu)$ are isometric. The pseudometric $d_{\BM}$ gives rise to a metric, that we still denote by $d_{\BM}$, on $\BM_n$. The metric space $(\BM_n, d_{\BM})$ is known to be compact and is called the \textit{$n$-dimensional Banach-Mazur compactum}, the metric $d_{\BM}$ being called the \textit{Banach-Mazur distance}. We refer the interested reader to \cite[\S 37]{BMdistanceBook}, where more details concerning the Banach-Mazur distance and Banach-Mazur compacta may be found.

\smallskip

Let $\BM = \bigcup_{n \in \Nat} \BM_n$, and endow it with the disjoint union topology. The space $\BM$ can be thought as the set of all isometry classes of finite-dimensional normed spaces. We will often abusively identify elements of $\BM$ with actual finite-dimensional normed spaces.

\smallskip

The following lemma establishes a link between convergence in $\PP$ and convergence in Banach-Mazur distance.

\begin{lemma}\label{lemma:convPBM}
    Let $(\mu_n)$ be a sequence of elements of $\PP$ converging to some $\mu \in \PP$. Let $F \subseteq V$ be finite, and suppose that $\mu$ is a norm on $\Span F$. Then, for large enough $n$, $\mu_n$ is a norm on $\Span F$, and $(\Span F, \mu_n)$ converges to $(\Span F, \mu)$ in the the space $\BM$.
\end{lemma}

The assumption that $\mu$ is a norm on $\Span F$ is essential in the above statement. If, for instance, $F$ is a nonzero singleton $\{x\}$, and if $\mu_n(x) = 1/n$, then $(\Span F, \mu_n)$ is isometric to $\mathbb{R}$ for every $n$, but $(\Span F, \mu)$ is isometric to $\{0\}$.

\begin{proof}[Proof of Lemma \ref{lemma:convPBM}]
    Without loss of generality, we can assume that $F$ is a linearly independent subset of $V$. Write $F =: \{x_1, \ldots, x_n\}$.  Since $\mu$ is a norm on $\Span F$, the set $F$ is still linearly independent in $X_\mu$. Given $\varepsilon > 0$, Lemma \ref{lem:openEquivalence} ensures that for large enough $n$, the 
    sequence $(x_1, \ldots, x_n)$ seen in $X_{\mu_n}$ and the same sequence seen in $X_\mu$ are $e^\varepsilon$-equivalent. Hence, for such an $n$, $\mu_n$ is a norm on $\Span F$ and $d_{\BM}((\Span F, \mu_n), (\Span F, \mu)) \leqslant 2\varepsilon$.
\end{proof}

\bigskip

\subsection{Classes of finite-dimensional normed spaces}

In this paper, we will often consider classes of finite-dimensional normed spaces. Those classes will be seen as subsets of $\BM$. A very important example is the \textit{age} of a Banach space, first defined in \cite{FLT}. It is inspired by the classical notion of age for a discrete structure, in the model-theoretic sense. We recall its definition below.

\begin{definition}\label{def:age}
Let $X$ be a separable Banach space. Denote by:
\begin{itemize}
    \item $\Age(X)$ its \textit{age}, that is, the class of all elements of $\BM$ that are isometric to a subspace of $X$;
    \item $\closedAge{X}$ the closure of $\Age(X)$ in the space $\BM$.
\end{itemize}
\end{definition}

Equivalently, $\closedAge{X}$ consists of those finite-dimensional normed spaces which are finitely representable in $X$.

\smallskip

The age of a Banach space always satisfies two properties, called the \textit{hereditary property} and the \textit{joint embedding property}, that are defined below in their full generality. Those two properties are classical for discrete structures, and the hereditary property was already considered for Banach spaces in \cite{FLT}. The third property defined below, called being \textit{infinite-dimensional}, is obviously satisfied by the age of any infinite-dimensional Banach space.

\begin{definition}\label{def:HP-JEP}
    Le $\KK \subseteq \BM$. Say that:
    \begin{itemize}
    \item $\mathcal{K}$ has the \textit{hereditary property (HP)}, or simply is  \textit{hereditary}, if for every $F \in \mathcal{K}$ and for every $E \in \BM$, if $E$ can be isometrically embedded into $F$, then $E \in \mathcal{K}$.
    \item $\KK$ has the \emph{joint embedding property} (JEP) if for any $F,G\in\KK$ there is $H\in\KK$ such that both $F$ and $G$ isometrically embed into $H$.
    \item $\KK$ is \emph{infinite-dimensional} if $\KK \cap \BM_n \neq \varnothing$ for every $n \in \Nat$.
    \end{itemize}
\end{definition}

\begin{definition}\label{def:sigmaK}
    For $\KK \subseteq \BM$, let $\sigma\KK:=\{\mu\in\PP\setsep \Age(X_\mu)\subseteq \KK\}$.
\end{definition}

If $X$ is a Banach space, then $\sigma\closedAge{X}$ is exactly the set of those $\mu \in \PP$ for which $X_\mu$ is finitely representable in $X$. Hence, we can rephrase \cite[Proposition 2.9]{CDDK} as follows.

\begin{proposition}\label{prop:closureOfIsomClass}
    Let $\mathcal{I} \in \{\PP, \PP_\infty, \BB\}$ and $X$ be a separable Banach space. Then $\sigma\closedAge{X} \cap \mathcal{I} = \overline{\isomtrclass[\mathcal{I}]{X}} \cap \mathcal{I}$.
\end{proposition}

It follows in particular that for every separable Banach space $X$, $\sigma\closedAge{X}$ is closed in $\PP$. This phenomenon is actually very general, as shown by next proposition.

\begin{proposition}\label{prop:closedSigmaK}
    Let $\KK \subseteq \BM$ be closed. Then $\sigma\KK$ is closed in $\PP$.
\end{proposition}

\begin{proof}
   Let $(\mu_n)$ be a sequence of elements of $\sigma\KK$ converging to some $\mu \in\PP$ and a finite-dimensional subspace $E \subseteq X_\mu$; we need to prove that $E \in \KK$. Fix a basis $(e_1, \ldots, e_d)$ of $E$. For each $1 \leqslant i \leqslant d$, let $(e_i^k)_{k \in \Nat}$ be a sequence of elements of $V$ converging to $e_i$. For each $k \in \Nat$, let $F_k = \{e_1^k, \ldots, e_d^k\} \subseteq V$. It follows from Lemma \ref{lem:perturbationArgument} that for large enough $k \in \Nat$, $\mu$ is a norm on $\Span(F_k)$, and $(\Span(F_k), \mu)_{k \in \Nat}$ converges to $E$ in the space $\BM$.

   \smallskip

   Without loss of generality, assume that for every $k \in \Nat$, $\mu$ is a norm on $\Span(F_k)$. For every $k, n \in \Nat$, since $\mu_n \in \sigma\KK$, we have $(\Span(F_k), \mu_n) \in \KK$. Moreover, by Lemma \ref{lemma:convPBM}, for every $k \in \Nat$, the sequence $(\Span(F_k), \mu_n)_{n \in \Nat}$ converges to $(\Span(F_k), \mu)$ in the space $\BM$, so by closedness of $\KK$, we deduce that $(\Span(F_k), \mu) \in \KK$. Making now $k$ tend to infinity and using once again closedness of $\KK$, we deduce that $E \in \KK$.
\end{proof}

\part{General theory of guarded Fra\"iss\'e Banach spaces}\label{FirstPart}

This first part is devoted to the study of the notion of a guarded Fra\"iss\'e Banach space and its links with descriptive complexity and Baire category in the Polish spaces $\PP$, $\PP_\infty$ and $\BB$. Given that the proofs of the three main results of this part, Theorems \ref{thm:BMgameForIsometryClasses}, \ref{thm:GuardedFraisseCorrespondence} and \ref{thm:IvanovBanach}, partly rely on each other, and the definition of our main notion is motivated by those proofs, it would be difficult to give a clear exposition of our notions and results while following the order in which proofs need to be done. For this reason, we will state and motivate our main definitions and results, without proofs, in Section \ref{sec:IntroPart1}, while the proofs of those results along with some other less important ones will be exposed in Sections \ref{sec:TechnStuffPart1} to \ref{sec:GuardedFraisseCorrespondence}.

\section{Definitions and presentation of the results}\label{sec:IntroPart1}

It has been proved by Dole\v{z}al, Kurka, the first and the last named author that the Gurari\u{\i} space $\mathbb{G}$ and the spaces $L_p$, $1\leqslant p < \infty$, all have a $G_\delta$ isometry class in the Polish spaces $\PP_\infty$ and $\BB$ (see \cite[Theorem 4.1]{CDDK} and \cite[Lemma 2.1 and Theorem 3.4]{CDDK23}; see also \cite[Corollary 4.5]{HMT23}, where this is a part of a more general result). One of our main goals in this paper is to provide new examples of Banach spaces satisfying this property, and for this, our first step is to exhibit an equivalent condition for having a $G_\delta$ isometry class. Our condition happens to be very similar to classical conditions coming from \textit{Fra\"iss\'e theory}, a field studying highly homogeneous mathematical structures and how they can be built from small (often, finitely generated) pieces. Before stating our condition, we quickly review some history and recent progress in Fra\"iss\'e theory, especially in the context of Banach spaces.

\smallskip

Fra\"iss\'e theory has been first developed for discrete structures \cite{FraisseOriginal}, originally in the setting of the \textit{theory of relations} by Fra\"iss\'e, but was mostly studied in the more general and more well-known language of model theory. In this setting, its objects of study are countable model-theoretic structures (\textit{e.g.} graphs, ordered sets, groups...) having the property that every isomorphism between two finitely-generated substructures extends to an automorphism of the whole structure; such structures are said to be \textit{ultrahomogeneous}, or sometimes simply \textit{homogeneous}. For a general introduction to the classical Fra\"iss\'e theory of model-theoretic structures, see \cite[Section 7.1]{Hodges}. While Fra\"iss\'e theory has mainly been studied by model theorists, it has been observed that many classical structures from topology and functional analysis exhibit properties that are very similar to ultrahomogeneity, and the use of Fra\"iss\'e-like methods is particularly efficient on them. One example is the Gurari\u{\i} space, see \cite{KubisSolecki}. This was one motivation to extend this theory to the setting of metric structures. This has been done in the quite general setting of continuous logic by Schoretsanitis \cite{Schoretsanitis} and Ben Yaacov \cite{BenYaacovMetricFraisse} (with slightly different approaches), and a more specific study in the case of Banach spaces was carried out recently by Ferenczi--Lopez-Abad--Mbombo--Todor\v cevi\'c \cite{FLT}, motivated in part by the \textit{Mazur rotation problem}, a well-known open problem from Banach space theory. In that paper, they define the following $\varepsilon$-$\delta$-version of ultrahomogeneity for Banach spaces.

\begin{definition}[Ferenczi--Lopez-Abad--Mbombo--Todor\v cevi\'c, \cite{FLT}]\label{def:FraisseFLMT}
    A Banach space $X$ is called \textit{weak Fra\"iss\'e} if for every finite-dimensional $E \subseteq X$ and every $\varepsilon > 0$, there exists $\delta > 0$ such that $\Iso(X)$ acts $\varepsilon$-transitively on $\Emb_\delta(E, X)$, i.e.
    for every $\phi, \psi \in \Emb_\delta(E, X)$, there exists $T \in \Iso(X)$ such that $\|\psi - T\circ \phi\| < \varepsilon$. Moreover, if in the statement above, $\delta$ only depends on $\varepsilon$ and $\dim(E)$, then $X$ is called \textit{Fra\"iss\'e}.
\end{definition}

\noindent Note that the terminology \textit{ultrahomogeneous} is avoided here because it is used in the same paper in its original sense, \textit{i.e.} without metric approximation. Examples of Fra\"iss\'e Banach spaces exhibited in \cite{FLT} are the Gurari\u{\i} space and the spaces $L_p$, for $p \in [1, \infty) \setminus \{2n \setsep n \in \Nat, \, n \geqslant 2\}$, while the $L_p$'s for $p \in \{2n \setsep n \in \Nat, \, n \geqslant 2\}$ are shown to have a slightly weaker property that is not formally named in \cite{FLT}, but will be called being \textit{cofinally Fra\"iss\'e} in the present paper (see Definition \ref{def:CofFraisse}), following terminology adopted in \cite{krKu} for discrete structures.

\smallskip

Our equivalent condition for a separable Banach space to have a $G_\delta$ isometry class is a metric version of a classical weakening of ultrahomogeneity. This weakening was first introduced in 1972 (in the model-theoretic setting) by Pabion \cite{Pabion}, who called it \textit{prehomogeneity}. For several decades, this property was sparsely studied, and its knowledge seems to have been limited to a small community. It was recently revived by the works of Kruckman \cite{KruckmanPhD}, Krawczyk--Kubi\'s \cite{krKu} and then Kubi\'s \cite{KubisWF}, who rather used the terminology of \textit{weak injectivity} or \textit{weak (ultra)homogeneity} (for a good understanding of this section, it is essential for the reader to note that the meaning of the prefix \textit{weak} used here has nothing to do with the prefix \textit{weak} from Definition \ref{def:FraisseFLMT}: this is an unfortunate coincidence!) This notion is particularly interesting because of its links with Baire category. Indeed, following Cameron \cite{Cameron}, say that a countable (relational) model-theoretic structure $X$ is \textit{ubiquitous in category} if its isomorphism class is comeager in the discrete version of what we denote in the present paper by $\sigma\Age(X)$. Here, comeagerness is expressed in the well-known topological coding $\Mod(\mathcal{L})$ of countable $\mathcal{L}$-structures (see \cite[Section 2.3]{Hjorth} for more details on $\Mod(\mathcal{L})$). An elementary argument based on the Vaught transform (see \cite[Lemma 2.50 and Notation 2.51]{Hjorth}) shows that ubiquity in category is equivalent to having a $G_\delta$ isomorphism class in $\Mod(\mathcal{L})$. The first result about ubiquity in category is due to Cameron \cite{Cameron}, who proved that every ultrahomogeneous structure is ubiquitous in category. Then, Pouzet and Roux \cite{PouzetRoux} proved that ubiquity in category is actually equivalent to weak ultrahomogeneity. The codings $\PP$, $\PP_\infty$ and $\BB$ of separable Banach spaces behaving in a very similar way as the coding $\Mod(\mathcal{L})$ of $\mathcal{L}$-structures, it was therefore natural for us to expect that having a $G_\delta$ isometry class in the codings $\PP$, $\PP_\infty$ and $\BB$ could be characterized by a property very similar to weak ultrahomogeneity.

\smallskip

In order to build such a characterization, we used a game-theoretic method inspired by Kruckman's \cite{KruckmanPhD} and Krawczyk--Kubi\'s's \cite{krKu} approaches to weak ultrahomogeneity. Recall that the \textit{Banach--Mazur game} is a well-known game used by descriptive set theorists to characterize several Baire category properties, where two players alternately choose open subsets of a given topological space in order to build an infinite descending chain (see \cite[Subsection 21.C]{Kechris}). In \cite{KruckmanPhD} and \cite{krKu}, a model-theoretic version of the Banach--Mazur game is considered, where players alternately extend a finitely generated structure, resulting in an infinite ascending chain. It is proved that a structure $X$ is weakly ultrahomogeneous if and only if the second player has a strategy in this game, played with structures embeddable into $X$, to ensure that the union of this chain is isomorphic to $X$. (Note that the same game was already considered and its relationship to ultrahomogeneity studied in \cite{KubisBM}, and that ideas used in \cite{KubisBM}, \cite{KruckmanPhD} and \cite{krKu} were already known in a different and more general form in the model-theoretic community, see \cite{HodgesGames}.) The equivalence between this model-theoretic Banach--Mazur game and the original one, played in a suitable topological space, allows one to recover Pouzet--Roux's characterization of ubiquity in category.

\smallskip

The Banach--Mazur game was even applied by Kubi\'s in the context of Banach spaces in \cite{K18} in order to obtain a new characterization of the Gurari\u{\i} space. Here we introduce a slightly different, Banach-space theoretic version of the game. This game, relative to a separable Banach space $X$ and denoted by $\BMG(X)$, is introduced in Section \ref{sec:BMGame} of the present paper. In this game, players alternately extend a finite-dimensional Banach space by way of approximate isometric embeddings (see Definition \ref{def:BMGame}). On one hand, we show by following standard descriptive set-theoretic techniques that if $\isomtrclass[\mathcal{B}]{X}$ is comeager in its closure in $\mathcal{B}$, then Player II has a winning strategy in $\BMG(X)$ (see Proposition \ref{prop:comeagerAndBMGame}). On the other hand, following the same scheme of proof as in \cite{krKu}, we exhibit a Fra\"iss\'e-like property of Banach spaces that is implied by the non-existence of a winning strategy for Player I in $\BMG(X)$ (see Proposition \ref{prop:winningI}). We call a Banach space having this Fra\"iss\'e-like property a \textit{guarded Fra\"iss\'e Banach space}. There are several equivalent expressions of the guarded Fra\"iss\'e property, whose equivalence is proved in Section \ref{sec:Ultrahom} (see Theorem \ref{thm:EquivalenceGFraGUH}), all of them being quite technical. We present two of them below, to illustrate their respective similarities with notions from previously mentioned works in Fra\"iss\'e theory.

\begin{definition}\label{def:GuardedFraisseIntroPart1}
        A separable Banach space $X$ is said to be \textit{guarded Fra\"iss\'e} if it satisfies one of the following two equivalent properties.
    \begin{enumerate}[label=(gF-\arabic*), series=gFraisse]    
    \item\label{it:gFraisseAction} For every finite-dimensional subspace $E\subseteq X$ and every $\varepsilon>0$, there exist $\delta>0$, and a finite dimensional subspace $F\subseteq X$ with $E \subseteq F$, such that $\Iso(X)$ acts $\varepsilon$-transitively on $\{\iota|E\setsep \iota\in \Emb_\delta(F,X)\}$, that is, for every $\iota,\iota'\in \Emb_\delta(F,X)$ there exists $T\in\Iso(X)$ such that $\|\iota'\restriction_{E} - T\circ\iota\restriction_{E}\|<\varepsilon$.
    
    \item\label{it:gFraisseForPlayerI} For every finite-dimensional subspace $E \subseteq X$ and every $\varepsilon>0$, there exist $F\in \closedAge{X}$, $\phi\in\Emb_{\varepsilon}(E,F)$ and $\delta>0$ such that for every $G\in\closedAge{X}$, $\psi\in\Emb_{\delta}(F,G)$  and $\eta>0$, there exists $\iota\in\Emb_{\eta}(G,X)$ such that $\|\Id_{E} - \iota\circ \psi\circ \phi\|<\varepsilon$.
    \end{enumerate}
\end{definition}

\begin{center}\includegraphics[scale=0.7]{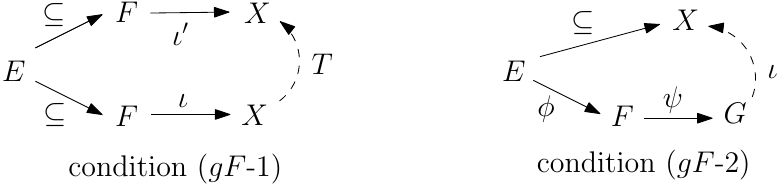}\end{center}

\noindent Property \ref{it:gFraisseForPlayerI} is the one we obtain in Section \ref{sec:BMGame} from the non-existence of a winning strategy for Player I in $\BMG(X)$. It is an $\varepsilon$-$\delta$-version of the property called \textit{weak injectivity} by Krawczyk--Kubi\'s \cite{krKu} and Kubi\'s \cite{KubisWF} and \textit{weak homogeneity} by Kruckman \cite{KruckmanPhD}. The seemingly stronger property \ref{it:gFraisseAction} is an $\varepsilon$-$\delta$-version of the property called \textit{weak homogeneity} by Kubi\'s \cite{KubisWF} and \textit{weak ultrahomogeneity} by Kruckman \cite{KruckmanPhD}. Despite those similarities, we chose not to follow Kruckman, Krawczyk and Kubi\'s's convention to add the prefix \textit{weak} everywhere to name our new notion, in order to avoid confusion with Ferenczi--Lopez-Abad--Mbombo--Todorcevic's notion of weak Fra\"iss\'e Banach spaces defined earlier. The terminology \textit{guarded Fra\"iss\'e} was kindly suggested to us by Dragan Ma\v{s}ulovi\'c, the \textit{guard} being, in property \ref{it:gFraisseForPlayerI}, the map $\phi$, and in property \ref{it:gFraisseAction}, the inclusion map $E \to F$. Also note that guarded Fra\"iss\'eness generalizes weak Fra\"iss\'eness of separable Banach spaces, as shown by the following proposition.

\begin{proposition}
    Every separable weak Fra\"iss\'e Banach space is guarded Fra\"iss\'e.
\end{proposition}

\begin{proof}
    The property defining weak Fra\"iss\'eness is obtained by taking $F = E$ in property \ref{it:gFraisseAction}.
\end{proof}

We are now ready to state the first of the three main results of Part \ref{FirstPart}, our characterization of separable Banach spaces having a $G_\delta$ isometry class.

\begin{theorem}\label{thm:BMgameForIsometryClasses}
Let $X$ be a separable Banach space and let $\mathcal{I} \in \{\PP, \PP_\infty, \BB\}$ if $\dim(X) = \infty$, or $\mathcal{I} = \PP$ if $\dim(X) < \infty$. The following conditions are equivalent:
\begin{enumerate}
    \item $\isomtrclass[\mathcal{I}]{X}$ is $G_\delta$ in $\mathcal{I}$;
    \item $\isomtrclass[\mathcal{I}]{X}$ is comeager in its closure in $\mathcal{I}$;
    \item Player II has a winning strategy in $\BMG(X)$;
    \item Player I does not have a winning strategy in $\BMG(X)$;
    \item $X$ is a guarded Fra\"iss\'e Banach space.
\end{enumerate} 
\end{theorem}

Before going further, let us comment on the techniques used to prove each of the implications in the above theorem. The implications (1) $\Rightarrow$ (2) and (3) $\Rightarrow$ (4) are obvious. As mentioned earlier, to prove (2) $\Rightarrow$ (3) we followed the classical proof of the counterpart of this implication for the standard (topological) Banach--Mazur game, and to prove (4) $\Rightarrow$ (5) we followed the proof from \cite{krKu} of the discrete analogue of this implication. Note that the converse implication (5) $\Rightarrow$ (3) is also proved in \cite{krKu} in the case of discrete structures, and that the analogue of (3) $\Rightarrow$ (2) is standard for the topological Banach--Mazur game; in an early draft of the present paper, we followed those proofs to prove the implication (5) $\Rightarrow$ (2) in Theorem \ref{thm:BMgameForIsometryClasses}. Since, as we were informed by Malicki, a simple proof of (2) $\Rightarrow$ (1) can be given using variants of Vaught transform arguments developed in \cite{BDN} (and this is made explicit in a more general context in \cite[Theorem 1.1]{HMT23}), we could have proved the implication (5) $\Rightarrow$ (1) and hence the whole Theorem \ref{thm:BMgameForIsometryClasses} using only the above mentioned ideas. However, we chose to prove (5) $\Rightarrow$ (1) in a completely different way, allowing us to simultaneously obtain another important result: the \textit{Fra\"iss\'e correspondence} for guarded Fra\"iss\'e Banach spaces.

\smallskip

The Fra\"iss\'e correspondence is one of the main features of Fra\"iss\'e theory. It first appeared, for discrete structures, in Fra\"issé's seminal paper \cite{FraisseOriginal} (in the language of the theory of relations; once again, see \cite[Section 7.1]{Hodges} for a modern exposition in the language of model theory). In that setting, it establishes that the operation of taking the age of a structure induces a bijection between the collection of all countable ultrahomogeneous structures (considered up to isomorphism) and the collection of all so-called \textit{Fra\"iss\'e classes}, that are, classes of (isomorphism types of) finite structures satisfying (HP), (JEP), and another property called the \textit{amalgamation property (AP)} (here, the notion of \textit{age} and the properties \textit{(HP)} and \textit{(JEP)} have to be understood as the discrete analogues of those defined for Banach spaces in Definitions \ref{def:age} and \ref{def:HP-JEP}). 

\smallskip
 
The inverse bijection of the age operation is called the \textit{Fra\"iss\'e limit} operation. The main interest of the Fra\"iss\'e correspondence is that this operation offers a way to \textit{construct} interesting ultrahomogeneous structures from classes of finite structures; it is actually, for many ultrahomogeneous structures, the simplest way to construct them. One example is the Gurari\u{\i} space: while it is likely that Gurari\u{\i} has no knowledge of Fra\"iss\'e's work when first building it, the construction he gave in \cite{Gurarii} is very similar to Fra\"iss\'e's construction of Fra\"iss\'e limits.

\smallskip

Previously cited works \cite{Schoretsanitis, BenYaacovMetricFraisse, FLT} extending Fra\"iss\'e theory to metric contexts all contain a version of the Fra\"iss\'e correspondence; we present below in more details the one proved by Ferenczi--Lopez-Abad--Mbombo--Todorcevic \cite{FLT} for Banach spaces. The relevant notion of \textit{Fra\"iss\'e classes} in that context is defined using an $\varepsilon$-$\delta$-version of the amalgamation property; classes of finite-dimensional normed spaces satisfying this $\varepsilon$-$\delta$-version are called \textit{amalgamation classes} in \cite{FLT} (the terminology \textit{amalgamation property} being avoided since it is already used in the same paper in its original sense, i.e. without metric approximation). We reproduce below the definition of a Fra\"iss\'e class of finite-dimensional normed spaces from \cite{FLT}.

\begin{definition}[Ferenczi--Lopez-Abad--Mbombo--Todorcevic, \cite{FLT}]
    Say that a class $\mathcal{K} \subseteq \BM$ is :
    \begin{itemize}
        \item a \textit{weak amalgamation class} if for every $F\in \mathcal{K}$ and $\varepsilon > 0$, there exists $\delta > 0$ such that for every $G, H \in \mathcal{K}$, $\psi_G \in \Emb_\delta(F, G)$ and $\psi_H \in \Emb_\delta(F, H)$, there exist $K \in \mathcal{K}$, $\iota_G \in \Emb(G, K)$ and $\iota_H \in \Emb(H, K)$ for which $\|\iota_G \circ \psi_G - \iota_H\circ\psi_H\| < \varepsilon$;
        \item an \textit{amalgamation class} if it is a weak amalgamation class with, in the above definition, $\delta$ depending only on $\varepsilon$ and $\dim(F)$;
        \item a \textit{Fra\"iss\'e class} if it is nonempty, closed in $\BM$, hereditary, and is an amalgamation class.
    \end{itemize}
\end{definition}

\noindent Note that every weak amalgamation class of finite-dimensional normed spaces satisfies (JEP) (just take $F=0$ in the above definition), hence it is not necessary to include it in the definition of a Fra\"iss\'e class. The Fra\"iss\'e  correspondence proved in \cite{FLT} is the following.

\begin{theorem}[Ferenczi--Lopez-Abad--Mbombo--Todorcevic, \cite{FLT}]\label{thm:FraisseCorrespFLMT}
$X \mapsto \Age(X)$ induces a bijection between the collection of all separable Fra\"iss\'e Banach spaces (considered up to isometry) and the collection of all Fra\"iss\'e classes of finite-dimensional normed spaces. 
\end{theorem}

\noindent In particular, if $\mathcal{K} \subseteq \BM$ is a Fra\"iss\'e class, then there exists a unique Fra\"iss\'e Banach space $X$ such that $\Age(X) = \mathcal{K}$; this space $X$ is called the \textit{Fra\"iss\'e limit} of $\mathcal{K}$ and is denoted by $\Flim(\mathcal{K})$.

\smallskip

In the discrete setting, a version of the Fra\"iss\'e correspondence is known for weakly ultrahomogeneous structures, involving a weakening of the amalgamation property. This weakening was first introduced by Ivanov \cite{Ivanov} who called it the \textit{almost amalgamation property}; it was then independently rediscovered by Kechris--Rosendal \cite{KechrisRosendal} who called it the \textit{weak amalgamation property (WAP)}, the main terminology in use nowadays. In both articles \cite{Ivanov} and \cite{KechrisRosendal}, this property was used to characterize the existence of a comeager isomorphism class in discretes analogues of what we call here $\sigma\mathcal{K}$, $\mathcal{K}$ being a class of finite structures or finite partial expansions of a given structure. Following \cite{krKu}, call a \textit{weak Fra\"iss\'e class} a class of (isomorphism types of) finite structures satisfying (the discrete versions of) (HP), (JEP), and (WAP); then the so-called \textit{weak Fra\"iss\'e correspondence} asserts that the operation of taking the age induces a bijection between the collection of all weakly ultrahomogeneous structures (considered up to isomorphism) and the collection of all weak Fra\"iss\'e classes. Although this correspondence seems to have been, for long, a folklore knowledge in a small community, and its proof implicitely appeared in \cite{Ivanov}, its first explicit statement and proof appeared, up to the authors' knowledge, in Kruckman's Ph.D. thesis \cite{KruckmanPhD} (where the terminology \textit{generalized Fra\"iss\'e classes} was rather used), and independently in Krawczyk--Kubi\'s paper \cite{krKu}. Given a weak Fra\"iss\'e class $\mathcal{K}$ of finite structures, the unique weakly ultrahomogeneous structure $X$ whose age is $\mathcal{K}$ has been called the \textit{generic limit} of $\mathcal{K}$ in \cite{KruckmanPhD} and \cite{KubisWF}, an simply its \textit{limit} in \cite{krKu}; in the present paper, we will stick to the \textit{Fra\"iss\'e limit} terminology. 

\smallskip

The second main result of this Part \ref{FirstPart} is a Fra\"iss\'e correspondence for guarded Fra\"iss\'e Banach spaces. 
Before stating it, we need to define a Banach space version of (WAP); for coherence, it will be called the \textit{guarded amalgamation property}. It is defined very similarly as the property of being a weak amalgamation class for Banach spaces (do not confuse it here with the weak amalgamation property for discrete structures), the main difference being, again, the presence of a ``guard'', which is, in the definition below, the map $\phi$.

\begin{definition}
    Let $\mathcal{K} \subseteq \BM$. Say that:
    \begin{itemize}
        \item $\mathcal{K}$ has the \textit{guarded amalgamation property (GAP)} if for every $E \in \mathcal{K}$ and $\varepsilon > 0$, there exist $F\in \mathcal{K}$, $\phi \in \Emb_\varepsilon(E, F)$ and $\delta > 0$ such that for every $G, H \in \mathcal{K}$, $\psi_G \in \Emb_\delta(F, G)$ and $\psi_H \in \Emb_\delta(F, H)$, there exist $K \in \mathcal{K}$, $\iota_G \in \Emb(G, K)$ and $\iota_H \in \Emb(H, K)$ for which $\|\iota_G \circ \psi_G \circ \phi - \iota_H\circ\psi_H \circ \phi\| < \varepsilon$;
        \begin{center}\includegraphics[scale=0.7]{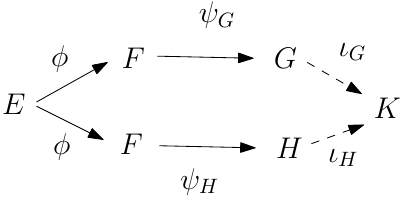}\end{center}
        
        \item $\mathcal{K}$ is a \textit{guarded Fra\"iss\'e class} if it is nonempty, closed in $\BM$, hereditary, and satisfies (JEP) and (GAP).
    \end{itemize}
\end{definition}

\noindent The proposition below describes how those notions relate to those from \cite{FLT}.

\begin{proposition}
    Every weak amalgamation class of finite-dimensional normed spaces satisfies (GAP). Every Fra\"iss\'e class of finite-dimensional normed spaces is a guarded Fra\"iss\'e class.
\end{proposition}

\begin{proof}
    The property of being a weak amalgamation class is obtained by taking $F=E$ and $\phi = \Id_E$ in the definition of (GAP). As previously mentioned, all weak amalgamation classes satisfy (JEP).
\end{proof}

\noindent The second main result of Part \ref{FirstPart} is the following.

\begin{theorem}\label{thm:GuardedFraisseCorrespondence}
   $X \mapsto \closedAge{X}$ induces a bijection between the collection of all guarded Fra\"iss\'e Banach spaces (considered up to isometry) and the collection of all guarded Fra\"iss\'e classes of finite-dimensional normed spaces. 
\end{theorem}

\noindent The fact that our Fra\"iss\'e correspondence involves the operation $X \mapsto \closedAge{X}$, and not $X \mapsto \Age(X)$ like in Theorem \ref{thm:FraisseCorrespFLMT}, might look surprising at first sight. However, since all Fra\"iss\'e Banach spaces have a closed age (see \cite[Theorem 2.12]{FLT}) those two operations are the same for those spaces. Theorem \ref{thm:GuardedFraisseCorrespondence} shows that $X \mapsto \closedAge{X}$ is the right generalization in the context of guarded Fra\"iss\'e Banach spaces. Theorem \ref{thm:GuardedFraisseCorrespondence} allows one to define the notion of the Fra\"iss\'e limit of a guarded Fra\"iss\'e class of finite-dimensional normed spaces. The notation $\Flim$ used in the definition below induces no confusion, since our notion extends the notion of Fra\"iss\'e limit defined in \cite{FLT}.

\begin{definition}
    Let $\mathcal{K} \subseteq \BM$ we a guarded Fra\"iss\'e class. The unique guarded Fra\"iss\'e Banach space $X$ (up to isometry) such that $\closedAge{X} = \mathcal{K}$ is called the \textit{Fra\"iss\'e limit} of the class $\mathcal{K}$ and is denoted by $\Flim(\mathcal{K})$. 
    Further, given a separable Banach space $X$ such that $\closedAge{X}$ is guarded Fra\"iss\'e, the space $\Flim(\closedAge{X})$ will be, in order to simplify the notation, denoted by $\Flim(X)$.
\end{definition}

We finally state the last main result of Part \ref{FirstPart}, a version of Ivanov's \cite{Ivanov} and Kechris--Rosendal's \cite{KechrisRosendal} characterization of the existence of a comeager isomorphism class in $\sigma \mathcal{K}$ in the Banach space context.

\begin{theorem}\label{thm:IvanovBanach}
    Let $\KK \subseteq \BM$ be a nonempty, closed and hereditary class satisfying (JEP). Let $\mathcal{I} \in \{\PP, \PP_\infty, \BB\}$ if $\KK$ is infinite-dimensional, or $\mathcal{I} = \PP$ otherwise. The following conditions are equivalent:
    \begin{enumerate}
        \item $\KK$ is a guarded Fra\"iss\'e class;
        \item there exists $\mu \in \sigma\KK\cap \mathcal{I}$ such that $\isomtrclass[\mathcal{I}]{\mu}$ is comeager in $(\sigma\KK) \cap \mathcal{I}$.
    \end{enumerate}
    Moreover, if those condition are satisfied, then the pseudonorm $\mu$ from condition (2) can be chosen in such a way that $X_\mu = \Flim(\KK)$.
\end{theorem}

While Theorems \ref{thm:BMgameForIsometryClasses}, \ref{thm:GuardedFraisseCorrespondence} and \ref{thm:IvanovBanach} can be stated separately, their proofs involve common arguments and are so intertwined that it would be difficult to carry them out separately. The rest of Part \ref{FirstPart} is devoted to their proofs, which will be carried out simultaneously. This part will be organized as follows. Section \ref{sec:TechnStuffPart1} contains technical results about the coding $\PP$ of separable Banach spaces. In Section \ref{sec:BMGame}, we define our version of the Banach--Mazur game and prove the implications (1) $\Rightarrow$ (2) $\Rightarrow$ (3) $\Rightarrow$ (4) $\Rightarrow$ (5); the definition of a guarded Fra\"iss\'e Banach space, under the form \ref{it:gFraisseForPlayerI}, will be motivated by the proof of (4) $\Rightarrow$ (5). In Section \ref{sec:Ultrahom}, we prove a technical result (Proposition \ref{prop:backAndForthConstruction}) having two main consequences: one one hand, the equivalence between the different expressions of the guarded Fra\"issé property given in Definition \ref{def:GuardedFraisseIntroPart1} (see Theorem \ref{thm:EquivalenceGFraGUH}) and on the other hand, isometric uniqueness of guarded Fra\"iss\'e Banach spaces having a given age. In Section \ref{sec:gFClass}, we develop the basic theory of guarded Fra\"iss\'e classes. Finally, in Section \ref{sec:GuardedFraisseCorrespondence}, using results from previous sections, we prove Theorems \ref{thm:GuardedFraisseCorrespondence}, \ref{thm:IvanovBanach}, and implication (5) $\Rightarrow$ (1) in Theorem \ref{thm:BMgameForIsometryClasses}; we also define and consider the notion of cofinally Fra\"iss\'e Banach spaces there.

\section{Preliminaries concerning our coding \texorpdfstring{$\PP$}{P}}\label{sec:TechnStuffPart1}

In this section, we prove preliminary results on spaces $\PP$, $\PP_\infty$ and $\BB$ that will be used in the next sections. We start with two lemmas allowing one to transfer Baire category statements from the codings $\PP$ and $\PP_\infty$ to the coding $\BB$.

\begin{lemma}\label{lem:equalClosures}
    Let $X$ be a separable, infinite-dimensional Banach space. Then $\closureOfIsomtrclass[\BB]{X} = \closureOfIsomtrclass[\PP]{X}$.
\end{lemma}

\begin{proof}
    The left to right inclusion is obvious. We prove the right-to-left one by picking $\mu \in \isomtrclass[\PP]{X}$ and showing that $\mu \in \closureOfIsomtrclass[\BB]{X}$. For this, fix a finite set $A \subseteq V$ and $\varepsilon > 0$; we need to show that $U_\PP[\mu, A, \varepsilon] \cap \isomtrclass[\BB]{X} \neq \varnothing$. Write $\Span_\Rea(A) = E \oplus F$, where $E$ and $F$ are two vectors subspaces such that $\mu$ induces a norm on $E$ and $\mu\restriction_F = 0$. Let $Y$ be a vector subspace of $c_{00}$ such that $c_{00} = E \oplus F \oplus Y$. Let $(e_n)_{n \in \Nat}$ be a basis of $F \oplus Y$. Since $\mu \in \isomtrclass[\PP]{X}$, we have $(E, \mu\restriction_E) \in \Age(X)$, so one can find $\phi \in \Emb((E, \mu\restriction_E), X)$. Let $Z \subseteq X$ be a vector subspace such that $X = \phi(E) \oplus Z$, and let $(f_n)_{n \in \Nat}$ be a linearly independent family of elements of $Z$ having dense span. Given $\delta > 0$, let $T_\delta \colon c_{00} \to X$ be the unique linear map extending $\phi$ and such that for every $n \in \Nat$, $T_\delta(e_n) = \delta f_n$. This map being injective, we can define $\nu_\delta \in \BB$ by $\nu_\delta(x) \coloneq \|T_\delta(x)\|$ for every $x \in c_{00}$. This way, $T_\delta \colon (c_{00}, \nu_\delta) \to X$ is an isometric embedding with dense range, so it extends to an onto isometry $X_{\nu_\delta} \to X$; in particular for every $\delta > 0$, one has $\nu_\delta \in \isomtrclass[\BB]{X}$. For every $\delta > 0$, since $T_\delta$ induces an isometric embedding $(E, \mu\restriction_E) \to X$, the pseudonorms $\mu$ and $\nu_\delta$ coincide on $E$; moreover, when $\delta \to 0$, $\nu_\delta\restriction_{F \oplus Y}$ tends to $0$ pointwise. So, given $x \in E$ and $y \in F$, one has:
    \begin{align*}
        |\nu_\delta(x+y) - \mu(x+y)| &\leqslant |\nu_\delta(x+y) - \nu_\delta(x)| + |\nu_\delta(x) - \mu(x)| + |\mu(x) - \mu(x+y)|\\
        &\leqslant \nu_\delta(y) + 0 + \mu(y)\\
        & = \nu_\delta(y) \to 0
    \end{align*}
    so $\nu_\delta$ tends to $\mu$ pointwise on $E \oplus F$, hence on $A$. Thus, for $\delta > 0$ small enough, one has $\nu_\delta \in U_\PP[\mu, A, \varepsilon] \cap \isomtrclass[\BB]{X}$.
\end{proof}

\begin{lemma}\label{lem:reflectingCategory}
    Let $X$ be a separable, infinite-dimensional Banach space, and $\mathcal{I} \in \{\PP, \PP_\infty\}$. Suppose that $\isomtrclass[\mathcal{I}]{X}$ is comeager in $\closureOfIsomtrclass[\mathcal{I}]{X} \cap \mathcal{I}$. Then $\isomtrclass[\mathcal{\BB}]{X}$ is comeager in $\closureOfIsomtrclass[\BB]{X} \cap \BB$.
\end{lemma}

\begin{proof}
    To make notation simpler, write $\mathcal{I}_X \coloneq \closureOfIsomtrclass[\mathcal{I}]{X} \cap \mathcal{I}$ and $\BB_X \coloneq \closureOfIsomtrclass[\BB]{X} \cap \BB$. Let $G \subseteq \isomtrclass[\mathcal{I}]{X}$ be a dense $G_\delta$ subset of the topological space $\mathcal{I}_X$. Since $\BB$ is a $G_\delta$ subset of $\PP$ (see \cite[Lemma 2.5]{CDDK}), the set $\BB_X$ is a $G_\delta$ subspace of the space $\mathcal{I}_X$. Moreover, $\BB_X$ contains $\isomtrclass[\BB]{X}$ which is dense in $\mathcal{I}_X$ by Lemma \ref{lem:equalClosures}; thus, $\BB_X$ is a dense $G_\delta$ subset of the space $\mathcal{I}_X$. It follows that $G' \coloneq G \cap \BB_X$ is a dense $G_\delta$ subset of $\mathcal{I}_X$ as well, so in particular it is a dense $G_\delta$ subset of $\BB_X$. Since $G' \subseteq \isomtrclass[\BB]{X}$, we deduce that $\isomtrclass[\BB]{X}$ is comeager in $\BB_X$.
\end{proof}

\smallskip

If $E$ is a finite-dimensional vector subspace of $V$, we will let $E_\Rea \coloneq \Span_\Rea(E) \subseteq c_{00}$, denote by $\PNorm(E)$ the set of all pseudonorms on $E_\Rea$, and by $\Norm(E)$ the subset of $\PNorm(E)$ whose elements are norms on $E_\Rea$. Elements of $\PNorm(E)$ will often be identified with their restrictions to $E$. The following very simple result allows one to ``represent'' finite-dimensional Banach spaces as subspaces of $c_{00}$ in a constrained way.

\begin{lemma}\label{lem:simpleReflection}
    Let $E \subseteq V$ be a finite-dimensional vector subspace, $F$ be a finite-dimensional Banach space, and $\phi \colon E_\Rea \to F$ be a linear map. Then there exist a finite-dimensional vector subspace $G \subseteq V$ containing $E$, $\mu \in \PNorm(G)$, and an onto isometry $\psi \colon (G_\Rea, \mu) \to F$ extending $\phi$. Moreover, if $\phi$ is injective, then $\mu$ can be chosen in $\Norm(G)$.
\end{lemma}

\begin{proof}
Let $H \subseteq F$ be a subspace such that $F = \phi(E_\Rea) \oplus H$. Let $K \subseteq V$ be a vector subspace such that $E_\Rea \cap K_\Rea = \varnothing$ and $\dim(K_\Rea) = \dim(H)$. Let $G \coloneq E \oplus K$ and consider a linear map $\psi \colon G_\Rea \to F$ such that $\psi\restriction_{E_\Rea} = \phi$ and $\psi$ induces a bijection $K_\Rea \to H$. In particular, $\psi$ is onto. We now define $\mu \in \PNorm(G)$ by $\mu(x) = \|\psi(x)\|$ for every $x \in G_\Rea$; in this way $\psi \colon (G_\Rea, \mu) \to F$ is an isometry. If moreover $\phi$ is injective, then $\psi$ is a bijection, so $\mu \in \Norm(G)$.
\end{proof}

In the rest of this section, we introduce and study alternative bases of neighborhoods in the space $\BB$ that are sometimes easier to manipulate than the $U_\BB[\mu, F, \varepsilon]$'s. Partial results for the codings $\PP$ and $\PP_\infty$ will also be obtained at the same time.

\begin{definition}
    Given $\mathcal{I} \in \{\PP, \PP_\infty, \BB\}$, a finite-dimensional vector subspace $E \subseteq V$, a pseudonorm $\mu \in \PNorm(E)$, and $\varepsilon > 0$, let $V_\mathcal{I}[\mu, E, \varepsilon]$ be the set of all $\nu \in \mathcal{I}$ for which the inclusion map $E_\Rea \to X_\nu$ belongs to $\Emb_{< \varepsilon}((E_\Rea, \mu), X_\nu)$.
\end{definition}

\begin{proposition}\label{prop:equivalentNormOpen}
Let $\mathcal{I} \in \{\PP, \PP_\infty, \BB\}$, $E\subseteq V$ be a finite-dimensional vector subspace, $\mu \in \Norm(E)$, and $\varepsilon > 0$. Then $V_\mathcal{I}[\mu, E, \varepsilon]$ is open in $\mathcal{I}$.
\end{proposition}

\begin{proof}
One can assume $\mathcal{I} = \PP$. Fix a basis $(e_1, \ldots, e_n)$ of $E$. Then $V_\PP[\mu, E, \varepsilon]$ is the set of all $\nu \in \PP$ such that $(e_1, \ldots, e_n)$, seen as a sequence in $X_\nu$, is $e^\varepsilon$-equivalent to itself seen as a sequence in $(E_\Rea, \mu)$. So the result follows from Lemma \ref{lem:openEquivalence}.
\end{proof}

\begin{proposition}\label{prop:alternativeBasisPrelim}
Let $\mathcal{I} \in \{\PP, \PP_\infty, \BB\}$, $\mu \in \mathcal{I}$, and $U$ be a neighborhood of $\mu$ in $\mathcal{I}$. Then there exists a finite-dimensional subspace $E \subseteq V$ and $\varepsilon > 0$ such that $V_\mathcal{I}[\mu\restriction_E, E, \varepsilon] \subseteq U$.
\end{proposition}

\begin{proof}
We can assume that $U = U_\mathcal{I}[\mu, A, \delta]$ for some finite set $A \subseteq V$ and some $\delta > 0$. Let $E \coloneq \Span_\Rat(A)$ and let $\varepsilon > 0$ be such that for every $x \in A$, $(e^\varepsilon - 1) \mu(x) < \delta$ and $(1 - e^{-\varepsilon})\mu(x) < \delta$. Then if $\nu \in V_\mathcal{I}[\mu\restriction_E, E, \varepsilon]$, one has for each $x \in A$:
$$\mu(x) - \delta < \mu(x) - (1-e^{-\varepsilon})\mu(x) = e^{-\varepsilon}\mu(x) \leqslant \nu(x) \leqslant e^{\varepsilon}\mu(x) = \mu(x) + (e^\varepsilon - 1)\mu(x) < \mu(x) + \delta,$$
thus $|\nu(x) - \mu(x)| < \delta$, and $\nu \in U$.
\end{proof}

\begin{corollary}\label{cor:alternativeBasis}
    Fix $\mu \in \BB$. Sets of the form $V_\BB[\mu\restriction_E, E, \varepsilon]$, where $E\subseteq V$ is a finite-dimensional vector subspace and $\varepsilon > 0$, form a basis of neighborhoods of $\mu$ in $\mathcal{B}$. 
\end{corollary}

\begin{proof}
This is an immediate consequence of Propositions \ref{prop:equivalentNormOpen} and \ref{prop:alternativeBasisPrelim} for $\mathcal{I} = \BB$.
\end{proof}

\begin{lemma}\label{lem:traceOnIsomClass}
    Let $X$ be an infinite-dimensional Banach space, $E \subseteq V$ be a finite-dimensional vector space, $\mu \in \Norm(E)$, and $\varepsilon > 0$. If $(E_\Rea, \mu) \in \closedAge{X}$, then $V_\BB[\mu, E, \varepsilon]\cap \isomtrclass[\BB]{X} \neq \varnothing$.
\end{lemma}

\begin{proof}
    Since $(E_\Rea, \mu) \in \closedAge{X}$, one can find $\phi \in \Emb_{<\varepsilon}((E_\Rea, \mu), X)$. Let $W \subseteq c_{00}$ and $Y \subseteq X$ be two vector subspaces such that $c_{00} = E_\Rea \oplus W$ and $X = \phi(E_\Rea) \oplus Y$. Let $(e_n)_{n \in \Nat}$ be a basis of $W$ and $(f_n)_{n \in \Nat}$ be a linearly independant family of elements of $Y$ having a dense span in $Y$. Let $T \colon c_{00} \to X$ be the unique linear map that maps $e_n$ to $f_n$ for each $n \in \Nat$ and extends $\phi$. Since $T$ is injective, $\nu(x) \coloneq \|T(x)\|$ defines a norm on $c_{00}$, that is, $\nu \in \BB$. The map $T \colon (c_{00}, \nu) \to X$ is by construction an isometry with dense range, thus it extends to an onto isometry $X_\nu \to X$, that we will still denote by $T$. It follows from the existence of $T$ that $\nu \in \isomtrclass[\BB]{X}$. Moreover $T^{-1} \circ \phi \in \Emb_{< \varepsilon}((E_\Rea, \mu), X_\nu)$, and since this map is the inclusion map $E_\Rea \to X_\nu$, we deduce that $\nu \in V_\BB[\mu, E, \varepsilon]$.
\end{proof}

\section{Variant of the Banach-Mazur game}\label{sec:BMGame}

In this section, we introduce our Banach-space theoretic version of the Banach-Mazur game, and we prove the implications (1) $\Rightarrow$ (2) $\Rightarrow$ (3) $\Rightarrow$ (4) $\Rightarrow$ (5) in Theorem \ref{thm:BMgameForIsometryClasses} (where (5) will be here understood as ``$X$ satisfies \ref{it:gFraisseForPlayerI}'', since the equivalence between the two properties \ref{it:gFraisseForPlayerI} and \ref{it:gFraisseAction} used to define the notion of a guarded Fra\"iss\'e Banach space in Definition \ref{def:GuardedFraisseIntroPart1} hasn't be proved yet). The implication (2) $\Rightarrow$ (3) is proved in Proposition \ref{prop:comeagerAndBMGame}, the implication (4) $\Rightarrow$ (5) is proved in Proposition \ref{prop:winningI}, and the two remaining ones are obvious.

\begin{definition}\label{def:directedLimit}Let $(E_n)$ be a sequence of Banach spaces and $f_n\in\mathcal{L}(E_n,E_{n+1})$ be a sequence of linear maps satisfying that $\prod_{n=1}^\infty \|f_n\| < \infty$. For $n < k$ we shall use the notation $f_{n,k}:=f_{k-1}\circ f_{k-2}\circ \ldots \circ f_n:E_n\to E_k$, and $f_{n, n} \coloneq \Id_{E_n}$.

\smallskip

Consider now the Banach space $E$ which is the quotient of $\ell_\infty(\prod_n E_n)$ by the subspace $c_0(\prod_n E_n)=\{(x_n)\in \prod_n E_n\setsep \lim_n \|x_n\| = 0\}$. For $(x_n)\in \prod_n E_n$ we denote by $[(x_n)]$ the equivalence class $\{(x_n) + z\setsep z\in c_0(\prod_n E_n)\}$.

\smallskip

Given $n\in\Nat$, by $f_{n, \infty}:E_n\to E$ we denote the map defined by the formula
\[
f_{n, \infty}(x):=[(0,0,\ldots,0,f_{n, n}(x),f_{n,n+1}(x),\ldots,f_{n,k}(x),\ldots)],\qquad x\in E_n,
\]
that is, $f_{n, \infty}(x) = [g_n(x)]$, where $g_n(x)(k) = 0$ for $k < n$ and $g_n(x)(k) = f_{n,k}(x)$ for $k\geq n$. (Note that the assumption $\prod_n \|f_n\|<\infty$ guarantees that we have $g_n(x)\in \ell_\infty(\prod_n E_n)$.)

\smallskip

Finally, by \emph{directed limit of $(E_n,f_n)$}, denoted by $\lim\limits_{\to} (E_n, f_n)$, we understand the Banach space $\overline{\bigcup_n f_{n, \infty}(E_n)} \subseteq E$.
\end{definition}

\begin{remark}\label{rem:directLimits}
Let us give here some remarks concerning Definition~\ref{def:directedLimit}.
\begin{enumerate}[label=(\arabic*)]
    \item\label{it:normLimit} It is standard and easy to prove that for $[(x_n)]\in E$ we have $\|[(x_n)]\| = \limsup_n \|x_n\|$.
    \item\label{it:compatibilityLimitMaps} For $n\leq m$, $x\in E_n$ and $y\in E_m$ we have that $f_{n, \infty}(x)= f_{m, \infty}(y)$ if and only if $\lim_{k\geq m} \|f_{n,k}(x) - f_{m,k}(y)\| = 0$. In particular, we have $f_{m, \infty}\circ f_{n,m} = f_{n, \infty}$ for $n\leqslant m$.
    \item If all the maps $f_n$ are contractions, then $\lim\limits_{\to} (E_n, f_n)$ is directed limit in the category of Banach spaces, where morphisms are linear contractions. This explains our terminology.
    \item If we have $E_n\subseteq E_{n+1}$, $n\in\Nat$ and mappings $f_n$ are inclusions, then it is easy to see that $\lim\limits_{\to} (E_n, f_n)$ is linearly isometric to the completion of $\bigcup_n E_n$.
    \item\label{it:finDimDirectLimits} If all the $E_n$'s have dimension bounded by some $k \in \Nat$, then the same holds for the spaces $f_{n, \infty}(E_n)$, $n \in \Nat$; it follows that $\dim(\lim\limits_{\to} (E_n, f_n)) \leqslant k$ and $f_{n, \infty}(E_n) = \lim\limits_{\to} (E_n, f_n)$ for every large enough $n$.
\end{enumerate}
\end{remark}

\begin{definition}\label{def:BMGame}
    Given a separable Banach space $X$, we define the game $\BMG(X)$ as follows:
    \smallskip

\begin{tabular}{ccccccc}
\textbf{I} & $E_1,\varepsilon_1$ & & $E_3,f_2,\varepsilon_3$ & & $\hdots$ & \\
\textbf{II} & & $E_2,f_1,\varepsilon_2$ & & $E_4,f_3,\varepsilon_4$ & & $\hdots$ 
\end{tabular}

\smallskip

\noindent where we have $E_n\in\closedAge{X}$, $\varepsilon_{n+1}<\varepsilon_n$ and $f_n\in\Emb_{\varepsilon_n}(E_n,E_{n+1})$ for every $n\in\Nat$. Player II wins if $\sum_{n=1}^\infty \varepsilon_n<\infty$ and $\lim\limits_{\to} (E_n, f_n)$ is isometric to the Banach space $X$, Player I wins otherwise.\\
(Note that $\sum_{n=1}^\infty \varepsilon_n<\infty$ implies $\prod \|f_n\| \leq \prod e^{\varepsilon_n} < \infty$, so $\lim\limits_{\to} (E_n, f_n)$ is defined.)
\end{definition}

\begin{remark}\label{rem:normsConnectingMaps}
    Keep notation from Definition \ref{def:BMGame} and assume that $\sum_{n=1}^\infty \varepsilon_n < \infty$. Let $E_\infty \coloneq \lim\limits_{\to} (E_n, f_n)$ and define maps $f_{n, k}$ for $n \geqslant 1$ and $n \leqslant k \leqslant \infty$ from the $f_n$'s as in Definition \ref{def:directedLimit}. Then for every positive integers $n \leqslant k$, we have $$f_{n, k} \in \Emb_{\sum_{j=n}^{k-1} \varepsilon_j}(E_n, E_k).$$ By Remark \ref{rem:directLimits} \ref{it:normLimit}, this remains true even when $k = \infty$ (taking $k-1 = \infty$ in that case).
\end{remark}

\begin{proposition}\label{prop:BMGAndFinDim}
    Let $E$ be a finite-dimensional Banach space. Then Player II has a winning strategy in $\BMG(E)$.
\end{proposition}

\begin{proof}
    The strategy of Player II consists in playing, for every $n \in \Nat$, $E_{2n} = E$. He can then always choose a map $f_{2n-1} \in \Emb_{\varepsilon_{2n-1}}(E_{2n-1}, E)$, since $E_{2n-1} \in \closedAge{E}$. Finally, he can take $\varepsilon_{2n} \coloneq \varepsilon_{2n-1}/2$, ensuring that in the end, $\sum_{n=1}^\infty \varepsilon_n  < \infty$. 

    \smallskip

    To see that this strategy is winning, let $E_\infty \coloneq \lim\limits_{\to} (E_n, f_n)$ and define maps  $f_{n, \infty} \colon E_n \to E_\infty$, $n \in \Nat$ from the $f_n$'s as in Definition \ref{def:directedLimit}. Let $r_n \coloneq \sum_{j = n}^\infty \varepsilon_j$ for every $n \in \Nat$. Since for large enough $n \in \Nat$, $f_{n, \infty} \in \Emb_{r_n}(E_n, E_\infty)$ and this map is onto (see Remark \ref{rem:normsConnectingMaps} and Remark \ref{rem:directLimits} \ref{it:finDimDirectLimits}), and $\lim_{n \to \infty} r_n = 0$, the sequence $(E_n)_{n \in \Nat}$ converges to $E_\infty$ in the space $\BM$. Since $E_{2n} = E$ for every $n \in \Nat$, we deduce that $E_\infty$ is isometric to $E$.
\end{proof}

\begin{proposition}\label{prop:comeagerAndBMGame}
Let $X$ be a separable Banach space and let $\mathcal{I} \in \{\PP, \PP_\infty, \BB\}$ if $\dim(X) = \infty$, or $\mathcal{I} = \PP$ if $\dim(X) < \infty$. If $\isomtrclass[\mathcal{I}]{X}$ is comeager in $\closureOfIsomtrclass[\mathcal{I}]{X}\cap \mathcal{I}$, then Player II has a winning strategy in $\BMG(X)$.
\end{proposition}

\begin{proof}
    The case when $\dim(X) < \infty$ is treated by Proposition \ref{prop:BMGAndFinDim}, so we now suppose that $\dim(X) = \infty$. By Lemma \ref{lem:reflectingCategory}, we can assume that $\mathcal{I} = \BB$. Let $(v_n)_{n \in \Nat}$ be an enumeration of $V$. Let $D_n$, $n \in \Nat$, be open dense subsets of the topological space $\closureOfIsomtrclass[\BB]{X}\cap \BB$ such that $\bigcap_{n \in \Nat} D_n \subseteq \isomtrclass[\BB]{X}$. Player I starts with playing $E_1$ and $\varepsilon_1$. In order not to consider separately the first step of the game, we let $E_0 = \{0\} \subseteq c_{00}$, $\varepsilon_0 = 2\varepsilon_1$, and $f_0$ be the unique element of $\Emb_{\varepsilon_0}(E_0, E_1)$.
    We also let $F_0:=\{0\}\subseteq V$ and $\mu_0 \in \isomtrclass[\BB]{X}$ be arbitrary. The strategy of Player II will consist in ensuring that for all $n \in \Nat$, the following conditions are kept satisfied:
    \begin{enumerate}[label=(A\alph*)]
    \item\label{it:subsOfV} $E_{2n}$ has the form $((F_{2n})_\Rea, \mu_{2n}\restriction_{(F_{2n})_\Rea})$, where $F_{2n}$ is a finite-dimensional vector subspace of $V$ and $\mu_{2n} \in \isomtrclass[\BB]{X}$;
    \item\label{it:exhaust} $v_n \in F_{2n}$;
    \item\label{it:inclusion} $F_{2n-2} \subseteq F_{2n}$, and $f_{2n-1} \circ f_{2n-2}$ is the inclusion map $E_{2n-2} \to E_{2n}$;
    \item\label{it:decrEps} $\varepsilon_{2n} \leqslant \varepsilon_{2n-1}/2$.
    \item\label{it:inDenseOpen} $V_\BB[\mu_{2n}\restriction_{F_{2n}}, F_{2n}, 4\varepsilon_{2n}] \cap \closureOfIsomtrclass[\BB]{X} \subseteq D_n$.
    \end{enumerate}
    (Note that $F_0$ and $\mu_0$ have been chosen so that condition \ref{it:subsOfV} is satisfied even for $n=0$.)

    \smallskip

    We first inductively show that Player II can always ensure that the five above conditions are kept satisfied. To simplify notation, auxiliary objects introduced in an inductive step will be denoted without indices, which will be no issue as they won't be reused elsewhere. Fix $n \in \Nat$ and suppose Player I just played $E_{2n-1}$, $f_{2n-2}$ and $\varepsilon_{2n-1}$. By Lemma \ref{lem:simpleReflection} applied to $f_{2n-2} \colon (F_{2n-2})_{\Rea} \to E_{2n-1}$, there exists a finite-dimensional vector space $G \subseteq V$ containing $F_{2n-2}$, $\nu \in \Norm(G)$, and an onto isometry $\psi \colon (G_\Rea, \nu) \to E_{2n-1}$ extending $f_{2n-2}$. By the rules of the game, $E_{2n-1}$, and hence $(G_\Rea, \nu)$, belong to $\closedAge{X}$, so by Lemma \ref{lem:traceOnIsomClass}, we have $V_\BB[\nu, G, \varepsilon_{2n-1}] \cap \closureOfIsomtrclass[\BB]{X} \neq \varnothing$. Since moreover, by Proposition \ref{prop:equivalentNormOpen}, $V_\BB[\nu, G, \varepsilon_{2n-1}] \cap \closureOfIsomtrclass[\BB]{X}$ is an open subset of $\closureOfIsomtrclass[\BB]{X}\cap \BB$, the density of $D_n$ ensures that $D_n \cap V_\BB[\nu, G, \varepsilon_{2n-1}]$ is a nonempty open subset of $\closureOfIsomtrclass[\BB]{X}\cap \BB$. Hence, one can pick $\mu_{2n} \in D_n \cap V_\BB[\nu, G, \varepsilon_{2n-1}] \cap \isomtrclass[\BB]{X}$. Since $D_n$ is a neighborhood of $\mu_{2n}$ in $\closureOfIsomtrclass[\BB]{X}\cap \BB$, by Corollary \ref{cor:alternativeBasis}, one can find a finite-dimensional vector subspace $F_{2n} \subseteq V$ and $\varepsilon_{2n} > 0$ such that condition \ref{it:inDenseOpen} is satisfied. Shrinking $\varepsilon_{2n}$ and extending $F_{2n}$ if necessary, we can ensure that conditions \ref{it:exhaust} and \ref{it:decrEps} are satisfied, and that $G \subseteq F_{2n}$. Let 
    $E_{2n} \coloneq ((F_{2n})_\Rea, \mu_{2n}\restriction_{(F_{2n})_\Rea})$, so that condition \ref{it:subsOfV} is satisfied. Denote by $\iota \colon G_\Rea \to E_{2n}$  the inclusion map, and let $f_{2n-1} \coloneq \iota \circ \psi^{-1} \colon E_{2n-1} \to E_{2n}$. Since $\psi$ extends $f_{2n-2}$, we deduce that condition \ref{it:inclusion} is satisfied. Since $E_{2n}$ is a Banach subspace of $X_{\mu_{2n}}$, which is itself isometric to $X$, we have $E_{2n} \in \Age(X)$. Moreover, since $\mu_{2n} \in  V_\BB[\nu, G, \varepsilon_{2n-1}]$, we have $\iota \in \Emb_{< \varepsilon_{2n-1}}((G_\Rea, \nu), E_{2n})$, thus $f_{2n-1} \in \Emb_{< \varepsilon_{2n-1}}(E_{2n-1}, E_{2n})$. Thus, Player II can legally play $E_{2n}$, $f_{2n-1}$,  and $\varepsilon_{2n}$; this finishes the inductive argument.

    \smallskip

    We now prove that the strategy we described is winning for Player II. Remember that for every $n \in \Nat$, we have $\varepsilon_{2n} \leqslant \varepsilon_{2n-1}/2$ by condition \ref{it:decrEps}, and $\varepsilon_{2n+1} < \varepsilon_{2n}$ by the rules of the game. It follows that for every $n \in \Nat$,
    $$\sum_{k = 2n}^\infty \varepsilon_k < 2\sum_{j = n}^\infty \varepsilon_{2j} \leqslant 2\sum_{j=n}^\infty \varepsilon_{2n}\cdot 2^{j-n} = 4\varepsilon_{2n}.$$
    In particular, $\sum_{n = 1}^\infty \varepsilon_n < \infty$. Let $E_\infty \coloneq \lim\limits_{\to} (E_n, f_n)$ and define maps  $f_{n, k}$ for $n \geqslant 1$ and $n \leqslant k \leqslant \infty$ from the $f_n$'s as in Definition \ref{def:directedLimit}. The above computation combined with Remark \ref{rem:normsConnectingMaps} shows that for every $n \in \Nat$, $f_{2n, \infty} \in \Emb_{4\varepsilon_{2n}}(E_{2n}, E_\infty)$. By condition \ref{it:inclusion} and Remark \ref{rem:directLimits} \ref{it:compatibilityLimitMaps}, for every $n \in \Nat$, the map $f_{2n+2, \infty}$ extends $f_{2n, \infty}$. Moreover by condition \ref{it:exhaust}, $\bigcup_{n \in \Nat} E_{2n} = c_{00}$. Hence, there is a unique linear and injective map $T \colon c_{00} \to E_\infty$ extending all the $f_{2n}$'s, $n \in \Nat$. Define $\mu \in \BB$ by $\mu(x) = \|T(x)\|$ for every $x \in c_{00}$, so that $T \colon (c_{00}, \mu) \to E_\infty$ is an isometric embedding. Since the range of $T$ contains $f_{2n, \infty}(E_{2n})$ for every $n \in \Nat$, by the definition of a direct limit, it is dense in $E_\infty$. We deduce that $T$ extends to an onto isometry $X_\mu \to E_\infty$, that we will still denote by $T$. For every $n \in \Nat$, we have $T^{-1} \circ f_{2n, \infty} \in \Emb_{4\varepsilon_{2n}}(E_{2n}, X_\mu)$, and $T^{-1} \circ f_{2n, \infty}$ is the inclusion map $E_{2n} \to X_\mu$, so $\mu \in V_\BB[\mu_{2n}, F_{2n}, 4\varepsilon_{2n}]$. In particular, for all positive integers $m \leqslant n$, we have $e^{-4\varepsilon_n}\mu(v_m)<\mu_{2n}(v_m)<e^{4\varepsilon_n}\mu(v_m)$, thus the sequence $(\mu_{2n})_{n \in \Nat}$ converges to $\mu$ in the space $\BB$. Since, for every $n \in \Nat$, $\mu_{2n} \in \isomtrclass[\BB]{X}$, we deduce that $\mu \in \closureOfIsomtrclass[\BB]{X}$. Hence, by condition \ref{it:inDenseOpen},
    $$\mu \in \bigcap_{n \in \Nat} V_\BB[\mu_{2n}, F_{2n}, 4\varepsilon_{2n}] \cap \closureOfIsomtrclass[\BB]{X} \subseteq \bigcap_{n \in \Nat} D_n \subseteq \isomtrclass[\BB]{X},$$
    and since $X_\mu$ is isometric to $E_\infty$, we deduce that $E_\infty$ is isometric to $X$.
\end{proof}

\begin{proposition}\label{prop:winningI}Let $X$ be a separable Banach space which does not satisfy the condition \ref{it:gFraisseForPlayerI}. Then Player I has a winning strategy in $\BMG(X)$.
\end{proposition}

\begin{proof}
Fix a finite-dimensional subspace $E\subseteq X$ and $\varepsilon>0$ witnessing that the negation of the condition \ref{it:gFraisseForPlayerI} holds. Let $\gamma \coloneq \varepsilon / 9$ and $\theta \coloneq \varepsilon e^{-4\gamma}/3$. We will inductively describe a winning strategy for Player I in $\BMG(X)$ on a play of this game whose moves will be denoted by $E_i$, $f_i$, and $\varepsilon_i$, as in Definition \ref{def:BMGame}. The direct limit $\lim\limits_{\to} (E_n, f_n)$ will be denoted by $E_\infty$ and maps  $f_{n, k}$ for $n \geqslant 1$ and $n \leqslant k \leqslant \infty$ will be defined from the $f_n$'s as in Definition \ref{def:directedLimit}.

\smallskip

At the first step of the game, Player I plays $E_1 \coloneq E$ and $\varepsilon_1 \coloneq \gamma$. Before describing the next steps, we describe conditions Player I will preserve and auxiliary objects he will construct while playing. At each step $n \geqslant 1$, Player I will ensure that $\varepsilon_{2n+1} \leqslant \varepsilon_{2n} / 2$; knowing by the rules of the game that moreover, one has $\varepsilon_{2n+2} < \varepsilon_{2n+1}$ for each $n$, it will follow that for every $n \geqslant 0$, one has:
$$\sum_{k = 2n+1}^\infty \varepsilon_n < 2\sum_{j = n}^\infty \varepsilon_{2j+1} \leqslant 2\sum_{j=n}^\infty \varepsilon_{2n+1}\cdot 2^{j-n} = 4\varepsilon_{2n+1},$$
which, combined with Remark \ref{rem:normsConnectingMaps}, will ensure that the following conditions are satisfied:
\begin{enumerate}[label=(B\alph*), series=winningI]
\item\label{it:etaBound} for every $n \geqslant 0$, one has $f_{2n+1, \infty} \in \Emb_{4\varepsilon_{2n+1}}(E_{2n+1}, E_\infty)$;
\item\label{it:gammaBound} for every $n \geqslant 1$ and $n \leqslant k \leqslant \infty$, $f_{n, k} \in \Emb_{4\gamma}(E_n, E_k)$.
\end{enumerate}
While playing, Player I will also choose, at the same time as $E_{2n+1}$, $f_{2n}$, and $\varepsilon_{2n+1}$, a finite $\theta$-net $\mathcal{N}_{2n}$ in the compact set $\Emb_{5\gamma}(E, E_{2n})$, for each $n \geqslant 1$. Letting, for every $n \geqslant 0$, $p_n \coloneq \sum_{m=1}^n \left|\mathcal{N}_{2m}\right|$, elements of $\mathcal{N}_{2n}$ will be denoted by $\chi_{p_{n-1}+1}, \ldots, \chi_{p_n}$. Note that, for each $n \geqslant 1$, $\Emb_{5\gamma}(E, E_{2n})$ is nonempty (as, by \ref{it:gammaBound}, it contains $f_{1, 2n}$), so $\mathcal{N}_{2n}$ is also nonempty, so $p_n > p_{n-1}$; it follows that for every $n \geqslant 1$, $p_n \geqslant n$ and hence, $\chi_n \in \mathcal{N}_{2m}$ for some $m \leqslant n$.

\smallskip

Fix $n \geqslant 1$; we now explain how Player I plays at step $2n+1$. He first chooses the net $\mathcal{N}_{2n}$. From the previous paragraph, we know that $\chi_n$ has already been constructed and belongs to $\mathcal{N}_{2m}$ for some $m \leqslant n$. Let $\phi_n \coloneq f_{2m, 2n} \circ \chi_n$. Since $f_{2m, 2n} \in \Emb_{4\gamma}(E_{2m}, E_{2n})$ and $\chi_n \in \Emb_{5\gamma}(E, E_{2m})$, we deduce that $\phi_n \in \Emb_\varepsilon(E, E_{2n})$. Hence, by the choice of $E$ and $\varepsilon$, we deduce the existence of $E_{2n+1} \in \closedAge{X}$, $f_{2n} \in \Emb_{\varepsilon_{2n}}(E_{2n}, E_{2n+1})$, and $\eta_n > 0$ such that for every $\iota \in \Emb_{\eta_n}(E_{2n+1}, X)$, one has $\|\Id_E - \iota \circ f_{2n} \circ \phi_n\| \geqslant \varepsilon$. Player I then plays $E_{2n+1}$, $f_{2n}$, and $\varepsilon_{2n+1} \coloneq \min(\varepsilon_{2n}/2, \eta_n/4)$.

\smallskip

We now prove that the strategy we just described is winning for Player I. Towards a contradiction, assume that $E_\infty$ is isometric to $X$ and fix an isometry $T \colon E_\infty \to X$. One has $X = \overline{\bigcup_{k \geqslant 1}T\circ f_{2k, \infty}(E_{2k})}$, so by Lemma \ref{lem:perturbationArgument}, one can find $m \geqslant 1$ and $\alpha \in \Emb_\gamma(E, X)$ such that $\|\alpha - \Id_E\| \leqslant \varepsilon /3$ and $\alpha(E) \in T\circ f_{2m, \infty}(E_{2m})$. Since $f_{2m, \infty} \in \Emb_{4\gamma}(E_{2m}, E_\infty)$, we deduce that $f_{2m, \infty}^{-1} \circ T^{-1} \circ \alpha \in \Emb_{5\gamma}(E, E_{2m})$. So we can find an element $\chi_n \in \mathcal{N}_{2m}$ such that $\|\chi_n - f_{2m, \infty}^{-1} \circ T^{-1} \circ \alpha\| \leqslant \theta$. We have:
\begin{align*}
    \|T \circ f_{2m, \infty} \circ \chi_n - \Id_E\| & \leqslant \|T \circ f_{2m, \infty} \circ \chi_n - \alpha\| + \|\alpha - \Id_E\|\\
    & \leqslant \|T\|\cdot\|f_{2m, \infty}\|\cdot\|\chi_n - f_{2m, \infty}^{-1} \circ T^{-1} \circ \alpha\| + \frac{\varepsilon}{3}\\
    & \leqslant 1 \cdot e^{4\gamma} \cdot \theta + \frac{\varepsilon}{3}\\
    & = \frac{2\varepsilon}{3} < \varepsilon.
\end{align*}
Also observe that, since $\chi_n \in \mathcal{N}_{2m}$, we have $n \geqslant m$ and $\phi_n = f_{2m, 2n} \circ \chi_n$. Hence, letting $\iota \coloneq T \circ f_{2n+1, \infty}$, we have $T \circ f_{2m, \infty} \circ \chi_n = \iota \circ f_{2n} \circ \phi_n$, and the last inequality becomes: $$\|\iota \circ f_{2n} \circ \phi_n - \Id_E\| < \varepsilon.$$
Knowing, by \ref{it:etaBound} and the choice of $\varepsilon_{2n+1}$, that $\iota \in \Emb_{4\varepsilon_{2n+1}}(E_{2n+1}, X) \subseteq \Emb_{\eta_n}(E_{2n+1}, X)$, this contradicts the choice of $E_{2n+1}$, $f_{2n}$, and $\eta_n$.
\end{proof}

We end this section by mentionning that, as an immediate consequence of Propositions \ref{prop:BMGAndFinDim} and \ref{prop:winningI}, all finite-dimensional Banach spaces are guarded Fra\"iss\'e.

\begin{proposition}
    Every finite-dimensional Banach space satisfies \ref{it:gFraisseForPlayerI}.
\end{proposition}

\section{Guarded ultrahomogeneity}\label{sec:Ultrahom}

This section is devoted to the proofs of two important theorems which, while being seemingly unrelated, are actually both consequences of the same technical result (Proposition \ref{prop:backAndForthConstruction}). The first one, Theorem \ref{thm:EquivalenceGFraGUH} below, is the equivalence between the various conditions defining guarded Fra\"iss\'eness: conditions \ref{it:gFraisseAction} and \ref{it:gFraisseForPlayerI} already introduced in Definition \ref{def:GuardedFraisseIntroPart1}, along with three new conditions.

\begin{theorem}\label{thm:EquivalenceGFraGUH}
Let $X$ be a separable Banach space. Consider the following conditions.
\begin{enumerate}[label=(gF-\arabic*), resume*=gFraisse]
    \item\label{it:gFraisseFourth} For every finite-dimensional subspace $E\subseteq X$ and every $\varepsilon>0$, there exist $F\in\Age(X)$, $\phi\in\Emb_{\varepsilon}(E,F)$ and $\delta>0$ such that for every $\psi\in\Emb_{\delta}(F,X)$, there exists $T\in \Iso(X)$ satisfying $\|T\restriction_E - \psi\circ\phi\|<\varepsilon$.
    
    \smallskip\item\label{it:gFraisseFifth} For every finite-dimensional subspace $E\subseteq X$ and every $\varepsilon>0$, there exist $F\in\Age(X)$, $\phi\in\Emb_{\varepsilon}(E,F)$ and $\delta > 0$ such that for every $G\in\Age(X)$ and $\psi\in\Emb_{\delta}(F,G)$, there exists $\iota\in \Emb(G,X)$ satisfying $\|\Id_E - \iota\circ\psi\circ\phi\|<\varepsilon$.

    \smallskip\item\label{it:gFraisseThird} For every finite-dimensional subspace $E\subseteq X$ and every $\varepsilon>0$, there exist a finite dimensional subspace $F$ with $E \subseteq F \subseteq X$ and $\delta>0$ such that for every $G \in \closedAge{X}$, $\psi \in \Emb_{\delta}(F, G)$ and $\eta > 0$, there exists $\iota \in \Emb_{\eta}(G, X)$ with  $\|\Id_E-\iota\circ \psi\restriction_ E\|<\varepsilon$.
\end{enumerate}
Then \ref{it:gFraisseAction} holds $\Leftrightarrow$ \ref{it:gFraisseForPlayerI} holds $\Leftrightarrow$ \ref{it:gFraisseFourth} holds $\Leftrightarrow$  \ref{it:gFraisseFifth} holds $\Leftrightarrow$ \ref{it:gFraisseThird} holds.
\end{theorem}

\noindent As already mentioned in Definition \ref{def:GuardedFraisseIntroPart1}, from now on, a \textit{guarded Fra\"iss\'e Banach space} will be defined as a separable Banach space satisfying the equivalent conditions \ref{it:gFraisseAction}, \ref{it:gFraisseForPlayerI}, \ref{it:gFraisseFourth}, \ref{it:gFraisseFifth} and \ref{it:gFraisseThird}. Before going further, let us comment on those conditions. Properties \ref{it:gFraisseAction} and \ref{it:gFraisseFourth} are metric versions of what Kruckman \cite{KruckmanPhD} calls \textit{weak ultrahomogeneity} in the discrete setting. Those two properties are easily seen to be equivalent, and their main difference is that, in \ref{it:gFraisseFourth}, the ``guard'' (that is, the map $\phi$) is not required to be an isometry, while in \ref{it:gFraisseAction} it is one (it is actually the inclusion map $E \to F$). On the other hand, properties \ref{it:gFraisseForPlayerI} and \ref{it:gFraisseThird} are metric versions of what Krawczyk--Kubis \cite{krKu} call \textit{weak injectivity} in the discrete setting. They are also easily seen to be equivalent, and their main difference is that, in \ref{it:gFraisseForPlayerI}, the ``guard'' is not required to be an isometry, while in \ref{it:gFraisseThird} it is an inclusion map. In the discrete setting, weak ultrahomogeneity trivially implies weak injectivity, but the converse requires a more involved back-and-forth argument; it will be similar here, the back-and-forth argument being the core of the proof of the technical Proposition \ref{prop:backAndForthConstruction}. Finally, property \ref{it:gFraisseFifth} is a slight modification of \ref{it:gFraisseForPlayerI} where the map $\iota$ is required to be isometric; we do not know any simple proof of the equivalence between those two property, and were only able to get it as a consequence of the back-and-forth argument.

\smallskip

The second important theorem to be proved in this section, Theorem \ref{thm:guardedUniqueByAge} is the ``injectivity'' part of Theorem \ref{thm:GuardedFraisseCorrespondence}, our Fra\"iss\'e correspondence for guarded Fra\"iss\'e spaces. It will also be instrumental in the proof of Theorem \ref{thm:BMgameForIsometryClasses}.

\begin{theorem}\label{thm:guardedUniqueByAge}
Let $X$ and $Y$ be two separable Banach spaces satisfying condition \ref{it:gFraisseForPlayerI}, such that $\closedAge{X} = \closedAge{Y}$. Then $X$ and $Y$ are linearly isometric.
\end{theorem}

We will first state and prove our technical Proposition \ref{prop:backAndForthConstruction}, and then obtain Theorems \ref{thm:EquivalenceGFraGUH} and \ref{thm:guardedUniqueByAge} as consequences. But before doing so, we first need to introduce definitions of some notions of pairs and triples witnessing conditions \ref{it:gFraisseAction} to \ref{it:gFraisseThird}. This will enable us to shorten some arguments and to obtain more precise results in what follows.

\begin{definition}\label{def:gFpairs}
Let $X$ be a separable Banach space, and $\varepsilon, \delta > 0$.
\smallskip
\begin{enumerate}
\item A pair $(E, F)$, where $E \subseteq F \subseteq X$ are finite-dimensional subspaces, is called:
\begin{enumerate}
    \item an \textit{($X, \varepsilon, \delta)$-\ref{it:gFraisseAction} pair} if there exists $\varepsilon' \in (0, \varepsilon)$ and $\delta' > \delta$ such that $\Iso(X)$ acts $\varepsilon'$-transitively on $\{\iota\restriction_E\setsep \iota\in \Emb_{\delta'}(F,X)\}$;
    \item an \textit{($X, \varepsilon, \delta)$-\ref{it:gFraisseThird} pair} if there exists $\varepsilon' \in (0, \varepsilon)$ and $\delta' > \delta$ such that for every $G \in \closedAge{X}$, $\psi \in \Emb_{\delta'}(F, G)$ and $\eta > 0$, there exists $\iota \in \Emb_{\eta}(G, X)$ with  $\|\Id_E-\iota\circ \psi\restriction_ E\|<\varepsilon'$.
\end{enumerate}

\smallskip\item A triple $(E, F, \phi)$, where $E \subseteq X$ is a finite-dimensional subspace, $F \in \closedAge{X}$, and $\phi\colon E \to F$, is called:
\begin{enumerate}
    \item an \textit{($X, \varepsilon, \delta)$-\ref{it:gFraisseForPlayerI} triple} if there exists $\varepsilon' \in (0, \varepsilon)$ and $\delta' > \delta$ such that $\phi \in \Emb_{\varepsilon'}(E, F)$ and for every $G\in\closedAge{X}$, $\psi\in\Emb_{\delta'}(F,G)$  and $\eta>0$, there exists $\iota\in\Emb_{\eta}(G,X)$ such that $\|\Id_{E} - \iota\circ \psi\circ \phi\|<\varepsilon'$;
    \item an \textit{($X, \varepsilon, \delta)$-\ref{it:gFraisseFourth} triple} if there exists $\varepsilon' \in (0, \varepsilon)$ and $\delta' > \delta$ such that $\phi \in \Emb_{\varepsilon'}(E, F)$ and for every $\psi\in\Emb_{\delta'}(F,X)$, there exists $T\in \Iso(X)$ such that $\|T\restriction_E - \psi\circ\phi\|<\varepsilon'$;
    \item an \textit{($X, \varepsilon, \delta)$-\ref{it:gFraisseFifth} triple} if there exists $\varepsilon' \in (0, \varepsilon)$ and $\delta' > \delta$ such that $\phi \in \Emb_{\varepsilon'}(E, F)$ and for every $G\in\Age(X)$ and $\psi\in\Emb_{\delta'}(F,G)$, there exists $\iota\in \Emb(G,X)$ such that $\|\Id_E - \iota\circ\psi\circ\phi\|<\varepsilon'$.
\end{enumerate}
\end{enumerate}
\end{definition}

\begin{remark}\label{rem:gFFromPairs}
As names indicate, each of the notions of pairs and triples introduced in Definition \ref{def:gFpairs} corresponds to one of the equivalent conditions defining guarded Fra\"iss\'eness. However, additional quantifications on $\varepsilon'$ and $\delta'$ have been added in order to make those notions stable under small perturbations of $\varepsilon$ and $\delta$; this will ease presentations of some results. Despite this addition, it is easy to see that, for a separable Banach space $X$, the following holds:
\begin{itemize}
    \item $X$ satisfies \ref{it:gFraisseAction} (resp. \ref{it:gFraisseThird}) iff for 
    every finite-dimensional subspace $E\subseteq X$ and every $\varepsilon>0$, there exist a finite dimensional subspace $F$ with $E \subseteq F \subseteq X$ and $\delta>0$ such that $(E, F)$ is an $(X,\varepsilon,\delta)$-\ref{it:gFraisseAction} pair (resp. an $(X, \varepsilon, \delta)$-\ref{it:gFraisseThird} pair);
    \item $X$ satisfies \ref{it:gFraisseForPlayerI} iff 
    for every finite-dimensional subspace $E \subseteq X$ and every $\varepsilon>0$, there exist $F\in \closedAge{X}$, $\phi\colon E \to F$ and $\delta>0$ such that $(E, F, \phi)$ is an $(X, \varepsilon, \delta)$-\ref{it:gFraisseForPlayerI} triple.
    \item $X$ satisfies \ref{it:gFraisseFourth} (resp. \ref{it:gFraisseFifth}) iff 
    for every finite-dimensional subspace $E \subseteq X$ and every $\varepsilon>0$, there exist $F\in \Age{X}$, $\phi\colon E \to F$, and $\delta>0$ such that $(E, F, \phi)$ is an $(X, \varepsilon, \delta)$-\ref{it:gFraisseFourth} triple (resp. an $(X, \varepsilon, \delta)$-\ref{it:gFraisseFifth} triple).
\end{itemize}
\end{remark}

\begin{proposition}\label{prop:backAndForthConstruction}
Let $X$ and $Y$ be two separable Banach spaces satisfying \ref{it:gFraisseForPlayerI} with\linebreak $\closedAge{X} = \closedAge{Y}$. Let $\varepsilon, \delta > 0$ and $(E, F, \phi)$ be an $(X, \varepsilon, \delta)$-\ref{it:gFraisseForPlayerI} triple. Then for every $\psi\in\Emb_{\delta}(F,Y)$, there exists $T\in\Iso(X,Y)$ such that $\|T\restriction_E - \psi \circ \phi\| < \varepsilon$.
\end{proposition}

\begin{proof}
For simplicity of notation, we will let, for every $n \geqslant 0$, $X_{2n} = X$ and $X_{2n+1} = Y$. Consider a sequence $(x_n)_{n \geqslant 0}$ with $x_n \in X_n$ for every $n \geqslant 0$, such that $(x_{2n})_{n \geqslant 0}$ is dense in $X$ and $(x_{2n+1})_{n \geqslant 0}$ is dense in $Y$. We also assume that $x_0 \in E$. We build by induction a sequence $(E_n)_{n \geqslant 0}$ where for every $n \geqslant 0$, $E_n$ is a finite-dimensional subspace of $X_n$, a sequence $(F_n)_{n \geqslant 0}$ of elements of $\closedAge{X}$, two sequences $(\varepsilon_n)_{n \geqslant 0}$ and $(\delta_n)_{n \geqslant 0}$ of positive real numbers, and for each $n \geqslant 0$, maps $\phi_n \colon E_n \to F_n$ and $\psi_n \colon F_n \to X_{n+1}$, satisfying the following properties, for every $n \geqslant 0$:
\begin{enumerate}[label=(c\arabic*), series=ultrahomConstruction]
    \item\label{it:subset} $E_n \subseteq E_{n+2}$
    \item\label{it:point} $x_n \in E_n$;
    \item\label{it:imageSubset} $\psi_n(F_n) \subseteq E_{n+1}$;
    \item\label{it:gFraTriple} $(E_n, F_n, \phi_n)$ is an $(X_n, \varepsilon_n, \delta_n)$-\ref{it:gFraisseForPlayerI} triple;
    \item\label{it:isoPsi} $\psi_n\in\Emb_{<\delta_n}(F_n,X_{n+1})$;
    \item\label{it:isoPhiPsi} $\phi_{n+1} \circ \psi_n\in\Emb_{\delta_n}(F_n,F_{n+1})$;
    \item\label{it:epsNext} $\varepsilon_{n+1} \leqslant \frac{\varepsilon_n}{2}$;
    \item\label{it:deltaNextt} $\delta_{n+1} \leqslant \frac{\delta_n}{2}$
    \item\label{it:commute} $\|Id_{X_n}\restriction_{E_n} - \psi_{n+1} \circ \phi_{n+1} \circ \psi_n \circ \phi_n\| < \varepsilon_n$;
\end{enumerate}
along with the following properties at the initial steps of the construction:
\begin{enumerate}[label=(c\arabic*), resume*=ultrahomConstruction]
    \item\label{it:epsOneFirstCondition} 
    $(\varepsilon_0\exp(\varepsilon_{1}) + 3\varepsilon_{1}\exp(\delta_0+\varepsilon_0)) < \frac{\varepsilon + \varepsilon_0}{2}$;
    \item\label{it:deltaOne} $\delta_1 \leqslant \varepsilon_1$.
\end{enumerate}
\begin{center}\includegraphics[scale=0.8]{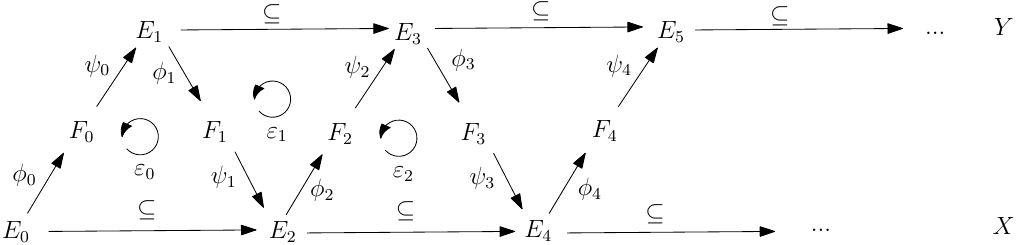}\end{center}
We start with taking $E_0 = E$, $F_0 = F$, $\phi_0 = \phi$, $\psi_0 = \psi$, and we choose $\varepsilon_0 \in (0, \varepsilon)$ and $\delta_0 > \delta$ such that $(E, F, \phi)$ is still an $(X, \varepsilon_0, \delta_0)$-\ref{it:gFraisseForPlayerI} triple.
We now fix $n \geqslant 0$. We assume that the $E_i$'s, the $F_i$'s, the $\phi_i$'s, the $\psi_i$'s, the $\varepsilon_i$'s and the $\delta_i$'s have been built for $i \leqslant n$, and we build $E_{n+1}$, $F_{n+1}$, $\phi_{n+1}$, $\psi_{n+1}$, $\varepsilon_{n+1}$, and $\delta_{n+1}$.

\smallskip

We first choose $E_{n+1}$ satisfying conditions \ref{it:point}, \ref{it:imageSubset}, and if $n \geqslant 1$, also condition \ref{it:subset}. We choose $\varepsilon_{n+1} > 0$ satisfying condition \ref{it:epsNext} and, if $n=0$, condition \ref{it:epsOneFirstCondition}, such that the following inequality is satisfied:
\begin{equation}\label{eqn:DefEpsilonN}e^{\varepsilon_{n+1}} \leqslant \frac{e^{\delta_n}}{\max(\|\psi_n\|, \|\psi_n^{-1}\|)}.\tag{A}\end{equation} Note that by condition \ref{it:isoPsi}, the right hand side of the above inequality is greater than $1$, so such a choice of $\varepsilon_{n+1}$ is always possible.
Now, since $X_{n+1}$ satisfies \ref{it:gFraisseForPlayerI}, we can find $F_{n+1} \in \closedAge{X}$, $\phi_{n+1} \colon E_{n+1} \to F_{n+1}$, and $\delta_{n+1} > 0$ such that $(E_{n+1}, F_{n+1}, \phi_{n+1})$ is an $(X_{n+1}, \varepsilon_{n+1}, \delta_{n+1})$-\ref{it:gFraisseForPlayerI} triple, that is, condition \ref{it:gFraTriple} is satisfied. Decreasing $\delta_{n+1}$ if necessary, we can ensure that condition \ref{it:deltaNextt} and, if $n=0$, condition \ref{it:deltaOne}, are satisfied. We know, by the definition of a \ref{it:gFraisseForPlayerI} triple, that $\phi_{n+1} \in \Emb_{\varepsilon_{n+1}}(E_{n+1}, F_{n+1})$, so inequality \eqref{eqn:DefEpsilonN} ensures that condition \ref{it:isoPhiPsi} is satisfied. Now, using that $(E_n, F_n, \phi_n)$ is an $(X_n, \varepsilon_n, \delta_n)$-\ref{it:gFraisseForPlayerI} triple, that $F_{n+1} \in \closedAge{X_n}$, and that $\phi_{n+1}\circ \psi_n\in\Emb_{\delta_n}(F_n,F_{n+1})$, we can find $\psi_{n+1} \in \Emb_{<\delta_{n+1}}(F_{n+1}, X_n)$ such that $\|Id_{X_n}\restriction_{E_n} - \psi_{n+1} \circ \phi_{n+1} \circ \psi_n \circ \phi_n\| < \varepsilon_n$. That is, this choice of $\psi_{n+1}$ simultaneously satisfies conditions \ref{it:isoPsi} and \ref{it:commute}. This finishes the inductive construction.

\smallskip

Fix $x \in \bigcup_{n \geqslant 0}E_{2n}$, and let $n_x \geqslant 0$ be minimal such that $x \in E_{2n_x}$. For $m \geqslant n_x$, we have, by conditions \ref{it:isoPsi}, \ref{it:commute}, and the fact (contained in condition \ref{it:gFraTriple}) that $\phi_n \in \Emb_{\varepsilon_n}(E_n, F_n)$ for all $n \geqslant 0$:
\begin{multline}\label{eqn:IneqB}
    \|\psi_{2m+2}\circ \phi_{2m+2}(x) -  \psi_{2m} \circ \phi_{2m}(x)\|\\
    \begin{aligned}
    \leqslant \;\; & \|\psi_{2m+2}\circ \phi_{2m+2}\circ(\Id_{E_{2m}} - \psi_{2m+1}\circ\phi_{2m+1}\circ\psi_{2m}\circ\phi_{2m})(x)\|\\  &  + \|(\psi_{2m+2}\circ\phi_{2m+2}\circ\psi_{2m+1}\circ\phi_{2m+1}  - \Id_Y\restriction_{E_{2m+1}})\circ(\psi_{2m}\circ\phi_{2m})(x)\| \\
    \leqslant \;\; & (\varepsilon_{2m}\cdot\|\psi_{2m+2}\|\cdot \|\phi_{2m+2}\| + \varepsilon_{2m+1}\cdot\|\psi_{2m}\|\cdot \|\phi_{2m}\|)\cdot\|x\| \\
    \leqslant \;\; & (\varepsilon_{2m}\exp(\delta_{2m+2}+\varepsilon_{2m+2}) + \varepsilon_{2m+1}\exp(\delta_{2m}+\varepsilon_{2m}))\cdot\|x\|.
    \end{aligned}
    \tag{B}
\end{multline}
In particular, if $m \geqslant 1$, using conditions \ref{it:epsNext} and \ref{it:deltaNextt}, we get:
\begin{align*}
    \|\psi_{2m+2}\circ \phi_{2m+2}(x) -  \psi_{2m} \circ \phi_{2m}(x)\| & \leqslant 2\varepsilon_{2m}\exp( \delta_0 + \varepsilon_0)\|x\| \\
    & \leqslant 2^{-2m+2}\varepsilon_1\exp(\delta_0+\varepsilon_0)\|x\|.
\end{align*}
Thus, for $n \geqslant m \geqslant \max(1, n_x)$, we have:
\begin{align}\label{eqn:IneqC}
\begin{split}
    \|\psi_{2n}\circ \phi_{2n}(x) & -  \psi_{2m} \circ \phi_{2m}(x)\| \leqslant \sum_{i=m}^{n-1} \|\psi_{2i+2}\circ \phi_{2i+2}(x) -  \psi_{2i} \circ \phi_{2i}(x)\| \\
    & \leqslant \varepsilon_1\exp(\delta_0+\varepsilon_0)\cdot \|x\| \cdot \sum_{i=m}^{n-1} 2^{-2i+2} \leqslant \varepsilon_1\exp(\delta_0+\varepsilon_0)\cdot \|x\| \cdot \sum_{j=2m-2}^\infty 2^{-j}\\
    & \leqslant \frac{\varepsilon_1\exp(\delta_0+\varepsilon_0)}{2^{2m-3}}\cdot \|x\|.
    \end{split}
    \tag{C}
\end{align}

It follows that the sequence $(\psi_{2n} \circ \phi_{2n}(x))_{n \geqslant n_x}$ is a Cauchy sequence in $Y$, so it converges to some $T(x) \in Y$. In this way we defined a linear mapping $T \colon \bigcup_{n \geqslant 0} E_{2n} \to Y$. As an immediate consequence of inequality \eqref{eqn:IneqC}, we obtain that for every $m \geqslant 1$ and $x \in E_{2m}$, one has:
\begin{equation}\label{eqn:IneqD}
    \|T(x) - \psi_{2m} \circ \phi_{2m}(x)\| \leqslant \frac{\varepsilon_1\exp(\delta_0+\varepsilon_0)}{2^{2m-3}}\cdot \|x\|.\tag{D}
\end{equation}
By conditions \ref{it:gFraTriple} and \ref{it:isoPsi}, we have for every $n \geqslant n_x$, $\exp(-\varepsilon_{2n} - \delta_{2n})\|x\| < \|\psi_{2n} \circ \phi_{2n}(x)\| < \exp(\varepsilon_{2n} + \delta_{2n})\|x\|$, so, using the fact that the sequences $(\varepsilon_n)_{n \geqslant 0}$ and $(\delta_n)_{n \geqslant 0}$ converge to $0$, we deduce that $T$ is isometric. Condition \ref{it:point} ensures that $\bigcup_{n \geqslant 0} E_{2n}$ is dense in $X$; it follows that $T$ can be extended to an isometric embedding $X \to Y$, that will still be denoted by $T$. Given $N \geqslant 0$ and $y \in E_{2N+1}$, we obtain, using successively condition \ref{it:commute}, inequality \eqref{eqn:IneqD}, and conditions \ref{it:gFraTriple} and \ref{it:isoPsi}, that for every $n \geqslant N$, 
\begin{eqnarray*}
    \|y - T \circ \psi_{2n+1} \circ \phi_{2n+1}(y)\| & \leqslant &  \|y - \psi_{2n+2}\circ\phi_{2n+2}\circ \psi_{2n+1} \circ \phi_{2n+1}(y)\| \\ & & + \|(\psi_{2n+2}\circ \phi_{2n+2} - T)\circ \psi_{2n+1} \circ \phi_{2n+1}(y)\|\\
    & \leqslant & \left(\varepsilon_{2n+1} + \frac{\varepsilon_1\exp(\delta_0+\varepsilon_0 + \delta_{2n+1} + \varepsilon_{2n+1})}{2^{2n-1}}\right)\|y\|.
\end{eqnarray*}
In particular, the sequence $(T \circ \psi_{2n+1} \circ \phi_{2n+1}(y))_{n \geqslant N}$ converges to $y$, hence showing that $y$ is in the closure of $T(X)$. Since, by condition \ref{it:point}, $\bigcup_{N \geqslant 0} E_{2N+1}$ is dense in $Y$, we deduce that $T$ has dense range, so $T \in \Iso(X, Y)$.

\smallskip

We now want to estimate $\|T\restriction_E - \psi \circ \phi\|$. We fix $x \in E \setminus \{0\}$. Combining inequalities \eqref{eqn:IneqB} for $m=0$ and \eqref{eqn:IneqD} for $m=1$, we obtain:
\begin{align*}
\|T(x) - \psi\circ\phi(x)\| 
&\leqslant (\varepsilon_0\exp(\delta_{2}+\varepsilon_{2}) + 3\varepsilon_{1}\exp(\delta_0+\varepsilon_0))\cdot\|x\|,
\end{align*}
so applying successively conditions \ref{it:epsNext}, \ref{it:deltaNextt}, \ref{it:deltaOne} and \ref{it:epsOneFirstCondition}, we obtain:
\begin{align*}
    \|T(x) - \psi\circ\phi(x)\| &\leqslant \left(\varepsilon_0\exp\left(\frac{\varepsilon_1}{2} + \frac{\delta_1}{2}\right) + 3\varepsilon_{1}\exp(\delta_0+\varepsilon_0)\right)\\
    & \leqslant (\varepsilon_0\exp(\varepsilon_{1}) + 3\varepsilon_{1}\exp(\delta_0+\varepsilon_0))\cdot\|x\|< \frac{\varepsilon + \varepsilon_0}{2}\|x\|.
\end{align*}
Thus $\|T\restriction_E - \psi \circ \phi\| \leqslant \frac{\varepsilon + \varepsilon_0}{2} < \varepsilon$.
\end{proof}

As an almost immediate application we obtain Theorem~\ref{thm:guardedUniqueByAge}.

\begin{proof}[Proof of Theorem~\ref{thm:guardedUniqueByAge}]
The existence of $T \in \Iso(X, Y)$ is given by Proposition \ref{prop:backAndForthConstruction} applied to the spaces $X$ and $Y$, the triple $(\{0\}, \{0\}, \phi)$ where $\phi$ is the unique map $\{0\} \to \{0\}$ (it is easily see to be an $(X, 1, 1)$-\ref{it:gFraisseForPlayerI} triple), and the unique map $\psi \colon \{0\} \to Y$.
\end{proof}

We now turn to the proof of Theorem ~\ref{thm:EquivalenceGFraGUH}. We start with a few lemmas about manipulation of pairs and triples introduced in Definition \ref{def:gFpairs}; they will also be useful in the next sections.

\begin{lemma}\label{lem:ImplicationTriples}
Let $X$ be a separable Banach space, $\varepsilon > 0$, $\delta > 0$, $E \subseteq X$ be a finite-dimensional subspace, $F \in \closedAge{X}$, and $\phi \colon E \to F$ be an operator. Consider the following statements:
\begin{enumerate}
    \item $(E, F, \phi)$ is an $(X, \varepsilon, \delta)$-\ref{it:gFraisseForPlayerI} triple;
    \item $(E, F, \phi)$ is an $(X, \varepsilon, \delta)$-\ref{it:gFraisseFourth} triple;
    \item $(E, F, \phi)$ is an $(X, \varepsilon, \delta)$-\ref{it:gFraisseFifth} triple.
\end{enumerate}
Then (2) $\Rightarrow$ (3) $\Rightarrow$ (1). Moreover, if $X$ satisfies \ref{it:gFraisseForPlayerI}, then (1) $\Rightarrow$ (2).
\end{lemma}

\begin{proof}
    Assume (2) holds, witnessed by $\varepsilon' \in (0, \varepsilon)$ and $\delta'>\delta$; we show that (3) holds, witnessed by the same $\varepsilon'$ and $\delta'$. Pick $G \in \Age(X)$ and $\psi\in\Emb_{\delta'}(F,G)$. Let $\rho \in \Emb(G, X)$. Then $\rho \circ \psi\in\Emb_{\delta'}(F,X)$; hence there exists $T\in\Iso(X)$ satisfying $\|T\restriction_E - \rho \circ \psi\circ \phi\| < \varepsilon'$. Left-composing by $T^{-1}$, we get that $\|\Id_E - T^{-1}\circ \rho \circ \psi \circ \phi\| < \varepsilon'$. So $\iota := T^{-1}\circ \rho\in\Emb(G,X)$ satisfies $\|\Id_E - \iota \circ \psi \circ \phi\| < \varepsilon'$.

    \smallskip Assume (3) holds, witnessed by  $\varepsilon' \in (0, \varepsilon)$ and $\delta' > \delta$; also let $\delta'' \in (\delta, \delta')$. To prove (1), we pick $G \in \closedAge{X}$, $\psi\in\Emb_{\delta''}(F,G)$, and $\eta>0$, and we want to find $\iota \in \Emb_{\eta}(G, X)$ such that $\|\Id_E - \iota \circ \psi \circ \phi\| < \varepsilon'$. We can assume that $\eta \leqslant \delta' - \delta''$. By the definition of $\closedAge{X}$, there are $G' \in \Age(X)$ and an onto map $\rho\in\Emb_{\eta}(G,G')$. We have $\rho \circ \psi \in \Emb_{\delta'}(F, G')$, so we can find $\iota' \in \Emb(G', X)$ such that $\|\Id_E - \iota'\circ \rho \circ \psi\circ\phi\| < \varepsilon'$. Letting $\iota \coloneq \iota' \circ \rho$, we have $\iota \in \Emb_{\eta}(G, X)$ and $\|\Id_E - \iota \circ \psi \circ \phi\| < \varepsilon'$.

    \smallskip Assume that $X$ satisfies \ref{it:gFraisseForPlayerI} and (1) holds. Fix $\varepsilon' \in (0, \varepsilon)$ and $\delta' > \delta$ such that $(E, F, \phi)$ is still an $(X, \varepsilon', \delta')$-\ref{it:gFraisseForPlayerI} triple. Then Proposition \ref{prop:backAndForthConstruction} applied to $Y = X$ immediately shows that $(E, F, \phi)$ is a $(X, \varepsilon, \delta)$-\ref{it:gFraisseFourth} triple, witnessed by $\varepsilon'$ and $\delta'$. Thus, (2) holds.
\end{proof}

\begin{lemma}\label{lem:equivalencePairsTriples}
    Let $X$ be a separable Banach space, $\varepsilon, \delta > 0$, and $E \subseteq F \subseteq X$ be finite-dimensional subspaces.
    \begin{enumerate}
        \item $(E, F)$ is an $(X, \varepsilon, \delta)$-\ref{it:gFraisseThird} pair iff $(E, F, \Id_E)$ is an $(X, \varepsilon, \delta)$-\ref{it:gFraisseForPlayerI} triple.
        \item If $(E, F, \Id_E)$ is an $(X, \varepsilon, \delta)$-\ref{it:gFraisseFourth} triple, then $(E, F)$ is an $(X, 2\varepsilon, \delta)$-\ref{it:gFraisseAction} pair.
        \item If $(E, F)$ is an $(X, \varepsilon, \delta)$-\ref{it:gFraisseAction} pair, then $(E, F, \Id_E)$ is an $(X, \varepsilon, \delta)$-\ref{it:gFraisseFourth} triple.
    \end{enumerate}
\end{lemma}

\begin{proof}
    \begin{enumerate}
        \item This is immediate from the definitions.
        
        \smallskip\item Let $\varepsilon' \in (0, \varepsilon)$ and $\delta'>\delta$ be witnessing that $(E, F, \Id_E)$ is an $(X, \varepsilon, \delta)$-\ref{it:gFraisseFourth} triple. Let $\iota, \iota' \in \Emb_{\delta'}(F, X)$. Then one can find $T, T' \in \Iso(X)$ such that $\|T\restriction_E - \iota\restriction_E\| = \|T\restriction_E - \iota\circ \Id_E\| < \varepsilon'$, and similarly $\|T'\restriction_E - \iota'\restriction_E\| < \varepsilon'$. Thus,
        \begin{align*}
           \|\iota'\restriction_E - T'\circ T^{-1}\circ\iota\restriction_E\| &\leqslant \|\iota'\restriction_E - T'\restriction_E\| + \|T'\restriction_E - T'\circ T^{-1}\circ\iota\restriction_E\| \\
           & \leqslant \|\iota'\restriction_E - T'\restriction_E\| + \|T'\|\cdot \|T^{-1}\|\cdot\|T\restriction_E -\iota\restriction_E\| \leqslant 2\varepsilon',
        \end{align*}
        so $(E, F)$ is an $(X, 2\varepsilon, \delta)$-\ref{it:gFraisseAction} pair, witnessed by $2\varepsilon'$ and $\delta'$.
        
        \smallskip\item Let $\varepsilon' \in (0, \varepsilon)$ and $\delta'>\delta$ be witnessing that $(E, F)$ is an $(X, \varepsilon, \delta)$-\ref{it:gFraisseAction} pair. Let $\psi \in \Emb_{\delta'}(F, X)$. Denote by $\chi$ the inclusion map $F \to X$. By $\varepsilon'$-transitivity of the action $\Iso(X) \curvearrowright \{\iota\restriction_E \setsep \iota \in \Emb_{\delta'}(F, X)\}$, one can find $T \in \Iso(X)$ such that $\|\psi\restriction_E - T\circ \chi\restriction_E\| < \varepsilon'$, in other words $\|\psi\circ \Id_E - T\restriction_E\| < \varepsilon'$. Hence,  $(E, F, \Id_E)$ is an $(X, \varepsilon, \delta)$-\ref{it:gFraisseFourth} triple, witnessed by $\varepsilon'$ and $\delta'$.  \qedhere
    \end{enumerate}
\end{proof}

\begin{corollary}\label{cor:implicationsPairs}
    Let $X$ be a separable Banach space, $\varepsilon, \delta > 0$, and $E \subseteq F \subseteq X$ be finite-dimensional subspaces. 
    \begin{enumerate}
        \item If $(E, F)$ is an $(X, \varepsilon, \delta)$-\ref{it:gFraisseAction} pair, then $(E, F)$ is an $(X, \varepsilon, \delta)$-\ref{it:gFraisseThird} pair.
        \item If $X$ satisfies \ref{it:gFraisseForPlayerI} and $(E, F)$ is an $(X, \varepsilon, \delta)$-\ref{it:gFraisseThird} pair, then $(E, F)$ is an $(X, 2\varepsilon, \delta)$-\ref{it:gFraisseAction} pair.
    \end{enumerate}
\end{corollary}

\begin{proof}
    Those are immediate consequences of Lemmas \ref{lem:ImplicationTriples} and \ref{lem:equivalencePairsTriples}.
\end{proof}

\begin{lemma}\label{lem:stabilityGFTriple}
    Let $X$ be a Banach space, $\varepsilon, \delta > 0$, and $(E, F, \phi)$ be an $(X, \varepsilon, \delta)$-\ref{it:gFraisseForPlayerI} triple. Then for every $\beta \in [0, 1]$ and $\gamma \in [0, \delta)$ with $\gamma \leqslant \beta$, $E' \subseteq X$ finite-dimensional subspace, $F' \in \closedAge{X}$, $\sigma \in \Emb_\gamma(F, F')$, and $\rho \in \Emb_\beta(E', E)$ with $\|\Id_{E'} - \rho\| \leqslant \beta$, it holds that $(E', F', \sigma \circ \phi \circ \rho)$ is an $(X, \varepsilon + 3\beta, \delta - \gamma)$-\ref{it:gFraisseForPlayerI} triple.

    \smallskip

    Moreover, $X$, $\varepsilon$, $\delta$ and the triple $(E, F, \phi)$ being fixed, we can find $\beta = \gamma > 0$ so that for any choice of $E'$, $F'$, $\sigma$ and $\rho$ as above, $(E', F', \sigma \circ \phi \circ \rho)$ is still an $(X, \varepsilon, \delta)$-\ref{it:gFraisseForPlayerI} triple.
\end{lemma}

\begin{proof}
    Let $\varepsilon' \in (0, \varepsilon)$ and $\delta'>\delta$ be witnessing that $(E, F, \phi)$ is an $(X, \varepsilon, \delta)$-\ref{it:gFraisseForPlayerI} triple. We show that $(E', F', \sigma \circ \phi \circ \rho)$ is an $(X, \varepsilon + 3\beta, \delta - \gamma)$-\ref{it:gFraisseForPlayerI} triple witnessed by $\varepsilon' + 3\beta$ and $\delta' - \gamma$; to get the ``moreover'' part, it will hence be enough to choose $\beta = \gamma > 0$ so that $\varepsilon' + 3\beta < \varepsilon$ and $\delta' - \gamma > \delta$.

    \smallskip
    
    First, observe that $\sigma\circ\phi\circ\rho \in \Emb_{\varepsilon' + 2\beta}(E', F')$. Now fix $G \in \closedAge{X}$, $\psi \in \Emb_{\delta'-\gamma}(F', G)$, and $\eta > 0$. Then $\psi \circ \sigma \in \Emb_{\delta'}(F, G)$, so there exists $\iota \in \Emb_\eta(G, X)$ such that $\|\Id_E - \iota \circ \psi \circ \sigma \circ \phi\| < \varepsilon'$. Thus,
    $$\|\Id_{E'} - \iota \circ \psi \circ \sigma \circ \phi\circ \rho\| \leqslant \|\Id_{E'} - \rho\| + \|\Id_E - \iota \circ \psi \circ \sigma \circ \phi\| \cdot \|\rho\| < \beta + \varepsilon'e^\beta \leqslant \varepsilon' + 3\beta,$$
the last inequality being consequence of the fact that $e^\beta \leqslant 1+2\beta$; this can be obtained by convexity of the exponential, remembering that $\beta \in [0, 1]$.
\end{proof}

We state below a version of Lemma \ref{lem:stabilityGFTriple} for \ref{it:gFraisseThird} pairs. Although it will not be used in the proof of Theorem \ref{thm:EquivalenceGFraGUH}, it will be useful in Section \ref{sec:GuardedFraisseCorrespondence}.

\begin{corollary}\label{cor:gFraPairIsomorph}
Let $X$ be a separable Banach space, $\varepsilon, \delta \in (0, 1)$, and $(E, F)$ be an $(X, \varepsilon, \delta)$-\ref{it:gFraisseThird} pair. Then for every $\gamma \in [0, \delta)$ and $E'\subseteq F'\subseteq X$ finite-dimensional subspaces, if there exists $\sigma \in \Emb_\gamma(F, F')$ with $E' \subseteq \sigma(E)$ and $\|\Id_{E'} - \sigma^{-1}\restriction_{E'}\| \leqslant \gamma$, then $(E', F')$ is an $(X, \varepsilon + 3 \gamma, \delta - \gamma)$-\ref{it:gFraisseThird} pair.
\end{corollary}

\begin{proof}
    Use Lemma \ref{lem:equivalencePairsTriples} and apply Lemma \ref{lem:stabilityGFTriple} with $\beta \coloneq \gamma$ and $\rho \coloneq \sigma^{-1}\restriction_{E'}$.
\end{proof}

\begin{proof}[Proof of Theorem~\ref{thm:EquivalenceGFraGUH}]
The implications \ref{it:gFraisseFourth} $\Rightarrow$ \ref{it:gFraisseFifth} $\Rightarrow$ \ref{it:gFraisseForPlayerI} are immediate consequences of Remark \ref{rem:gFFromPairs} and Lemma \ref{lem:ImplicationTriples}. To see that \ref{it:gFraisseForPlayerI} $\Rightarrow$ \ref{it:gFraisseFourth} holds, use Remark \ref{rem:gFFromPairs} and fix $E \subseteq X$ and $\varepsilon \in (0, 1)$. Then by \ref{it:gFraisseForPlayerI}, we can find $F \in \closedAge{X}$, $\phi \colon E \to F$, and $\delta \in (0, 1)$ such that $(E, F, \phi)$ is an $(X, \varepsilon, \delta)$-\ref{it:gFraisseForPlayerI} triple. Let $\gamma$ be as given in the ``moreover'' part of Lemma \ref{lem:stabilityGFTriple} for this triple, and find $F' \in \Age(X)$ and $\sigma \in \Emb_\gamma(F, F')$. Then $(E, F', \sigma \circ \phi)$ in an $(X, \varepsilon, \delta)$-\ref{it:gFraisseForPlayerI} triple, so by Lemma \ref{lem:ImplicationTriples}, it is an $(X, \varepsilon, \delta)$-\ref{it:gFraisseFourth} triple. Hence, $X$ satsifies \ref{it:gFraisseFourth}.

\smallskip

The implication \ref{it:gFraisseThird} $\Rightarrow$ \ref{it:gFraisseForPlayerI} is an immediate consequence of Remark \ref{rem:gFFromPairs} and Lemma \ref{lem:equivalencePairsTriples}. We now prove that the converse implication holds, again using Remark \ref{rem:gFFromPairs}. Fix a finite-dimensional subspace $E \subseteq X$ and $\varepsilon \in (0, 1)$. One can find $F \in \closedAge{X}$, $\phi \colon E \to F$ and $\delta \in(0, 1)$ such that $(E, F, \phi)$ is an $(X, \varepsilon/23, 2\delta)$-\ref{it:gFraisseForPlayerI} triple. Fix $\gamma > 0$ as given in the ``moreover'' part of Lemma \ref{lem:stabilityGFTriple} for this triple. Applying the definition of a \ref{it:gFraisseForPlayerI} triple to $(E, F, \phi)$, $G = F$, $\psi = \Id_F$, and $\eta = \gamma$, one can find $\sigma \in \Emb_{\gamma}(F, X)$ such that $\|\Id_E - \sigma \circ \phi\| < \varepsilon/23$. Let $F' \subseteq X$ be a finite-dimensional subspace containing both $E$ and $\sigma(F)$, and define $\phi' \colon E \to F'$ by $\phi' \coloneq \sigma \circ \phi$. Then $\|\Id_E - \phi'\| < \varepsilon/23$, and by the choice of $\gamma$, $(E, F', \phi')$ is an $(X, \varepsilon/23, 2\delta)$-\ref{it:gFraisseForPlayerI} triple. We now show that $(E, F')$ is an $(X, \varepsilon, \delta)$-\ref{it:gFraisseThird} pair, witnessed by $22\varepsilon/23$ and $2\delta$. For this, fix $G \in \closedAge{X}$, $\psi \in \Emb_{2\delta}(F', G)$, and $\eta \in (0, 1)$. Knowing that $(E, F', \phi')$ is an $(X, \varepsilon/23, 2\delta)$-\ref{it:gFraisseForPlayerI} triple, one can find $\iota \in \Emb_\eta(G, X)$ such that $\|\Id_E - \iota \circ \psi \circ \phi'\| < \varepsilon/23$. Thus,
\begin{align*}
    \|\Id_E - \iota \circ \psi\restriction_E\| \leqslant \|\Id_E - \iota \circ \psi \circ \phi'\| + \|\iota\|\cdot \|\psi\| \cdot \|\phi' - \Id_E\| < \frac{\varepsilon}{23} + e^{3}\frac{\varepsilon}{23} \leqslant \frac{22\varepsilon}{23}.
\end{align*}

\smallskip

Finally, using the (already proved) fact that \ref{it:gFraisseThird} implies \ref{it:gFraisseForPlayerI}, the equivalence between \ref{it:gFraisseAction} and \ref{it:gFraisseThird} is a direct consequence of Remark \ref{rem:gFFromPairs} and Corollary \ref{cor:implicationsPairs}.
\end{proof}

\section{Guarded Fra\"iss\'e classes}\label{sec:gFClass}

In this section, we develop the basic theory of guarded Fra\"iss\'e classes. Our main goals are to prove stability of classical properties of subclasses of $\BM$ under taking closures (see Propositions \ref{prop:HPstableUnderClosures} and \ref{prop:ageAmalgThenItsClosureAlso}), to draw a first link with guarded Fra\"iss\'eness of Banach spaces (see Proposition \ref{prop:ageOfGuardedHasAmalg}), and to prove
several preliminary results to be reused in the next sections. We start with introducing notions of pairs and triples associated to the (GAP). Given $\KK \subseteq \BM$, we will denote by $\Incl(\KK)$ the class of all pairs $(E, F)$ of finite-dimensional Banach spaces such that $E, F \in \KK$ and $E \subseteq F$.

\begin{definition}\label{def:amalgPair}
Fix $\mathcal{K} \subseteq \BM$, and $\varepsilon, \delta > 0$.
    \begin{itemize}
         \item A triple $(E, F, \phi)$, where $E, F \in \mathcal{K}$ and $\phi \colon E \to F$, is called a \textit{$(\mathcal{K}, \varepsilon, \delta)$-amalgamation triple} if there exist $\varepsilon' \in (0, \varepsilon)$ and $\delta' > \delta$ such that $\phi \in \Emb_{\varepsilon'}(E, F)$ and for every $G, H \in \mathcal{K}$, $\psi_G \in \Emb_{\delta'}(F, G)$, $\psi_H \in \Emb_{\delta'}(F, H)$, there exist $K \in \mathcal{K}$, $\iota_G \in \Emb(G, K)$ and $\iota_H \in \Emb(H, K)$ for which $\|\iota_G \circ \psi_G \circ \phi - \iota_H\circ\psi_H \circ \phi\| < \varepsilon'$.
        \item A pair $(E, F) \in \Incl(\KK)$ is called a \textit{$(\mathcal{K}, \varepsilon, \delta)$-amalgamation pair} if $(E, F, \Id_E)$ is a $(\mathcal{K}, \varepsilon, \delta)$-amalgamation triple.
    \end{itemize}
\end{definition}

\begin{remark}\label{rem:equivAmalgPair}
    \begin{enumerate}
        \item As in the definition of pairs and triples associated to guarded Fra\"iss\'eness, the additional quantifications on $\varepsilon'$ and $\delta'$ in Definition \ref{def:amalgPair} have been added to ensure stability of the notions of amalgamation pairs and triples under sufficiently small perturbations of $\varepsilon$ and $\delta$.
        \item\label{it:GAPandPairs} A class $\KK \subseteq \BM$ satisfies (GAP) iff for every $E \in \KK$ and $\varepsilon > 0$, there exists $F \in \KK$, $\phi \colon E \to F$, and $\delta > 0$ such that $(E, F, \phi)$ is a $(\mathcal{K}, \varepsilon, \delta)$-amalgamation triple.
        \item\label{it:JEPandPairs} A class $\KK \subseteq \BM$ containing $\{0\}$ satisfies (JEP) iff $(\{0\}, \{0\})$ is a $(\mathcal{K}, 1, 1)$-amalgamation pair.
    \end{enumerate}
\end{remark}

We first prove some results of stability under small perturbations for amalgamation pairs and triples.

\begin{lemma}\label{lem:stabilityAmalgTriples}
    Let $\KK \subseteq \BM$, $\varepsilon, \delta \in (0, 1)$, and $(E, F, \phi)$ be a $(\KK, \varepsilon, \delta)$-amalgamation triple. Then for every $\beta \geqslant 0$, $\gamma \in [0, \delta)$ with $\gamma \leqslant \beta$, $E', F' \in \KK$, $\rho \in \Emb_\beta(E', E)$, and $\sigma \in \Emb_\gamma(F, F')$, it holds that $(E', F', \sigma \circ \phi \circ \rho)$ is a $(\KK, \varepsilon + 2\beta, \delta - \gamma)$-amalgamation triple. 
    
    \smallskip
    
    Moreover, $\KK$, $\varepsilon$, $\delta$ and the triple $(E, F, \phi)$ being fixed, one can find $\beta = \gamma > 0$ so that for any choice of $E'$, $F'$, $\rho$ and $\sigma$ as above, $(E', F', \sigma \circ \phi \circ \rho)$ is still a $(\KK, \varepsilon, \delta)$-amalgamation triple.
\end{lemma}

\begin{proof}
    Let $\varepsilon' \in (0, \varepsilon)$ and $\delta' > \delta$ be witnessing that $(E, F, \phi)$ is a $(\KK, \varepsilon, \delta)$-amalgamation triple. We show that $(E', F', \sigma \circ \phi \circ \rho)$ is an $(\KK, \varepsilon + 2\beta, \delta - \gamma)$-amalgamation triple witnessed by $\varepsilon' + 2\beta$ and $\delta' - \gamma$; to get the ``moreover'' part, it will hence be enough to choose $\beta = \gamma > 0$ so that $\varepsilon' + 2\beta < \varepsilon$ and $\delta' - \gamma > \delta$.
   
\smallskip   
   
    First, observe that $\sigma \circ \phi \circ \rho \in \Emb_{\varepsilon'+\beta+\gamma}(E', F') \subseteq \Emb_{\varepsilon'+2\beta}(E', F')$. Now pick $G, H \in \mathcal{K}$, $\psi_G \in \Emb_{\delta'-\gamma}(F', G)$, $\psi_H \in \Emb_{\delta'-\gamma}(F', H)$. Then $\psi_G \circ \sigma \in \Emb_{\delta'}(F, G)$ and $\psi_H \circ \sigma \in \Emb_{\delta'}(F, H)$, so one can find $K \in \KK$, $\iota_G \in \Emb(G, K)$ and $\iota_H \in \Emb(G, K)$ such that $\|\iota_G \circ \psi_G \circ \sigma \circ \phi - \iota_H \circ \psi_H \circ \sigma \circ \phi\| < \varepsilon'$. Thus,
    $$\|\iota_G \circ \psi_G \circ \sigma \circ \phi\circ\rho - \iota_H \circ \psi_H \circ \sigma \circ \phi\circ \rho\| < \varepsilon'\|\rho\| \leqslant e^\beta\varepsilon',$$
    so the result follows from the fact that $\varepsilon' < 1$ and $e^\beta \leqslant 1 + 2\beta$ (this can be obtained by convexity of the exponential, remembering that $\beta \in (0, 1)$).
\end{proof}

\begin{corollary}\label{cor:amalgPairUnderIsomorphisms}
Let $\KK \subseteq \BM$, $\varepsilon, \delta \in (0, 1)$, and $(E, F)$ be a $(\KK, \varepsilon, \delta)$-amalgamation pair. Then there exists $\gamma > 0$ with the following property: for every $(E', F') \in \Incl(\KK)$, if there exists $\sigma \in \Emb_\gamma(F, F')$ with $E' \subseteq \sigma(E)$, then $(E', F')$ is still a $(\KK, \varepsilon, \delta)$-amalgamation pair.
\end{corollary}

\begin{proof}
Observe that $\Id_{E'} = \sigma \circ \Id_E \circ (\sigma^{-1})\restriction_{E'}$ and apply the ``moreover'' part of Lemma \ref{lem:stabilityAmalgTriples}, taking $\rho \coloneq  (\sigma^{-1})\restriction_{E'}$.
\end{proof}

\begin{lemma}\label{lem:approxAmalgPair}
    Let $\KK \subseteq \BM$, $\varepsilon, \delta \in (0, 1)$, and $E, E', F$ be such that $(E, F) \in \Incl(\KK)$ and $(E', F) \in \Incl(\KK)$. Suppose that there exists a linear map $\rho \colon E \to E'$ with $\|\Id_E - \rho\| \leqslant \varepsilon$, and $(E', F)$ is a $(\KK, \varepsilon, \delta)$-amalgamation pair. Then $(E, F)$ is a $(\KK, 8\varepsilon, \delta)$-amalgamation pair.
\end{lemma}

\begin{proof}
    Let $\varepsilon' \in (0, \varepsilon)$ and $\delta' \in (\delta, 1)$ be witnessing that $(E', F)$ is a $(\KK, \varepsilon, \delta)$-amalgamation pair. Pick $G, H \in \mathcal{K}$, $\psi_G \in \Emb_{\delta'}(F, G)$, $\psi_H \in \Emb_{\delta'}(F, H)$. Then there exist $K \in \KK$, $\iota_G \in \Emb(G, K)$ and $\iota_H \in \Emb(H, K)$ such that $\|\iota_G\circ\psi_G\restriction_{E'} - \iota_H\circ\psi_H\restriction_{E'}\| < \varepsilon'$. Thus for $x \in S_E$, one has:
    \begin{multline*}
        \|\iota_G\circ\psi_G(x) - \iota_H\circ\psi_H(x)\| \\
        \begin{aligned}
            \leqslant & \|\iota_G\circ\psi_G(\rho(x)) - \iota_H\circ\psi_H(\rho(x))\| + \|\iota_G \circ \psi_G - \iota_H \circ\psi_H\| \cdot \|x - \rho(x)\| \\
        < & \varepsilon'\|\rho\| + 2e^{\delta'}\varepsilon < \varepsilon'(1+\varepsilon) + 6\varepsilon< \varepsilon' + 7\varepsilon.
        \end{aligned}
    \end{multline*}
    Thus $(E, F)$ is a $(\KK, 8\varepsilon, \delta)$-amalgamation pair, witnessed by $\delta'$ and $\varepsilon' + 7 \varepsilon$.
\end{proof}

The following lemma gives an alternative characterization of amalgamation triples (and hence, pairs) in the case of closed hereditary subclasses of $\BM$.

\begin{lemma}\label{lem:amalgableSelfImprovement}
 Let $\mathcal{K} \subseteq \BM$ be closed and hereditary, $\varepsilon, \delta > 0$, $E, F \in \KK$, and $\phi \colon E \to F$. The following are equivalent:
 \begin{enumerate}
    \item $(E, F, \phi)$ is a $(\KK, \varepsilon, \delta)$-amalgamation triple;
    \item\label{it:weakAmalgTriple} there exist $\varepsilon' \in (0, \varepsilon)$ and $\delta' > \delta$ such that $\phi \in \Emb_{\varepsilon'}(E, F)$ and for every $G, H \in \mathcal{K}$, $\psi_G \in \Emb_{\delta'}(F, G)$, $\psi_H \in \Emb_{\delta'}(F, H)$, and $\eta > 0$, there exist $K \in \mathcal{K}$, $\iota_G \in \Emb_\eta(G, K)$ and $\iota_H \in \Emb_\eta(H, K)$ for which $\|\iota_G \circ \psi_G \circ \phi - \iota_H\circ\psi_H \circ \phi\| < \varepsilon'$.
 \end{enumerate}
\end{lemma}

\begin{proof}
    The implication (1) $\Rightarrow$ (2) is obvious; we prove (2) $\Rightarrow$ (1). Fix $\varepsilon', \delta' > 0$ as given by (2). Let $G, H \in \mathcal{K}$, $\psi_G \in \Emb_{\delta'}(F, G)$, $\psi_H \in \Emb_{\delta'}(F, H)$. Then by (2), for every $n \in \Nat$,  we can find $K_n \in \mathcal{K}$, $\iota_{G, n} \in \Emb_{1/n}(G, K_n)$ and $\iota_{H, n} \in \Emb_{1/n}(H, K_n)$ for which $\|\iota_{G, n} \circ \psi_G \circ \phi - \iota_{H, n}\circ\psi_H \circ \phi\| < \varepsilon'$.
 By replacing, for each $n$, $K_n$ with $\iota_{G, n}(G) + \iota_{H, n}(H)$ if necessary, we can assume that for every $n\in\Nat$, we have $\dim(K_n) \leqslant \dim(G) + \dim (H)$ and therefore, up to passing to a subsequence, we may assume that $\dim K_n$, $n\in\Nat$ is constant, and $(K_n)_{n \in \Nat}$ converges to some $K \in \KK$. For every $n\in\Nat$, pick a surjective isomorphism $\rho_n\in \Emb_{d_{\BM}(K_n, K)}(K_n,K)$. Since $\mathcal{L}(G,K)$ is finite-dimensional, passing to a subsequence, we may assume that $\rho_n\circ \iota_{G, n} \to \iota_G$; and since $d_{\BM}(K_n, K) \to 0$, we necessarily have $\iota_G \in \Emb(G, K)$. Similarly, we may assume that $\rho_n\circ \iota_{H, n}\to \iota_H\in \Emb(H,K)$. To conclude, observe that $\|\iota_G\circ \psi_H \circ\psi - \iota_H\circ \psi_H \circ \phi\|\leq \limsup_n \|\rho_n\|\cdot\|\iota_{G, n}\circ \psi_H \circ\psi - \iota_{H, n}\circ \psi_H \circ \phi\|\leq \varepsilon'$, so (1) holds, witnessed by $\delta'$ and, for instance, $(\varepsilon' + \varepsilon)/2$.
\end{proof}

We now pass to the study of stability under closures.

\begin{proposition}\label{prop:HPstableUnderClosures}
	The property (HP) is stable under taking closures.
\end{proposition}

\begin{proof}
Let $\KK \subseteq \BM$ satisfy (HP). Let $F \in \overline{\KK}$ and $E \subseteq F$ be a subspace. To show that $E \in \overline{\KK}$, we pick $\varepsilon > 0$ and we show the existence of $E' \in \KK$ with $d_{\BM}(E, E') \leqslant \varepsilon$. We can first pick $F' \in \KK$ with $d_{\BM}(F, F') \leqslant \varepsilon$, and a surjective isomorphism $\phi \colon F \to F'$ witnessing this last inequality. We let $E' \coloneq \phi(E)$; then $\phi\restriction_E$ witnesses that $d_{\BM}(E, E') \leqslant \varepsilon$.
\end{proof}

\begin{lemma}\label{lem:pairsUnderClosures}
    Let $\mathcal{K} \subseteq \BM$ be hereditary, and $\varepsilon, \delta > 0$. Every $(\KK, \varepsilon, \delta)$-amalgamation triple is also a $(\overline{\KK}, \varepsilon, \delta)$-amalgamation triple. The same result holds for amalgamation pairs.
\end{lemma}

\begin{proof}
    The result for pairs is an immediate consequence of this for triples. Let $(E, F, \phi)$ be a $(\KK, \varepsilon, \delta)$-amalgamation triple, witnessed by $\varepsilon' \in (0,\varepsilon)$ and $\delta' > \delta$. Let $\delta'' \in (\delta, \delta')$; we show that $(E, F, \phi)$ is an $(\overline{\KK}, \varepsilon, \delta)$-amalgamation triple by showing that condition (\ref{it:weakAmalgTriple}) in Lemma \ref{lem:amalgableSelfImprovement} holds for $\varepsilon'$ and $\delta''$ (this lemma applies, since $\overline{\KK}$ is hereditary by Proposition \ref{prop:HPstableUnderClosures}). Let $G, H \in \overline{\mathcal{K}}$, $\psi_G \in \Emb_{\delta''}(F, G)$, $\psi_H \in \Emb_{\delta''}(F, H)$, and $\eta > 0$. Let $G', H' \in \KK$, and $\rho_G \in \Emb_{\min(\eta, \delta' - \delta'')}(G, G')$, $\rho_H\in \Emb_{\min(\eta, \delta' - \delta'')}(H, H')$. Then $\rho_G \circ \psi_G \in \Emb_{\delta'}(F, G')$ and $\rho_H \circ \psi_H \in \Emb_{\delta'}(F, H')$, so one can find $K \in \KK$, $\iota_G \in \Emb(G', K)$ and $\iota_H \in \Emb(H', K)$ such that $\|\iota_G\circ\rho_G\circ\psi_G\circ\phi - \iota_H\circ\rho_H\circ\psi_H\circ\phi\| < \varepsilon'$. This concludes the proof, since $\iota_G\circ\rho_G \in \Emb_\eta(G, K)$ and $\iota_H\circ\rho_H \in \Emb_\eta(H, K)$.
\end{proof}

\begin{proposition}\label{prop:ageAmalgThenItsClosureAlso}
Let $\KK \subseteq \BM$ be hereditary. If $\KK$ satisfies (JEP) (resp. (GAP)), then $\overline{\KK}$ does, too.
\end{proposition}

\begin{proof}
In the case of (JEP), this is an immediate consequence of Remark \ref{rem:equivAmalgPair} (\ref{it:JEPandPairs}) and Lemma \ref{lem:pairsUnderClosures}. We now show the result for (GAP) using Remark \ref{rem:equivAmalgPair} (\ref{it:GAPandPairs}). Let $E \in \overline{\KK}$ and $\varepsilon \in (0, 1)$. Pick $E' \in \KK$ and an onto map $\rho \in \Emb_{\varepsilon/3}(E, E')$. By the (GAP) for $\KK$, one can find $F \in \KK$, $\phi \colon E' \to F$ and $\delta \in (0, 1)$ such that $(E', F, \phi)$ is a $(\KK, \varepsilon/3, \delta)$-amalgamation triple. By Lemma \ref{lem:pairsUnderClosures}, this is also a $(\overline{\KK}, \varepsilon/3, \delta)$-amalgamation triple. So by Lemma \ref{lem:stabilityAmalgTriples} (applied with $\beta = \varepsilon/3$, $\gamma = 0$ and $\sigma = \Id_F$), $(E, F, \phi \circ \rho)$ is a $(\overline{\KK}, \varepsilon, \delta)$-amalgamation triple.
\end{proof}

The last result of this section establishes a first link between guarded Fra\"isséness of Banach spaces and of subclasses of $\BM$.

\begin{proposition}\label{prop:ageOfGuardedHasAmalg}
Let $X$ be a separable Banach space and $\varepsilon, \delta > 0$. Then every $(X, \varepsilon, \delta)$-\ref{it:gFraisseAction} pair is an $(\Age(X), \varepsilon, \delta)$-amalgamation pair. In particular, if $X$ is guarded Fra\"iss\'e, then $\Age(X)$ satisfies (GAP), and $\closedAge{X}$ is a guarded Fra\"iss\'e class. 
\end{proposition}

\begin{proof}
Let $(E, F)$ be an $(X, \varepsilon, \delta)$-\ref{it:gFraisseAction} pair, witnessed by $\varepsilon' \in (0, \varepsilon)$ and $\delta' > \delta$. Pick $G, H \in \Age(X)$, $\psi_G \in \Emb_{\delta'}(F, G)$, and $\psi_H \in \Emb_{\delta'}(F, H)$. Let $\zeta_G \in \Emb(G, X)$ and $\zeta_H \in \Emb(H, X)$.  Both $\zeta_G \circ \psi_G\restriction_E$ and $\zeta_H \circ \psi_H\restriction_E$ belong to $\{\iota\restriction_E \setsep \iota \in \Emb_{\delta'}(F, X)\}$, on which $\Iso(X)$ acts $\varepsilon'$-transitively, so one can find $T \in \Iso(X)$ such that $\|\zeta_H \circ \psi_H\restriction_E - T \circ \zeta_G \circ \psi_G\restriction_E\| < \varepsilon'$. Let $K \coloneq T \circ \zeta_G(G) + \zeta_H(H)$, $\iota_G \coloneq T \circ \zeta_G$ and $\iota_H \coloneq \zeta_H$, both seen as elements of $\Emb(G, K)$ and $\Emb(H, K)$, respectively; then $K \in \Age(X)$ and $\|\iota_H \circ \psi_H \restriction_E - \iota_G \circ \psi_G \restriction_E\| < \varepsilon'$, hence showing that $(E, F)$ is an $(\Age(X), \varepsilon, \delta)$-amalgamation pair, witnessed by $\varepsilon'$ and $\delta'$.

\smallskip

If $X$ is guarded Fra\"iss\'e, it now follows from Remarks \ref{rem:gFFromPairs} and \ref{rem:equivAmalgPair} that $\Age(X)$ satisfies (GAP). Since it also satisfies (HP) and (JEP), it follows from Propositions \ref{prop:HPstableUnderClosures} and \ref{prop:ageAmalgThenItsClosureAlso} that $\closedAge{X}$ is a guarded Fra\"iss\'e class.
\end{proof}

\section{Proofs of the main theorems}\label{sec:GuardedFraisseCorrespondence}

This section contains the proofs of Theorems \ref{thm:GuardedFraisseCorrespondence} and \ref{thm:IvanovBanach}, and implication (5) $\Rightarrow$ (1) in Theorem \ref{thm:BMgameForIsometryClasses}. All will be consequences of a common technical result, Proposition \ref{prop:AmalgamationsAndGurarii}, that we will state and prove first. We will also see, at the end of the section, how this technical result can be used to deduce other Fra\"iss\'e correspondences; we will in particular state and prove one for cofinally Fra\"iss\'e Banach spaces (see Definition \ref{def:CofFraisse} and Theorem \ref{thm:CofFraisseCorresp}). Recall that the notation $\sigma\KK$, essential in this section, has been introduced in Definition~\ref{def:sigmaK}, and that $\sigma\KK$ is closed in $\PP$ whenever $\KK$ is closed in $\BM$ (see Proposition \ref{prop:closedSigmaK}).

\begin{proposition}\label{prop:AmalgamationsAndGurarii}
Let $\KK\subseteq \BM$ be a guarded Fra\"iss\'e class. Then there exists a unique (up to isometry) guarded Fra\"iss\'e Banach space $X$ satisfying $\KK = \closedAge{X}$. The space $X$ moreover satisfies the following.
\begin{enumerate}[label=(P-\arabic*), series=gAmalgGFra]  
    \item\label{it:gFraIsGDelta} $\isomtrclass{X}$ is a dense $G_\delta$ subset of $\sigma\KK$.
    
    \item \label{it:gAmalgThird} For every finite-dimensional subspace $E\subseteq X$ and every $\varepsilon>0$, there exist a finite-dimensional subspace $F \subseteq X$ with $E \subseteq F$ and $\delta>0$ for which $(E, F)$ is a $(\KK,\varepsilon,\delta)$-amalgamation pair.
    \item \label{it:gAmalgImpliesGFra} For every finite-dimensional subspaces $E\subseteq F\subseteq X$ and every $\varepsilon,\delta \in (0, 1)$, if $(E, F)$ is $(\KK,\varepsilon,\delta)$-amalgamation pair, then it  is an $(X,\varepsilon,\delta)$-\ref{it:gFraisseThird} pair.
\end{enumerate}
\end{proposition}

The proof of Proposition \ref{prop:AmalgamationsAndGurarii} relies on a Baire category argument; for more clarity, we divide it into a few technical lemmas.

\begin{lemma}\label{lem:sigmaKnonempty}
    Suppose that $\KK \subseteq \BM$ is nonempty and hereditary. Then $\sigma\KK \neq \varnothing$.
\end{lemma}

\begin{proof}
    Pick $E \in \KK$. Then by heredity, $\Age(E) \subseteq \KK$, so any $\mu \in \PP$ with $X_\mu$ isometric to $E$ belongs to $\sigma\KK$.
\end{proof}

Recall that, by Propositions \ref{prop:HPstableUnderClosures} and \ref{prop:ageAmalgThenItsClosureAlso}, if $\KK \subseteq \BM$ is of the form $\closedAge{X}$ for a separable Banach space $X$, then $\KK$ is  closed, hereditary, and satisfies (JEP). The following lemma, which constitutes the first part of our Baire category argument, provides a strong converse to this.

\begin{lemma}\label{lemma:gdeltaSigmaK}
Suppose $\KK \subseteq \BM$ is closed, hereditary, and satisfies (JEP). Then $\{\mu\in\sigma\KK \setsep \closedAge{X_\mu} = \KK\}$ is a dense $G_\delta$ subset of $\sigma\KK$.
\end{lemma}

\begin{proof}
Pick a countable dense subset $\mathcal{D} \subseteq \KK$. For each $E \in \mathcal{D}$ and $n \in \Nat$, let $U_{E, n} \coloneq \{\mu \in \sigma\KK \setsep \Emb_{<1/n}(E, X_\mu) \neq \varnothing\}$. Then:
$$\{\mu\in\sigma\KK \setsep \closedAge{X_\mu} = \KK\} = \bigcap_{E \in \mathcal{D}, \, n \in \Nat} U_{E, n},$$
so it is enough to show that each $U_{E, n}$ is a dense open subset of $\sigma\KK$.

\smallskip

So fix $E \in \mathcal{D}$ and $n \in \Nat$. To see that $U_{E, n}$ is open, fix $\mu \in U_{E, n}$ and $\phi \in \Emb_{< 1/n}(E, X_\mu)$; by Lemma \ref{lem:perturbationArgument}, left-composing $\phi$ with an isomorphism if necessary, one can assume that $\phi(E) = F_\Rea$ for some vector subspace $F \subseteq V$ such that $\mu\restriction_F \in \Norm(F)$. Define $\nu \in \Norm(F)$ by $\nu(\phi(x)) \coloneq \|x\|$ for every $x \in E$. Then given $\mu' \in \PP$, we have $\mu' \in V_\PP[\nu, F, 1/n]$ iff $\phi \in \Emb_{<1/n}(E, X_{\mu'})$; thus $V_\PP[\nu, F, 1/n] \cap \sigma\KK \subseteq U_{E, n}$. We conclude by observing that $V_\PP[\nu, F, 1/n]$ is a neighborhood of $\mu$ in $\PP$, by Proposition \ref{prop:equivalentNormOpen}.

\smallskip

To see that $U_{E, n}$ is dense in $\sigma\KK$, pick $\mu \in \sigma\KK$, a finite subset $A \subseteq V$, and $\varepsilon > 0$, and show that $U_\PP[\mu, A, \varepsilon] \cap U_{E, n} \neq \varnothing$. Let $F \coloneq \Span_\Rat(A)$. Then $E, (F_\Rea, \mu) \in \KK$, so by (JEP), one can find $G \in \KK$, $\phi \in \Emb(E, G)$ and $\psi \in \Emb((F_\Rea, \mu), G)$. Use Lemma \ref{lem:simpleReflection} to find a finite-dimensional vector subspace $H \subseteq V$ containing $F$, $\nu \in \PNorm(H)$, and an onto isometry $\chi \colon (H_\Rea, \nu) \to G$ extending $\psi$. Fix a vector subspace $W \subseteq V$ with $V = H \oplus W$ and extend $\nu$ to $\lambda \in \PP$ by letting $\lambda\restriction_W = 0$. Then for every $x \in F$, one has $\mu(x) = \|\psi(x)\| = \|\chi(x)\| = \nu(x) = \lambda(x)$, so $\lambda \in U_\PP[\mu, A, \varepsilon]$. The choice of the extension $\lambda$ of $\nu$ shows that $X_\lambda$ is isometric to $(H_\Rea, \nu)$, itself isometric to $G$. In particular, $X_\lambda \in \KK$, and since $\KK$ is hereditary, $\Age(X_\lambda) \subseteq \KK$ so $\lambda \in \sigma\KK$. Finally, $E$ isometrically embeds into $G$, so also in $X_\lambda$, thus $\lambda \in U_{E, n}$.
\end{proof}

Before going further, let us state a version of Lemma \ref{lemma:gdeltaSigmaK} for the codings $\PP_\infty$ and $\BB$; it will be useful in the proof of Theorem \ref{thm:IvanovBanach}.

\begin{corollary}\label{cor:gDeltaSigmaK}
    Suppose $\KK \subseteq \BM$ is closed, infinite-dimensional, hereditary, and satisfies (JEP). Let $\mathcal{I} \in \{\PP_\infty, \BB\}$. Then $\{\mu \in \sigma\KK \cap \mathcal{I} \setsep \closedAge{X_\mu} = \KK\}$ is a dense $G_\delta$ subset of $\sigma\KK \cap \mathcal{I}$.
\end{corollary}

\begin{proof}
Let $\mathcal{G} \coloneq \{\mu \in \sigma\KK \setsep \closedAge{X_\mu} = \KK\}$. We know by Lemma \ref{lemma:gdeltaSigmaK} that $\mathcal{G}\cap \mathcal{I}$ is a $G_\delta$ subset of $\sigma\KK \cap \mathcal{I}$, so what we need to show is that it is dense in $\sigma\KK \cap \mathcal{I}$. Knowing, again by Lemma \ref{lemma:gdeltaSigmaK}, that $\mathcal{G}$ is dense in $\sigma\KK$, it is enough to show that $\mathcal{G} \cap \mathcal{I}$  is dense in $\mathcal{G}$. But if $\mu \in \mathcal{G}$, then $X_\mu$ is infinite-dimensional, so by Lemma \ref{lem:equalClosures}, one has $\mu \in \isomtrclass[\PP]{\mu} \subseteq \closureOfIsomtrclass[\mathcal{I}]{\mu} \subseteq \overline{\mathcal{G} \cap \mathcal{I}}$, the last inclusion coming from the fact that $\mathcal{G}$ is invariant under isomorphism.
\end{proof}

The second part of out Baire category argument consists in proving that condition \ref{it:gAmalgThird} in Proposition \ref{prop:AmalgamationsAndGurarii} is satisfied by $X = X_\mu$  for comeagerly many $\mu \in \sigma\KK$. This will be done in Lemma \ref{lem:manygAmalgPairs}. We will first prove an abstract version of this lemma (Lemma \ref{lem:manyWidePairs}), which will also be useful later for obtaining our Fra\"iss\'e correspondence for cofinally Fra\"iss\'e Banach spaces. We start with a definition.

\begin{definition}\label{def:wide}
Let $\KK \subseteq \BM$. A family $\mathcal{W} \subseteq \Incl(\KK)$ is said to be \textit{wide} if it satisfies the three following conditions:
\begin{enumerate}[label=(\arabic*)]
\item\label{it:wideFirst} for every $(E, F) \in \mathcal{W}$ and every subspace $E' \subseteq E$ with $E' \in \KK$, one has $(E', F) \in \mathcal{W}$;
\item\label{it:wideSecond} the family $\mathcal{W}$ is open in $\Incl(\KK)$ in the following sense: for every $(E, F) \in \mathcal{W}$, there exists $\gamma > 0$ such that for every $(E', F') \in \Incl(\KK)$, if there exists an onto map $\sigma \in \Emb_{\gamma}(F, F')$ such that $\sigma(E) = E'$, then $(E', F') \in \mathcal{W}$;
\item\label{it:wideThird} the set $\Proj_1(\mathcal{W}) \coloneq \{E \in \KK \setsep (\exists F \in \KK) ((E, F) \in \mathcal{W})\}$ is dense in $\KK$.
\end{enumerate}
\end{definition}

\begin{lemma}\label{lem:amWide}
If $\KK \subseteq \BM$ satisfies (HP) and (GAP), then for every $\varepsilon \in (0, 1)$, the family $\Am_\varepsilon(\KK)$ of all $(E, F) \in \Incl(\KK)$ that are $(\KK, \varepsilon, \delta)$-amalgamation pairs for some $\delta > 0$ is wide.
\end{lemma}

\begin{proof}
Condition \ref{it:wideFirst} follows from Corollary \ref{cor:amalgPairUnderIsomorphisms} applied to $F' \coloneq F$ and $\sigma \coloneq \Id_F$. Condition \ref{it:wideSecond} also immediately follows from Corollary \ref{cor:amalgPairUnderIsomorphisms}. We now fix $E \in \KK$ and $\beta \in (0, \varepsilon/3)$ and, in order to check condition (\ref{it:wideThird}), we look for $(E', F) \in \Am_\varepsilon(\KK)$ such that $d_{\BM}(E, E') \leqslant 2\beta$. By Remark \ref{rem:equivAmalgPair}, we can find $F \in \KK$, $\phi \colon E \to F$ and $\delta \in (0, 1)$ such that $(E, F, \phi)$ is a $(\KK, \beta, \delta)$-amalgamation triple. Let $E' \coloneq \phi(E)$. Since $\phi \in \Emb_\beta(E, F)$, we have $d_{\BM}(E, E') \leqslant 2\beta$. Moreover, by Lemma \ref{lem:stabilityAmalgTriples} applied to $\varepsilon \coloneq \beta$, $\gamma \coloneq 0$, $F' \coloneq F$, $\rho \coloneq \phi^{-1}$ and $\sigma \coloneq \Id_F$, one gets that $(E', F, \Id_{E'})$ is a $(\KK, \varepsilon, \delta)$-amalgamation triple, so $(E', F) \in \Am_\varepsilon(\KK)$.
\end{proof}

\begin{lemma}\label{lem:manyWidePairs}
Let $\KK\subseteq \BM$ be closed and hereditary, $\mathcal{W} \subseteq \Incl(\KK)$ be wide, and $E \subseteq V$ be a finite-dimensional subspace. Denote by $\mathcal{G}$ the set of all $\mu\in\sigma\KK$ satisfying the following property: either $\mu\restriction_E \notin \Norm(E)$, or there exists a finite-dimensional subspace $F \subseteq V$ with $E \subseteq F$ such that $((E_\Rea, \mu\restriction_E), (F_\Rea, \mu\restriction_F)) \in \mathcal{W}$. Then $\mathcal{G}$ is a dense $G_\delta$ subset of $\sigma\KK$.
\end{lemma}

\begin{proof}
Write $\mathcal{G} = \mathcal{F} \cup \mathcal{U}$, where $\mathcal{F} \coloneq \{\mu \in \sigma\KK \setsep \mu\restriction_E \notin \Norm(E)\}$ and $\mathcal{U}$ is the set of all $\mu \in \sigma \KK$ such that $\mu\restriction_E \in \Norm(E)$ and there exists a finite-dimensional subspace $F \subseteq V$ with $E \subseteq F$ such that $((E_\Rea, \mu\restriction_E), (F_\Rea, \mu\restriction_F)) \in \mathcal{W}$. For each $\mu \in \sigma\KK \setminus \mathcal{F}$, we have $V_\PP[\mu\restriction_E, E, 1] \cap \mathcal{F} = \varnothing$, so by Proposition \ref{prop:equivalentNormOpen}, $\mathcal{F}$ is a closed subset of $\sigma\KK$. We now show that $\mathcal{U}$ is open in $\sigma\KK$; it will follow that $\mathcal{G}$ is $G_\delta$ in $\sigma\KK$. Let $\mu \in \mathcal{U}$ and $F \subseteq V$ be witnessing this membership; replacing if necessary $F$ with some subspace $F' \subseteq F$ containing $E$ such that $F_\Rea = (F')_\Rea \oplus (F_\Rea \cap \ker \mu)$, one can assume that $\mu\restriction_F \in \Norm(F)$. Let $\gamma$ be as given by condition \ref{it:wideSecond} in Definition \ref{def:wide} for the pair $((E_\Rea, \mu\restriction_E), (F_\Rea, \mu\restriction_F))$. We show that $\sigma\KK \cap V_{\PP}[\mu\restriction_F, F, \gamma] \subseteq \mathcal{U}$; by Proposition \ref{prop:equivalentNormOpen}, it is enough to conclude. Let $\nu \in \sigma\KK \cap V_{\PP}[\mu\restriction_F, F, \gamma]$. Then condition \ref{it:wideSecond} in Definition \ref{def:wide} applied to $(E', F') \coloneq ((E_\Rea, \nu\restriction_E), (F_\Rea, \nu\restriction_F))$ and $\sigma \coloneq \Id_{F_\Rea}$ ensures that $((E_\Rea, \nu\restriction_E), (F_\Rea, \nu\restriction_F)) \in \mathcal{W}$, and hence $\nu \in \mathcal{U}$.

\smallskip

We now show that $\mathcal{G}$ is dense in $\sigma\KK$. Let $U \subseteq \sigma\KK$ be a nonempty open subset, and fix $\widetilde{\mu} \in U$. Then by Proposition \ref{prop:alternativeBasisPrelim}, one can find a finite-dimensional subspace $F \subseteq V$ and $\varepsilon > 0$ for which $\sigma\KK \cap V_\PP[\mu, F, \varepsilon] \subseteq U$, where $\mu \coloneq \widetilde{\mu}\restriction_F$. One can even assume that $E \subseteq F$. Note that since $\widetilde{\mu} \in \sigma\KK$, we have $(F_\Rea, \mu) \in \KK$, so by condition \ref{it:wideThird} in Definition \ref{def:wide}, one can find $(F', G') \in \mathcal{W}$ with $d_{\BM}((F_\Rea, \mu), F') < \varepsilon$. Let $\phi \in \Emb_{<\varepsilon}((F_\Rea, \mu), G')$ be such that $\phi(F_\Rea) = F'$. Applying Lemma \ref{lem:simpleReflection} to $\phi$, one can find a vector subspace $G \subseteq V$ containing $F$, $\nu \in \PNorm(G)$, and an onto isometry $\psi \colon (G_\Rea, \nu) \to G'$ extending $\phi$. Fix a vector subspace $W \subseteq V$ with $V = G \oplus W$ and extend $\nu$ to $\lambda \in \PP$ by letting $\lambda\restriction_W = 0$. Similarly as in the proof of Lemma \ref{lemma:gdeltaSigmaK}, we see that $\lambda \in \sigma\KK$. The inclusion map $(F_\Rea, \mu) \to X_\lambda$ is equal to $\psi^{-1}\circ \phi$ which belongs to $\Emb_{< \varepsilon}((F_\Rea, \mu), X_\lambda)$, so $\lambda \in V_\PP[\mu, F, \varepsilon]$. We deduce that $\lambda \in U$. Finally, since $\psi$ is an onto isometry $(G_\Rea, \lambda\restriction_G) \to G'$ and induces an onto isometry $(F_\Rea, \lambda\restriction_F) \to F'$, and since $(F', G') \in \mathcal{W}$, we deduce that $((F_\Rea, \lambda\restriction_F), (G_\Rea, \lambda\restriction_G)) \in \mathcal{W}$; thus, by condition \ref{it:wideFirst} in Definition \ref{def:wide}, we have $((E_\Rea, \lambda\restriction_E), (G_\Rea, \lambda\restriction_G)) \in \mathcal{W}$, so $\lambda \in \mathcal{G}$.
\end{proof}

\begin{lemma}\label{lem:manygAmalgPairs}
Let $\KK\subseteq \BM$ be a guarded Fraiss\'e class. Then the set of all $\mu\in\sigma\KK$ 
such that $X = X_\mu$ satisfies condition \ref{it:gAmalgThird} in Proposition \ref{prop:AmalgamationsAndGurarii} is a dense $G_\delta$ subset of $\sigma\KK$.
\end{lemma}

\begin{proof}
For every finite-dimensional subspace $E \subseteq V$ and every $\varepsilon > 0$, denote by $\mathcal{G}_{E, \varepsilon}$ the set denoted by $\mathcal{G}$ in the statement of Lemma \ref{lem:manyWidePairs} relative to $E$ and $\mathcal{W} \coloneq \Am_\varepsilon(\KK)$ defined in Lemma~\ref{lem:amWide}. Let $\mathcal{G} \coloneq \bigcap_{E \subseteq V, \, \varepsilon \in \Rat \cap (0, 1)} \mathcal{G}_{E, \varepsilon}$, where $E$ varies among all finite-dimensional subspaces of $V$. Finally, denote by $\mathcal{H}$ the set of all $\mu\in\sigma\KK$ 
such that $X = X_\mu$ satisfies condition \ref{it:gAmalgThird} in Proposition \ref{prop:AmalgamationsAndGurarii}. It follows from Lemmas \ref{lem:amWide} and \ref{lem:manyWidePairs} and the Baire category theorem that $\mathcal{G}$ is a dense $G_\delta$ subset of $\sigma\KK$. Thus, to prove the lemma it is enough to prove that $\mathcal{G} = \mathcal{H}$.

\smallskip

Fix $\mu \in \sigma\KK$; in the rest of this proof, all real finite-dimensional spaces we consider will be subspaces of $X_\mu$, so in order to save notation, the norm will not be specified each time. First suppose that $\mu \in \mathcal{H}$. Fix a finite-dimensional subspace $E \subseteq V$ and $\varepsilon \in \Rat \cap (0, 1)$. Let $F \subseteq X_\mu$ be a finite-dimensional subspace with $E_\Rea \subseteq F$ and $\delta \in (0, 1)$ be such that $(E_\Rea, F)$ is a $(\KK, \varepsilon, \delta)$-amalgamation pair. Let $\gamma > 0$ be as given by Corollary \ref{cor:amalgPairUnderIsomorphisms} for the pair $(E_\Rea, F)$. By Lemma \ref{lem:openEquivalence}, we can find a finite-dimensional subspace $F'\subseteq V$ containing $E$ and an onto map $\sigma \in \Emb_\gamma(F, F'_\Rea)$ with $\sigma\restriction_{E_\Rea} = \Id_{(E_\Rea)}$. Thus, the choice of $\gamma$ ensures that $(E_\Rea, F'_\Rea)$ is a $(\KK, \varepsilon, \delta)$-amalgamation pair, and $\mu \in \mathcal{G}_{E, \varepsilon}$.

\smallskip

Conversely, suppose $\mu \in \mathcal{G}$. Fix a finite-dimensional subspace $E \subseteq X$ and $\varepsilon > 0$. Let $\varepsilon' \in \Rat \cap (0, 1)$ be such that $8\varepsilon' \leqslant \varepsilon$. By Lemma \ref{lem:openEquivalence}, one can find a finite-dimensional subspace $E' \subseteq V$ such that $\mu\restriction_{E'} \in \Norm(E')$ and an onto linear map $\rho \colon E \to E'_\Rea$ satisfying $\|\Id_E - \rho\| \leqslant \varepsilon'$. Since $\mu \in \mathcal{G}_{E', \varepsilon'}$, one can find a finite-dimensional subspace $F' \subseteq V$ containing $E'$ and $\delta \in (0, 1)$ such that $(E'_\Rea, F'_\Rea)$ is a $(\KK, \varepsilon', \delta)$-amalgamation pair. Let $F \coloneq E + F'_\Rea$. It follows from Corollary \ref{cor:amalgPairUnderIsomorphisms} that $(E'_\Rea, F)$ is also a $(\KK, \varepsilon', \delta)$-amalgamation pair, and it follows in turn from Lemma \ref{lem:approxAmalgPair} applied to $\rho$ that $(E, F)$ is an $(\KK, \varepsilon, \delta)$-amalgamation pair. Hence, $\mu \in \mathcal{H}$.
\end{proof}

We now turn to the third and last part of our Baire category argument, dealing with condition \ref{it:gAmalgImpliesGFra} in Proposition \ref{prop:AmalgamationsAndGurarii}. It is contained in the two next lemmas.

\smallskip

\begin{lemma}\label{lem:thirdBaireCatFirstPart}
   Let $\KK\subseteq\BM$ be a guarded Fra\" iss\' e class, $\varepsilon > 0$, $\delta > 0$, and $E\subseteq F \subseteq V$ be finite-dimensional vector subspaces. Then the set of pseudonorms $\mu\in\sigma\KK$ satisfying the following condition:
\begin{enumerate}[label=(BC-\arabic*),series=gAmalg]
       \item\label{cond:thirdBaireCatFirstPartA} if $\mu$ is a norm on $F_\Rea$ and $((E_\Rea, \mu), (F_\Rea, \mu))$ is a $(\KK, \varepsilon, \delta)$-amalgamation pair, then for every $G \in \KK$, $\psi \in \Emb_{<\delta} ((F_\Rea,\mu), G)$ and $\eta > 0$, there exists $\iota \in \Emb_{<\eta}(G, X_\mu)$ with $\|\Id_{E_\Rea} - \iota\circ\psi\restriction_{E_\Rea}\|_{\mathcal{L}((E_\Rea, \mu), X_\mu)} < \varepsilon$
   \end{enumerate}
   is a dense $G_\delta$ subset of $\sigma\KK$.
\end{lemma}

\begin{proof}Pick a Hamel basis $\mathfrak{b}_1$ of $E$ and extend it to a Hamel basis $\mathfrak{b}_2$ of $F$. For every $G\in\KK$, pick a countable dense subset $D_G\subseteq G$ and denote by $\Phi(G)$ the countable set of all mappings $\psi:\mathfrak{b}_2\rightarrow D_G$ having linearly independent range; we shall identify each $\psi\in\Phi(G)$ with its extension to an injective linear map $F_\Rea\to G$. We shall use the following sets:
\begin{itemize}
    \item the set $S$ of all $\mu \in \sigma\KK$ such that $\mu$ restricts to a norm on $F_\Rea$ and $((E_\Rea, \mu), (F_\Rea, \mu))$ is a $(\KK, \varepsilon, \delta)$-amalgamation pair;
    \item given $G\in\KK$, $\psi\in\Phi(G)$ and $\eta>0$, the set $T(G,\psi,\eta)$ of all $\mu\in S$ such that if $\psi\in \Emb_{<\delta}((F_\Rea,\mu),G)$ then there exists $\iota \in \Emb_{< \eta}(G, X_\mu)$ with $\|\Id_{E_\Rea} - \iota \circ \psi\restriction_{E_\Rea}\nolinebreak\|_{\mathcal{L}((E_\Rea, \mu\restriction_E), X_\mu)} < \varepsilon$.
\end{itemize}
Our general strategy will be to show that each $(\sigma\KK\setminus S) \cup T(G,\psi,\eta)$ is a dense $G_\delta$ subset of $\sigma\KK$, and then to intersect a countable family of those.

\medskip

\textbf{Step 1: }Given $G\in\KK$, $\psi\in\Phi(G)$, and $\eta > 0$, the set $(\sigma\KK\setminus S) \cup T(G,\psi,\eta)$ is $G_\delta$ in $\sigma\KK$.
\begin{proof}[Proof of Step 1] Denote by $\mu_G$ the unique element of $\Norm(F)$ such that \linebreak $\psi \in \Emb((F_\Rea, \mu_G), G)$. Given $\mu\in S$, it follows from Corollary~\ref{cor:amalgPairUnderIsomorphisms} that $V_\PP[\mu\restriction_{F}\nolinebreak,F,\gamma] \cap \sigma \KK \subseteq S$ for a sufficiently small $\gamma>0$, so by Proposition \ref{prop:equivalentNormOpen}, $S$ is an open subset of $\sigma\KK$. Also observe that $T(G,\psi,\eta) = (S\setminus V_\PP[\mu_G, F, \delta]) \cup U(G,\psi,\eta)$, where the set $U(G,\psi,\eta)$ consists of those $\mu\in S$ for which there exists $\iota \in \Emb_{< \eta}(G, X_\mu)$ with $\|\Id_{E_\Rea} - \iota \circ \psi\restriction_{E_\Rea}\nolinebreak\|_{\mathcal{L}((E_\Rea, \mu\restriction_E), X_\mu)} < \varepsilon$. Since $V_\PP[\mu_G, F, \delta]$ is open in $\PP$ by Proposition \ref{prop:equivalentNormOpen}, to achieve our first step it suffices to show $U(G,\psi,\eta)$ is open in $S$. For this, let $\mu \in U(G,\psi,\eta)$ and a witnessing map $\iota$. Using the perturbation Lemma~\ref{lem:perturbationArgument} we may assume that $\iota(G) = H_\Rea$ and $(\Id_{E_\Rea} - \iota \circ \psi\restriction_{E_\Rea})(E_\Rea) = K_\Rea$ for vector subspaces $H, K \subseteq V$ on which $\mu$ restricts to a norm. Fix $\gamma \in (0, \eta)$ such that $\iota \in \Emb_{< \eta - \gamma}(G, X_\mu)$ and $e^{2\gamma} \cdot \|\Id_{E_\Rea} - \iota \circ \psi\restriction_{E_\Rea}\nolinebreak\|_{\mathcal{L}((E_\Rea, \mu\restriction_E), X_\mu)} < \varepsilon$. We show that $S \cap V_\PP[\mu\restriction_{E}\nolinebreak, E, \gamma] \cap V_\PP[\mu\restriction_{H}\nolinebreak, H, \gamma] \cap V_\PP[\mu\restriction_{K}\nolinebreak, K, \gamma] \subseteq U(G,\psi,\eta)$, and the desired result will follow from Proposition~\ref{prop:equivalentNormOpen}. Let $\nu \in S \cap V_\PP[\mu\restriction_{E}\nolinebreak, E, \gamma] \cap V_\PP[\mu\restriction_{H}\nolinebreak, H, \gamma] \cap V_\PP[\mu\restriction_{K}, K, \gamma]$. Then the map $\iota \colon G \to X_\nu$ can be written as the composition
$$G \xrightarrow{\iota} (H_\Rea, \mu\restriction_H) \xrightarrow{\Id_{H_\Rea}} X_\nu,$$
where the first map belongs to $\Emb_{< \eta - \gamma}(G, (H_\Rea, \mu\restriction_H))$ and the second to \linebreak $\Emb_{< \gamma}((H_\Rea, \mu\restriction_H\nolinebreak), X_\nu)$, so $\iota \in \Emb_{< \eta}(G, X_\nu)$. Similarly, the map $\Id_{E_\Rea} - \iota \circ \psi\restriction_{E_\Rea}$, seen as a map $(E_\Rea, \nu\restriction_E) \to X_\nu$, can be written as the composition
$$(E_\Rea, \nu\restriction_E) \xrightarrow{\Id_{E_\Rea}} (E_\Rea, \mu\restriction_E) \xrightarrow{\Id_{E_\Rea} - \iota \circ \psi\restriction_{E_\Rea}} (K_\Rea, \mu\restriction_{K}) \xrightarrow{\Id_{K_\Rea}} X_\nu,$$
where the first and last map have norm at most $e^\gamma$. Thus, $\|\Id_{E_\Rea} - \iota \circ \psi\restriction_{E_\Rea}\nolinebreak\|_{\mathcal{L}((E_\Rea, \nu\restriction_E), X_\nu)} < \varepsilon$, and $\nu \in U(G, \psi, \eta)$.
\end{proof}

\textbf{Step 2: }Given $G\in \KK$, $\psi\in\Phi(G)$, and $\eta > 0$, the set $(\sigma\KK\setminus S)\cup T(G,\psi,\eta)$ is dense in $\sigma\KK$.
\begin{proof}[Proof of Step 2] It suffices to prove that $S$ is contained in the closure of $T(G,\psi,\eta)$. Pick $\widetilde{\nu}\in S$. If $\psi\notin\Emb_{<\delta}((F_\Rea,\widetilde{\nu}),G)$ then $\widetilde{\nu} \in T(G,\psi,\eta)$ and we are done, so from now on, we will assume that $\psi\in\Emb_{<\delta}((F_\Rea,\widetilde{\nu}),G)$. We now pick a neighborhood $U$ of $\widetilde{\nu}$ in $\PP$ and show that $U \cap T(G,\psi,\eta) \neq \varnothing$. By Proposition \ref{prop:alternativeBasisPrelim}, there exists a finite-dimensional subspace $H \subseteq V$ and $\gamma > 0$ such that $V_\PP[\nu, H_\Rea, \gamma] \subseteq U$, where $\nu \coloneq \widetilde{\nu}\restriction_H$. Extending $H$ if necessary, we can even assume that $F \subseteq H$. Since $\widetilde{\nu} \in S$, it follows that $(H_\Rea, \nu) \in \KK$, $\nu$ restricts to a norm on $F_\Rea$ and $((E_\Rea, \nu\restriction_E), (F_\Rea, \nu\restriction_E))$ is a $(\KK, \varepsilon, \delta)$-amalgamation pair. Applying this latter fact to $\psi_G:=\psi$ and $\psi_H$ being the inclusion map $(F_\Rea,\nu\restriction_F) \to (H_\Rea,\nu)$ we obtain $K\in\KK$, $\iota_G\in \Emb(G,K)$ and $\iota_H\in \Emb((H_\Rea,\nu),K)$ such that 
\begin{equation}\label{Eq:10}
    \|\iota_G\circ \psi_G\restriction_{E_\Rea} - \iota_H\circ\psi_H\restriction_{E_\Rea}\|_{\mathcal{L}((E_\Rea,\nu\restriction_E),K)}<\varepsilon.
\end{equation}
By Lemma~\ref{lem:simpleReflection} there exists
a finite-dimensional vector subspace $L \subseteq V$ containing $H$, $\mu \in \PNorm(L)$, and an onto isometry $\chi \colon (L_\Rea, \mu) \to K$ extending $\iota_H$. Fix a vector subspace $W \subseteq V$ such that $V = L \oplus W$, and extend $\mu$ to a pseudonorm (still denoted by $\mu$) on the whole $c_{00}$ by letting $\mu\restriction_W = 0$. Since $\chi$ extends $\iota_H\in \Emb((H_\Rea,\nu),K)$, we have $\mu\restriction_{H} = \nu$. In particular,  $\mu \in V_\PP[\nu, H_\Rea, \gamma] \subseteq U$. We now conclude by proving that $\mu \in T(G, \psi, \eta)$. For this, first observe that $\chi$ induces an onto isometry $X_\mu \to K$ (that we will still denote by $\chi$); in particular, $X_\mu \in \KK$, thus $\mu \in \sigma\KK$. From this plus the fact that $\mu$ extends $\nu$, it follows that $\mu \in S$. Finally, let $\iota \coloneq \chi^{-1} \circ \iota_G \in \Emb(G, X_\mu)$. Left-composing with $\chi^{-1}$ in inequality (\ref{Eq:10}), we get $\|\iota \circ \psi\restriction_{E_\Rea} - \Id_{E_\Rea}\nolinebreak\|_{\mathcal{L}((E_\Rea, \mu\restriction_E), X_\mu)} < \varepsilon$, so $\mu \in T(G, \psi, \eta)$.
\end{proof}

\medskip

\textbf{Step 3: }The last step is now to prove that $\mu\in\sigma\KK$ satisfies the condition \ref{cond:thirdBaireCatFirstPartA} if and only if it belongs to the set
\[
T:=\bigcap_{G\in D_\KK}\bigcap_{\psi\in \Phi(G)}\bigcap_{\eta>0, \eta\in\Rat} \left[(\sigma\KK\setminus S)\cup T(G,\psi,\eta)\right],
\]
where $D_\KK\subseteq \KK$ is a countable dense subset of $\KK$.
\begin{proof}[Proof of Step 3]Obviously, if condition \ref{cond:thirdBaireCatFirstPartA} holds for $\mu$, then $\mu\in T$, because elements of $T$ are those satisfying the version of condition \ref{cond:thirdBaireCatFirstPartA} where only $G$'s in $D_\KK$, $\psi$'s in $\Phi(G)$, and $\eta$'s in $\Rat$ are considered. Let us prove the other inclusion. Pick $\mu\in T$, and suppose $\mu$ is restricts to a norm on $F_\Rea$ and $((E_\Rea, \mu), (F_\Rea, \mu))$ is a $(\KK, \varepsilon, \delta)$-amalgamation pair. Further, pick $G\in\KK$, $\eta>0$ and $\psi\in\Emb_{<\delta}((F_\Rea,\mu),G)$. Find $G'\in D_\KK$ and a surjective isomorphism $\theta:G\to G'$ such that $\theta\in\Emb_{<\eta/2}(G,G')$ and $\theta\circ \psi\in \Emb_{<\delta}((F_\Rea,\mu),G')$. Left-composing $\theta$ with an isomorphism $G' \to G'$ if necessary, by the perturbation Lemma~\ref{lem:perturbationArgument} we may assume that $(\theta\circ \psi)(\mathfrak{b}_2)\subseteq D_{G'}$. Thus we have $(\theta\circ \psi)\in \Phi(G')$. Using that $\mu\in T$ we then find $\iota\in\Emb_{<\eta/2}(G',X_\mu)$ such that $\|\Id_{E_\Rea} - \iota\circ \theta\circ \psi\restriction_{E_\Rea}\|_{\mathcal{L}((E_\Rea, \mu\restriction_E), X_\mu)}<\varepsilon$. Since $\iota\circ \theta\in \Emb_{<\eta}(G,X_\mu)$ this shows that the \ref{cond:thirdBaireCatFirstPartA} is satisfied for $G$, $\psi$ and $\eta$.
\end{proof}
The proof of Lemma~\ref{lem:thirdBaireCatFirstPart} follows now immediately from Steps 1-3.
\end{proof}

\begin{lemma}\label{lem:thirdBaireCatSecondPart}
    Let $\KK\subseteq\BM$ be a guarded Fra\" iss\' e class. Then the set of all $\mu\in\sigma\KK$ 
such that $X = X_\mu$ satisfies condition \ref{it:gAmalgImpliesGFra} in Proposition \ref{prop:AmalgamationsAndGurarii}
    is a dense $G_\delta$ subset of $\sigma\KK$.
\end{lemma}

\begin{proof}Given $\varepsilon,\delta \in (0, 1)$ and finite-dimensional subspaces $E \subseteq F \subseteq V$ we denote by $T(\varepsilon,\delta,E,F)$ those $\mu\in \sigma\KK$ which satisfy \ref{cond:thirdBaireCatFirstPartA} for those $\varepsilon$, $\delta$, $E$ and $F$. Using Lemma~\ref{lem:thirdBaireCatFirstPart} it suffices to prove that given $\mu\in \sigma\KK$, $X=X_\mu$ satisfies condition \ref{it:gAmalgImpliesGFra} in Proposition \ref{prop:AmalgamationsAndGurarii} if and only if $\mu$ belongs to the set
\[
T:=\bigcap_{\varepsilon,\delta\in (0,1)\cap \Rat} \bigcap_{E\subseteq F \subseteq V} T(\varepsilon,\delta,E,F).
\]
Pick $\mu\in \sigma\KK$. First, assume that $\mu \in T$ and, in order to check that $X = X_\mu$ satisfies condition \ref{it:gAmalgImpliesGFra} in Proposition \ref{prop:AmalgamationsAndGurarii}, pick $\varepsilon,\delta\in (0,1)$ and finite-dimensional spaces $E\subseteq F\subseteq X$ such that $(E,F)$ is a $(\KK,\varepsilon,\delta)$-amalgamation pair. Fix $\varepsilon' \in (0, \varepsilon)$ and $\delta' > \delta$, both in $\Rat$, such that $(E,F)$ is still a $(\KK,\varepsilon',\delta')$-amalgamation pair. Using Corollary~\ref{cor:amalgPairUnderIsomorphisms} we find $\gamma>0$ with $\gamma < \min\{\tfrac{\varepsilon-\varepsilon'}{4}, \tfrac{\delta'-\delta}{3}\}$ such that whenever $(E',F')\in \Incl(\KK)$ satisfy that there exists $\sigma\in\Emb_\gamma(F,F')$ with $E' \subseteq \sigma(E)$ then $(E',F')$ is a $(\KK,\varepsilon',\delta')$-amalgamation pair. By the perturbation Lemma~\ref{lem:perturbationArgument} we find finite-dimensional subspaces $E'\subseteq F'\subseteq V$ such that $\mu$ is a norm on $F'_\Rea$ and there exists a surjective map $\sigma\in\Emb_\gamma(F,F'_\Rea)$ with $E'_\Rea = \sigma(E)$ and $\|\Id_{E} - \sigma\restriction_{E}\|<\gamma$. It follows that $(E'_\Rea,F'_\Rea)$ is a $(\KK,\varepsilon',\delta')$-amalgamation pair. So since $\mu \in T$, it is also an $(X,\varepsilon' + \gamma,\delta' - 2\gamma)$-\ref{it:gFraisseThird} pair, witnessed by $\varepsilon'$ and $\delta' - \gamma$. Thus, by Corollary~\ref{cor:gFraPairIsomorph} we obtain that $(E,F)$ is $(X,\varepsilon'+4\gamma,\delta'-3\gamma)$-\ref{it:gFraisseThird} pair and therefore it is also an $(X,\varepsilon,\delta)$-\ref{it:gFraisseThird} pair. Since $\varepsilon$, $\delta$, $E$ and $F$ were arbitrary, we obtain that $X=X_\mu$ satisfies condition \ref{it:gAmalgImpliesGFra} in Proposition \ref{prop:AmalgamationsAndGurarii}.

On the other hand, if $X=X_\mu$ satisfies condition \ref{it:gAmalgImpliesGFra} in Proposition \ref{prop:AmalgamationsAndGurarii}, then we easily observe that $\mu\in T$, so the above mentioned equivalence holds and this finishes the proof.
\end{proof}

We are now ready to prove Proposition~\ref{prop:AmalgamationsAndGurarii}.

\begin{proof}[Proof of Proposition~\ref{prop:AmalgamationsAndGurarii}]
Let $\mathcal{G}$ be the set of all $\mu \in \sigma\KK$ such that $\closedAge{X_\mu} = \KK$ and $X = X_\mu$ satisfies conditions 
\ref{it:gAmalgThird} and \ref{it:gAmalgImpliesGFra}. Let $\mathcal{H}$ be the set of all $\mu \in \sigma\KK$ for which $\closedAge{X_\mu} = \KK$ and $X_\mu$ is a guarded Fra\"iss\'e Banach space. By Remark \ref{rem:gFFromPairs}, we have $\mathcal{G} \subseteq \mathcal{H}$. By Lemmas \ref{lemma:gdeltaSigmaK}, \ref{lem:manygAmalgPairs} and \ref{lem:thirdBaireCatSecondPart}, $\mathcal{G}$ is a dense $G_\delta$ subset of $\sigma \KK$. Remembering that $\sigma\KK \neq \varnothing$ (see Lemma \ref{lem:sigmaKnonempty}), it follows that $\mathcal{G} \neq \varnothing$, so $\mathcal{H} \neq \varnothing$, too. It follows from Theorem \ref{thm:guardedUniqueByAge} that $\mathcal{H}$ is of the form $\isomtrclass{X}$ for some separable Banach space $X$. Since $\mathcal{G} \subseteq \mathcal{H} = \isomtrclass{X}$ and $\mathcal{G}$ is nonempty and obviously invariant under isometry, it follows that $\mathcal{G} = \mathcal{H} = \isomtrclass{X}$. The proposition follows.
\end{proof}

We now prove the main theorems of this first part.

\begin{proof}[Proof of Theorem \ref{thm:BMgameForIsometryClasses}]
    Implications (1) $\Rightarrow$ (2) and (3) $\Rightarrow$ (4) are obvious. Implications (2) $\Rightarrow$ (3) and (4) $\Rightarrow$ (5) are respectively Proposition \ref{prop:comeagerAndBMGame} and \ref{prop:winningI}. We now prove (5) $\Rightarrow$ (1). Let $X$ be a guarded Fra\"iss\'e Banach space. Let $\KK \coloneq \closedAge{X}$. Then by Proposition \ref{prop:ageOfGuardedHasAmalg}, $\KK$ is a guarded Fra\"iss\'e class, so by Proposition \ref{prop:AmalgamationsAndGurarii}, $X$ is the unique guarded Fra\"iss\'e Banach space whose closed age is $\KK$, and $\isomtrclass[\PP]{X}$ is a $G_\delta$ subset of $\sigma\KK$, so it is also a $G_\delta$ subset of $\PP$. Hence $\isomtrclass[\mathcal{I}]{X}$ is a $G_\delta$ subset of $\mathcal{I}$.
\end{proof}

\begin{proof}[Proof of Theorem \ref{thm:GuardedFraisseCorrespondence}]
    This theorem immediately follows from Propositions \ref{prop:ageOfGuardedHasAmalg} and \ref{prop:AmalgamationsAndGurarii}.
\end{proof}

Now that Theorem \ref{thm:GuardedFraisseCorrespondence} has been proved, the notion of the Fra\"iss\'e limit of a guarded Fra\"iss\'e class is correctly defined, and we will freely use it in the rest of this paper.

\begin{proof}[Proof of Theorem \ref{thm:IvanovBanach}]
    Suppose (1) holds and let $X \coloneq \Flim(\KK)$.
    If $X$ is infinite-dimensional, then one can pick $\mu \in \sigma\KK \cap \mathcal{I}$ such that $X_\mu \equiv X$; if $X$ is finite-dimensional, then $\KK = \closedAge{X}$ cannot be infinite-dimensional, so $\mathcal{I} = \PP$ and such a choice of $\mu$ can also be made. The space $X_\mu$ is guarded Fra\"iss\'e, so $\isomtrclass[\mathcal{I}]{\mu}$ is comeager in its closure $\mathcal{I}$ by Theorem \ref{thm:BMgameForIsometryClasses}, this closure being $\sigma\KK \cap \mathcal{I}$ by Proposition \ref{prop:closureOfIsomClass}.

    \smallskip

    Now suppose (2) holds, and fix $\mu$ witnessing it.
    We know, by Lemma \ref{lemma:gdeltaSigmaK} when $\mathcal{I} = \PP$ and Corollary \ref{cor:gDeltaSigmaK} when $\mathcal{I} \in \{\PP_\infty, \BB\}$, that  $\mathcal{G} \coloneq \{\nu \in \sigma\KK \cap \mathcal{I} \setsep \closedAge{X_\nu} = \KK\}$ is comeager in $\sigma\KK \cap \mathcal{I}$. Since $\isomtrclass[\mathcal{I}]{\mu}$ is also comeager in $\sigma\KK \cap \mathcal{I}$, we deduce that $\mathcal{G} \cap \isomtrclass[\mathcal{I}]{\mu} \neq \varnothing$, thus $\closedAge{X_\mu} = \KK$. Moreover, since $\sigma\KK \cap \mathcal{I}$ is closed in $\mathcal{I}$, it follows that $\isomtrclass[\mathcal{I}]{\mu}$ is comeager in its closure in $\mathcal{I}$, so by Theorem \ref{thm:BMgameForIsometryClasses}, $X_\mu$ is a guarded Fra\"iss\'e Banach space.  It follows by Proposition \ref{prop:ageOfGuardedHasAmalg} that $\KK$ is a guarded Fra\"iss\'e class.
\end{proof}

Proposition \ref{prop:AmalgamationsAndGurarii} also has the following useful byproduct.

\begin{corollary}\label{cor:equivAmalgGFPairs}
Let $X$ be a guarded Fra\"iss\'e Banach space, $\varepsilon, \delta \in (0, 1)$, and $E \subseteq F \subseteq X$ be finite-dimensional subspaces.
\begin{enumerate}
    \item If $(E, F)$ is an $(X, \varepsilon, \delta)$-\ref{it:gFraisseAction} pair, then it is an $(\closedAge{X}, \varepsilon, \delta)$-amalgamation pair.
    \item If $(E, F)$ is an $(\closedAge{X}, \varepsilon, \delta)$-amalgamation pair, then it is an $(X, 2\varepsilon, \delta)$-\ref{it:gFraisseAction} pair.
\end{enumerate}
\end{corollary}

\begin{proof}
(1) follows from Proposition \ref{prop:ageOfGuardedHasAmalg} and Lemma \ref{lem:pairsUnderClosures}. For (2), condition \ref{it:gAmalgImpliesGFra} in Proposition \ref{prop:AmalgamationsAndGurarii} applied to $\KK \coloneq \closedAge{X}$ shows that $(E, F)$ is an $(X, \varepsilon, \delta)$-\ref{it:gFraisseThird} pair, so by Corollary \ref{cor:implicationsPairs}, it is an $(X, 2\varepsilon, \delta)$-\ref{it:gFraisseAction} pair.
\end{proof}

Corollary \ref{cor:equivAmalgGFPairs} can be used as a tool for proving restricted versions of our Fra\"iss\'e correspondence, for subcalsses of the class of all guarded Fra\"iss\'e Banach spaces. For instance, Ferenczi--Lopez-Abad--Mbombo--Todorcevic's Fra\"iss\'e correspondence from \cite{FLT} (i.e. Theorem \ref{thm:FraisseCorrespFLMT}), can be recovered in this way, as shown below.

\begin{lemma}\label{lem:FraisseCorrespFLMT}
    Let $X$ be a guarded Fra\"iss\'e Banach space. The following are equivalent:
    \begin{enumerate}
        \item $X$ is Fra\"iss\'e;
        \item $\Age(X)$ is an amalgamation class;
        \item $\closedAge{X}$ is a Fra\"iss\'e class.
    \end{enumerate}
\end{lemma}

\begin{proof}
    Observe that:
    \begin{itemize}
        \item A Banach space $X$ is Fra\"iss\'e iff for every $\varepsilon > 0$ and $d \in \Nat$, there exists $\delta > 0$ such that for every subspace $E \subseteq X$ of dimension $d$, the pair $(E, E)$ is an $(X, \varepsilon, \delta)$-\ref{it:gFraisseAction} pair.
        \item A class $\KK \subseteq \BM$ is an amalgamation class iff for every $\varepsilon > 0$ and $d \in \Nat$, there exists $\delta > 0$ such that for every $E \in \KK$ of dimension $d$, the pair $(E, E)$ is a $(\KK, \varepsilon, \delta)$-amalgamation pair.
    \end{itemize}
    Using those observations, (1) $\Rightarrow $ (2) follows from Proposition \ref{prop:ageOfGuardedHasAmalg}, (2) $\Rightarrow$ (3) follows from \cite[Proposition 2.24]{FLT}, and (3) $\Rightarrow$ (1) follows from Corollary \ref{cor:equivAmalgGFPairs}.
\end{proof}

\begin{proof}[Proof of Theorem \ref{thm:FraisseCorrespFLMT}]
    This is an immediate consequence of our Fra\"iss\'e correspondence (Theorem \ref{thm:GuardedFraisseCorrespondence}), Lemma \ref{lem:FraisseCorrespFLMT}, and the fact that the age of a Fra\"iss\'e Banach space is always closed (see \cite[Theorem 2.12]{FLT}).
\end{proof}

Another interesting Fra\"iss\'e correspondence concerns \textit{cofinally Fra\"iss\'e Banach spaces}, which we define below (along with the associated notion for classes). Here, if $E$ and $F$ are two subspaces of the same Banach space, and $\varepsilon > 0$, we write $E \subseteq_\varepsilon F$ if for every $x \in S_E$, $d(x, S_F) \leqslant \varepsilon$.

\begin{definition}\label{def:CofFraisse}
    Let $X$ be a separable Banach space.
    \begin{itemize}
        \item A class $\mathcal{F}$ of finite-dimensional subspaces of $X$ is said to be \textit{cofinal} in $X$ if for every finite-dimensional subspace $E \subseteq X$ and every $\varepsilon > 0$, there exists $F \in \mathcal{F}$ such that $E \subseteq_\varepsilon F$.
        \item The space $X$ is said to be \textit{cofinally Fra\"iss\'e} if for every $\varepsilon>0$ the set \[\Fr_\varepsilon(X):=\big\{F\in\Age(X)\setsep \exists\delta>0\;\big(\Iso(X)\curvearrowright \Emb_\delta(F,X)\text{ is }\varepsilon\text{-transitive}\big)\big\}\] is cofinal in $X$. Recall that $\Iso(X)\curvearrowright\Emb_\delta(F,X)$ being $\varepsilon$-transitive means that for every $\phi,\psi\in\Emb_\delta(F,X)$ there is $T\in\Iso(X)$ such that $\|\psi-T\circ\phi\|<\varepsilon$. Note that $(F,F)$ is $(X,\varepsilon,\delta)$-\ref{it:gFraisseAction} pair for some $\delta>0$ if and only if there exists $\varepsilon'\in (0,\varepsilon)$ such that $F\in \Fr_{\varepsilon'}(X)$. This observation will be used quite frequently in proofs below.
    \end{itemize}
\end{definition}

\begin{definition}\label{def:CofFraisseClass}
    Let $\KK \subseteq \BM$.
    \begin{itemize}
        \item A subclass $\mathcal{L} \subseteq \KK$ is said to be \textit{cofinal} in $\KK$ if for every $E \in \KK$ and every $\varepsilon > 0$, there exists $F \in \mathcal{L}$ such that $\Emb_\varepsilon(E, F) \neq \varnothing$.
        \item Say that $\KK$ satisfies the \textit{cofinal amalgamation property (CAP)} if for every $\varepsilon > 0$, the class $\Fr_\varepsilon(\KK)$ of all $F \in \KK$ such that $(F, F)$ is a $(\KK, \varepsilon, \delta)$-amalgamation pair for some $\delta > 0$ is cofinal in $\KK$.
        \item Say that $\KK$ is a \textit{cofinally Fra\"iss\'e class} if it is nonempty, closed, satisfies (HP), (JEP) and (CAP).
    \end{itemize}
\end{definition}

\begin{proposition}\label{prop:cofinallySpaces}
    Let $X$ be a separable Banach space. Then the following implications hold: $X$ is weak Fra\"iss\'e $\Rightarrow$ $X$ is cofinally Fra\"iss\'e $\Rightarrow$ $X$ is guarded Fra\"iss\'e.
\end{proposition}

\begin{proof}
    The space $X$ is weak Fra\"iss\'e iff for every $\varepsilon > 0$, the class $\Fr_\varepsilon(X)$ contains all finite-dimensional subspaces of $X$. In this case, this class is in particular cofinal, so $X$ is cofinally Fra\"iss\'e. Now suppose that $X$ is cofinally Fra\"iss\'e and fix a finite-dimensional subspace $E \subseteq X$ and $\varepsilon \in (0, 1)$.  By Lemma \ref{lem:perturbationArgument}, we can find $F \in \Fr_{\varepsilon/5}(X)$ and $\rho \in \Emb_{\varepsilon/4}(E, F)$ such that $\|\Id_E - \rho\| \leqslant \varepsilon/4$. Since $F\in \Fr_{\varepsilon/5}(X)$, there exists $\delta>0$ such that $(F,F)$ is $(X,\varepsilon/4,\delta)$-\ref{it:gFraisseAction} pair and so by Corollary \ref{cor:implicationsPairs} and Lemma \ref{lem:equivalencePairsTriples}, $(F, F, \Id_F)$ is an $(X, \varepsilon/4, \delta)$-\ref{it:gFraisseForPlayerI} triple, so by Lemma \ref{lem:stabilityGFTriple}, $(E, F, \rho)$ is an $(X, \varepsilon, \delta)$-\ref{it:gFraisseForPlayerI} triple. We conclude that $X$ is guarded Fra\"iss\'e using Remark \ref{rem:gFFromPairs}.
\end{proof}

Note that in the proof above it was important that our notion of cofinality involves ``$E\subseteq_\varepsilon F$'' and not only the weaker condition ``$\Emb_\varepsilon(E,F)\neq \emptyset$''. In the following proof however, the latter (weaker) notion is sufficient. This is of the reasons why the conditions on cofinality in Definitions~\ref{def:CofFraisse} and \ref{def:CofFraisseClass} are different (the other reason, more intuitive but less direct, is that this is how analogical notions are defined for discrete structures).

\begin{proposition}
    Let $\KK \subseteq \BM$. Then the following implications hold: $\KK$ is a weak amalgamation class $\Rightarrow$ $\KK$ satisfies (CAP) $\Rightarrow$ $\KK$ satisfies (GAP).
\end{proposition}

\begin{proof}
    The class $\KK$ is a weak amalgamation class iff for every $\varepsilon > 0$, $\Fr_\varepsilon(\KK) = \KK$. From this condition it follows that $\KK$ satisfies (CAP). Now suppose that $\KK$ satisfies (CAP), and fix $E \in \KK$ and $\varepsilon \in (0, 1)$. Then we can find $F \in \Fr_{\varepsilon/3}(\KK)$ and $\rho \in \Emb_{\varepsilon/3}(E, F)$; from Lemma \ref{lem:stabilityAmalgTriples}, it follows that $(E, F, \rho)$ is an $(\KK, \varepsilon, \delta)$-amalgamation triple for some $\delta > 0$. We conclude using Remark \ref{rem:equivAmalgPair}.
\end{proof}

\begin{example}
    Although not all $L_p$'s, $1 \leqslant p < \infty$, are weak Fra\"iss\'e (see \cite[Proposition 2.10]{FLT}), all of them are cofinally Fra\"iss\'e. This immediately follows from \cite[Proposition 3.7(2)]{FLT}, which states that \textit{$L_p$ is the Fra\"iss\'e limit of the class $\{\ell_p^n \setsep n \in \Nat\}$}; the notion of \textit{Fra\"iss\'e limit} used in \cite{FLT} is valid for non-hereditary classes, and in our language, the latter statement can be translated as the conjunction of the two following results:
    \begin{itemize}
        \item for every $\varepsilon > 0$ and every $n \in \Nat$, there exists $\delta > 0$ such that $\Iso(L_p)\curvearrowright\Emb_\delta(\ell_p^n,L_p)$ is $\varepsilon$-transitive.
        \item the class of all subspaces $E \subseteq \ell_p$ that are isometric copies of an $\ell_p^n$, $n \in \Nat$, is cofinal in $L_p$.
    \end{itemize}
\end{example}

We now state and prove our Fra\"iss\'e correspondance for cofinally Fra\"iss\'e Banach spaces; we start with a lemma playing a similar role as Lemma \ref{lem:FraisseCorrespFLMT}.

\begin{lemma}\label{lem:CofFraisseCorresp}
    Let $X$ be a guarded Fra\"iss\'e Banach space. The following are equivalent:
    \begin{enumerate}[label=(\arabic*)]
        \item\label{it:cofinallyFraisseSpace} $X$ is cofinally Fra\"iss\'e;
        \item\label{it:cofinallyFraisseClass} $\closedAge{X}$ is a cofinally Fra\"iss\'e class.
    \end{enumerate}
\end{lemma}

\begin{proof}
    Suppose \ref{it:cofinallyFraisseSpace}, and let $\varepsilon > 0$. Since $\Fr_{\varepsilon}(X)$ is cofinal in $X$, it follows from Lemma \ref{lem:perturbationArgument} that $\Fr_\varepsilon(X)$, when seen as a subclass of $\Age(X)$, is cofinal in $\Age(X)$. Thus, it is also cofinal in $\closedAge{X}$. To conclude that \ref{it:cofinallyFraisseClass} holds, just observe that, by Corollary \ref{cor:equivAmalgGFPairs}, $\Fr_\varepsilon(X) \subseteq \Fr_{2\varepsilon}(\closedAge{X})$.

    \smallskip

    Now suppose \ref{it:cofinallyFraisseClass} holds. Let $\KK \coloneq \closedAge{X}$, and fix $\varepsilon \in (0, 1)$. Let $\mathcal{W} \coloneq \{(E, F)  \in \Incl(\KK) \setsep F \in \Fr_\varepsilon(\KK)\}$. Then the family $\mathcal{W}$ is wide: indeed, in Definition \ref{def:wide}, condition \ref{it:wideFirst} is obviouisly satisfied, condition \ref{it:wideSecond} comes from Corollary \ref{cor:amalgPairUnderIsomorphisms}, and condition \ref{it:wideThird} comes from the fact that $\Fr_\varepsilon(\KK)$ is cofinal in $\KK$. Let $\mathcal{G}$ be the set of all $\mu \in \sigma\KK$ satisfying the following property: for every finite-dimensional subspace $E \subseteq V$, either $\mu\restriction_E \notin \Norm(E)$, or there exists a finite-dimensional subspace $F \subseteq V$ with $E \subseteq F$ such that $((E_\Rea, \mu\restriction_E\nolinebreak), (F_\Rea, \mu\restriction_F\nolinebreak)) \in \mathcal{W}$. Then by Lemma \ref{lem:manyWidePairs}, $\mathcal{G}$ is comeager in $\sigma\KK$. It follows from Lemma \ref{lem:perturbationArgument} that for every $\mu \in \PP$ and every finite-dimensional subspace $E \subseteq X_\mu$ and $\gamma > 0$, there exists a finite-dimensional subspace $E' \subseteq V$ such that  $\mu\restriction_{E'} \in \Norm(E')$ and $E \subseteq_\gamma E'_\Rea$; we deduce that for every $\mu \in \mathcal{G}$, the set $\{F \subseteq X_\mu \setsep F \in \Fr_\varepsilon(\KK)\}$ is cofinal in $X_\mu$.

    \smallskip

    By Proposition \ref{prop:AmalgamationsAndGurarii}, $\isomtrclass[\PP]{X}$ is comeager in $\sigma\KK$. Since $\mathcal{G}$ is also comeager, there exists $\mu \in \mathcal{G}$ such that $X_\mu \equiv X$. Thus, the set $\{F \subseteq X \setsep F\in \Fr_\varepsilon(\KK)\}$ is cofinal in $X$. But by Corollary \ref{cor:equivAmalgGFPairs}, the latter set is contained in $\Fr_{2\varepsilon}(X)$, which is hence also cofinal in $X$.
\end{proof}

\begin{theorem}\label{thm:CofFraisseCorresp}
    $X \mapsto \closedAge{X}$ induces a bijection between the collection of all cofinally Fra\"iss\'e Banach spaces (considered up to isometry) and the collection of all cofinally Fra\"iss\'e classes of finite-dimensional Banach spaces.
\end{theorem}

\begin{proof}
    This is an immediate consequence of Theorem \ref{thm:GuardedFraisseCorrespondence} and Lemma \ref{lem:CofFraisseCorresp}.
\end{proof}

\part{Examples and connections with \texorpdfstring{$\omega$}{omega}-categorical Banach spaces}
The second part complements the first part where the general theory of guarded Fra\"iss\' e Banach spaces was developed. We present here new examples and methods how to produce them. This is another place where the notion of guarded Fra\"iss\' e Banach spaces can be appreciated by showing the potential richness of this new class which is demonstrated by particular examples which are well studied spaces in Banach space theory. We remark this is in contrast to the stronger notion of Fra\"iss\' e Banach spaces from \cite{FLT} where it is conjectured that the Gurari\u{\i} space and $L_p([0,1])$, for $p\neq 4,6,8,\ldots$ are the only separable Fra\"iss\' e Banach spaces.

The key idea is the connection of guarded Fra\"iss\' e Banach spaces with the so-called $\omega$-categorical Banach spaces. The latter is a class of Banach spaces studied in model theory which however admits many equivalent definitions and can be presented quite simply without using notions from logic. We take into our advantage the existence of rich literature and results on this subject to obtain our new examples.

This part consists of two sections. Section~\ref{sec:categoricity} introduces $\omega$-categorical Banach spaces and their relation to guarded Fra\"iss\'e Banach spaces and Section~\ref{sec:L_pL_q} completely characterizes for which values $p,q\in\left[1,\infty\right)$ $L_p(L_q)$ is a guarded Fra\"iss\'e Banach space, in particular it provides the new examples.
\section{Basic definitions and preliminary facts}\label{sec:categoricity}

In this section, we shall work with $\omega$-categorical Banach spaces since they are, as mentioned above, an important source of guarded Fra\"iss\' e Banach spaces, or more precisely, their ages are sources of guarded Fra\" iss\' e classes. As the informed reader may guess, these are separable Banach spaces whose theory in a certain metric logic has a unique, up to linear isometry, separable model. Since we do not assume any knowledge of logic, which is beyond the scope of this paper, we provide an alternative definition of these spaces.
The interested reader is referred to \cite{ModelTheoryForMS} and \cite{Ya09} for information about the model theory of Banach spaces.

\begin{definition}Given a metric space $M$ and an action of a group $G$ on $M$ by isometries, for $x\in M$ its \emph{orbit closure} is the set $q(x):=\overline{\{gx\setsep g\in G\}}$. By $M\sslash G$ we denote the set $\{q(x)\setsep x\in M\}$ endowed with the metric $\rho(q(x),q(y)):=\inf\{d(z,z')\setsep z\in q(x), z'\in q(y)\}$. Then it is well-known and easy to prove that $M\sslash G$ is a metric space and the quotient map $q:M\to M\sslash G$ is open.
\end{definition}

The notion below does not depend on the choice of the distance on $X^n$, but for the purpose of this paper when the distance on $X^n$ is not specified we understand the $\ell_\infty$ distance given as $d(x,y):=\max\{\|x_i-y_i\|\setsep i = 1,\ldots,n\}$ for $x,y\in X^n$.

\begin{definition}\label{def:categorical}
Let $X$ be an infinite-dimensional separable Banach space. Then $X$ is \emph{$\omega$-categorical} if the action $\Iso(X)\curvearrowright B_X$ is \emph{approximately oligomorphic}, that is, the quotient space $(B_X)^n\sslash \Iso(X)$ is compact for every $n\in\Nat$.
\end{definition}

 See e.g. \cite{BYU07}, where this is attributed to Henson, for the proof that Definition~\ref{def:categorical} is equivalent with the standard definition of $\omega$-categorical Banach spaces. In fact, the proof follows from \cite[Theorem 12.10 and Corollary 12.11]{ModelTheoryForMS}. For the notion of approximately oligomorphic action we refer the interested reader to \cite{YaTs16} where more details may be found.

 We note here that given a Banach space $X$ and $n\in\Nat$, the quotient space $(B_X)^n\sslash \Iso(X)$ is complete (see e.g. \cite[comment below Definition 2.1]{YaTs16}) and so $(B_X)^n\sslash \Iso(X)$ is compact if and only if it is totally bounded.

\begin{example}\label{ex:omegaCat}Examples of $\omega$-categorical Banach spaces are the following:
\begin{itemize}
    \item $\ell_2$ (this is folklore and there are many possible ways of seeing this: using the model-theoretic equivalent definition mentioned above, it easily follows from the parallelogram law, below we mention even the more general well-known fact that any $L_p([0,1])$ space is $\omega$-categorical; it also follows from the fact that $\ell_2$ is Fra\"iss\'e, see \cite[Example 2.3]{FLT} and the fact that Fra\"iss\'e Banach spaces are $\omega$-categorical, see \cite[Theorem 5.13]{FeRV23})
    \item the Gurari\u{\i} space (see \cite[Theorem 2.3]{YaHen17})
    \item $L_p([0,1])$ for $p\in[1,\infty)$ (see \cite[Proposition 17.4]{ModelTheoryForMS} due to which it is $\omega$-categorical as a Banach lattice and \cite[Corollary 12.13]{ModelTheoryForMS} by which this implies it is also $\omega$-categorical Banach space)
    \item $L_p([0,1]; L_q([0,1]))$ for any $1\leq p,q <\infty$ (if $p\neq q$, by \cite[Proposition 2.6]{HenRay11} it is $\omega$-categorical as a Banach lattice and so using \cite[Corollary 12.13]{ModelTheoryForMS} it is also $\omega$-categorical Banach space; if $p=q$ then it is easy to see that $L_p([0,1]; L_p([0,1]))$ is isometric to $L_p([0,1]^2)$ which in turn is isometric to $L_p([0,1])$ and $L_p([0,1])$ is $\omega$-categorical as noticed above)
    
    \item $C(2^\Nat)$ (this was proved by Henson, but since no direct reference is available, we present (a sketch of) a different simple proof): We fix $n \in \Nat$ and $\varepsilon > 0$ and show that there exists a finite $\varepsilon$-net in $(B_X)^n\sslash \Iso(X)$, where $X = C(2^\Nat)$. In this argument, elements of $(B_X)^n$ will be identified with continuous maps $2^\Nat \to [-1, 1]^n$. Fix a finite $\varepsilon$-net $D \subseteq [-1, 1]^n$. For every $\varnothing \neq A  \subseteq D$, fix a continuous map $f_A \colon 2^\Nat \to [-1, 1]^n$ having range exactly $A$. We will show that classes of $f_A$'s in $(B_X)^n\sslash \Iso(X)$ form an $\varepsilon$-net, and for this, we fix an arbitrary $f \in (B_X)^n$. Then $2^\Nat = \bigcup_{a \in D} f^{-1}(\Ball(a, \varepsilon))$ (here, $\Ball(a, \varepsilon)$ denotes the open ball for the $\ell_\infty$ metric on $[-1, 1]^n$), so using compactness and total zero-dimensionality, we may find clopen sets $C_a \subseteq f^{-1}(\Ball(a, \varepsilon))$, $a \in D$, such that $2^\Nat = \bigsqcup_{a \in D} C_a$. Let $A \coloneq \{a \in D \mid C_a \neq \varnothing\}$. Since $(C_a)_{a \in A}$ and $(f_A^{-1}(\{a\}))_{a \in A}$ are partitions of $2^\Nat$ into nonempty clopen sets, we can find a homeomorphism $\phi \colon 2^\Nat \to 2^\Nat$ such that for every $a \in A$, $\phi(C_a) = f_A^{-1}(\{a\})$. For every $x \in C_a$ we have $f_A\circ \phi(x) = a$ and $f(x) \in \Ball(a, \varepsilon)$; thus, defining $T \in \Iso(X)$ by $T(g) = g\circ \phi$ for all $g \in X$, we have $d(T(f_A), f) < \varepsilon$.
\end{itemize}
\end{example}

The main result of this section is the following.

\begin{theorem}\label{thm:omegaCategoricalIsFraisse}
If a separable Banach space $X$ is $\omega$-categorical, then $\Age(X)$ is guarded Fra\"iss\'e (and in particular, closed) and  $\Flim(X)$ is an $\omega$-categorical space.
\end{theorem}

Actually, we shall prove even a slightly more general results. Namely,
\begin{itemize}
    \item we identify two different conditions guaranteeing that $\closedAge{X}$ is guarded Fra\"iss\'e (see Theorem~\ref{thm:gASufficient}),
    \item we observe that both of those two conditions are satisfied by $\omega$-categorical spaces (see Theorem~\ref{thm:CatImpliesSRU} and Lemma~\ref{lem:catImpliesCountableAmalg}),
    \item we observe that one of those two conditions imply that $\Flim(X)$ is $\omega$-categorical (see Proposition~\ref{prop:flimSepCat}).
\end{itemize}

Apart from the main result mentioned above we also find a sufficient condition for an $\omega$-categorical space to be guarded Fra\"iss\'e, which we shall use in Section~\ref{sec:L_pL_q}.  We remark that a similar result to Theorem~\ref{thm:sufficent_general} independently appeared in \cite[Theorem 5.13]{FeRV23}, where the idea is attributed to Ben Yaacov.

\begin{theorem}\label{thm:sufficent_general}
    Let $X$ be an $\omega$-categorical separable Banach space such that there is a cofinal family $\mathcal{F}$ of finite-dimensional subspaces of $X$ (see Definition~\ref{def:CofFraisse}) satisfying that for any $F\in\mathcal{F}$ and $\eta>0$, $\Iso(X)\curvearrowright\Emb(F,X)$ is $\eta$-transitive.
    Then $X$ is cofinally Fra\"iss\'e (and therefore also guarded Fra\"iss\'e).
\end{theorem}

Let us now proceed to our arguments leading to the proofs of the above mentioned results. We start with the following characterization of $\omega$-categorical spaces.

\begin{proposition}\label{prop:CompacityEmb}
Let $X$ be an infinite-dimensional separable Banach space. Then the following conditions are equivalent.
\begin{enumerate}[label=(Cat\arabic*)]
    \item\label{it:Cat} $X$ is $\omega$-categorical.
    \item\label{it:EmbCpct} For every finite-dimensional normed space $E$ and every $C \geqslant 0$, the quotient $\Emb_C(E, X)\sslash \Iso(X)$ is compact.
    \item\label{it:EmbCpctRelaxed} There is $C'>0$ such that for every $C\in(0,C')$ and $E\in\closedAge{X}$, the quotient $\Emb_C(E, X)\sslash \Iso(X)$ is totally bounded.
\end{enumerate}
\end{proposition}
\begin{proof}
\ref{it:Cat}$\Rightarrow$\ref{it:EmbCpct} Fix $C \geqslant 0$, $E$ be a finite-dimensional normed space, and $e=(e_1, \ldots, e_n)$ an Auerbach basis of $E$. For every family $f = (f_1, \ldots, f_n) \in X^n$, denote by $T_f \colon E \to X$ the unique operator such that for all $1\leqslant i \leqslant n$, $T_f(e_i) = e^C f_i$. The mapping $X^n \to \mathcal{L}(E, X)$, $f \mapsto T_f$ is linear, and it is also continuous; indeed, we have, for all $(a_1, \ldots, a_n) \in \mathbb{R}^n$,
$$\left\|T_f\left(\sum_{i=1}^n a_i e_i\right)\right\| = e^C\cdot\left\|\sum_{i=1}^n a_i f_i\right\| \leqslant e^C\cdot\sum_{i=1}^n |a_i| \cdot \|f_i\|,$$
and since $e$ is Auerbach, we have for every $1 \leqslant i \leqslant n$ that $|a_i| \leqslant \|\sum_{j=1}^n a_j e_j\|$, so:
$$\left\|T_f\left(\sum_{i=1}^n a_i e_i\right)\right\| \leqslant e^C \cdot \left(\sum_{i=1}^n \|f_i\|\right)\left\|\sum_{i=1}^n a_i e_i\right\|,$$
showing that $\|T_f\|\leqslant e^C \cdot \sum_{i=1^n}\|f_i\|$. Now let $A:=\{f \in X^n \setsep T_f \in \Emb_C(E, X)\}$. Since $\Emb_C(E, X)$ is closed in $\mathcal{L}(E, X)$, $A$ is closed in $X^n$; moreover, it is clearly $\Iso(X)$-invariant and contained in $(\Ball_X)^n$. Denote by $q$ the quotient map $\Ball_X^n\to\Ball_X^n\sslash\Iso(X)$. Since $q$ is open and $A\sslash\Iso(X)=q[A]$ and $q[(B_X)^n\setminus A]$ are disjoint since $A$ is $\Iso(X)$-invariant, it follows that $A \sslash \Iso(X)$ is a closed subset of $\Ball_X^n\sslash\Iso(X)$, and thus it is compact. The surjective continuous mapping $A \to \Emb_C(E, X)$, $f \mapsto T_f$ induces a surjective continuous mapping $A \sslash \Iso(X) \to \Emb_C(E, X) \sslash \Iso(X)$; hence $\Emb_C(E, X) \sslash \Iso(X)$ is the continuous image of a compact set, so it is compact.

\ref{it:EmbCpct}$\Rightarrow$\ref{it:EmbCpctRelaxed} is obvious.

\ref{it:EmbCpctRelaxed}$\Rightarrow$\ref{it:Cat} Pick $n\in\Nat$. Since $(B_X)^n\sslash \Iso(X)$ is complete, it suffices to show it is totally bounded, that is, it does not contain an infinite $\varepsilon$-separated sequence. Pick $\varepsilon>0$ and an infinite sequence $(y_k)_{k\in\Nat}\in \big((B_X)^n\big)^\Nat$. In order to get a contradiction, assume that $([y_k])$ is $\varepsilon$-separated in $(B_X)^n\sslash \Iso(X)$. Passing to a subsequence we may assume that the sequence of spaces $E_k:=\Span\{y_k(i)\setsep i\leq n\}$ converges in the Banach-Mazur distance to some $E\in\closedAge{X}$. Pick a basis $(e_i)_{i\leq \dim E}$ of the space $E$ with dual functionals $(f_i)_{i\leq \dim E}$. Fix some $\delta\in\big(0,\min\{\tfrac{C'}{2},1, \frac{\varepsilon}{e(1+\dim E)}\}\big)$. Passing to a subsequence, we may assume there are surjective $e^\delta$-isomorphisms $\eta_k:E_k\to E$, $k\in\Nat$. Since for each $i\in\{1,\ldots,\dim E\}$ and $j\in\{1,\ldots,n\}$ the sequence $\big(f_i(\eta_k(y_k(j)))\big)_{k\in\Nat}$ is bounded, passing to a subsequence we may assume that for each $i\in\{1,\ldots,\dim E\}$, $j\in\{1,\ldots,n\}$ and $k,k'\in\Nat$ we have
\[
|f_i(\eta_k(y_k(j))) - f_i(\eta_{k'}(y_{k'}(j)))| < \delta.
\]
Since $(\eta_k)^{-1}\in\Emb_{\delta}(E,X)$ for every $k\in\Nat$ and $\Emb_{\delta}(E,X)\sslash \Iso(X)$ is totally bounded by the assumption, there are $k,k'\in\Nat$, $k\neq k'$ and $T\in\Iso(X)$ with $\|T\eta_k^{-1} - \eta_{k'}^{-1}\|<\delta$. But then for every $j\in\{1,\ldots,n\}$ we obtain
\[\begin{aligned}
\|Ty_k(j) - y_{k'}(j)\| & \leq \big\|T\big(\eta_k^{-1}(\eta_k(y_k(j)))\big) - \eta_{k'}^{-1}(\eta_k(y_k(j)))\big\| + \|\eta_{k'}^{-1}\|\cdot \|\eta_k(y_k(j)) - \eta_{k'}(y_{k'}(j))\|\\
& \leq \delta\|\eta_k(y_k(j))\| + e\cdot\Big\|\sum_{i=1}^{\dim E} \big(f_i(\eta_k(y_k(j))) - f_i(\eta_{k'}(y_{k'}(j)))\big) e_i\Big\|\\
& \leq e \delta + e \cdot \dim E \cdot \max_{i=1,\ldots,\dim E} |f_i(\eta_k(y_k(j))) - f_i(\eta_{k'}(y_{k'}(j)))|\\
& \leq e\delta(1 + \dim E) < \varepsilon,
\end{aligned}\]
which contradicts that the sequence $([y_k])$ was $\varepsilon$-separated in $(B_X)^n\sslash \Iso(X)$.
\end{proof}

Now, we aim at proving that whenever $X$ is separable and $\omega$-categorical, then $\closedAge{X} = \Age(X)$. Actually, something more general holds.

\begin{definition}
Let $X$ be a separable Banach space. We say that $X$ is \textit{relatively universal} if for every separable Banach space $Y$, if $Y$ is finitely representable in $X$, then $Y$ can be isometrically embedded into $X$.
\end{definition}

\begin{lemma}\label{lem:RUSImpliesClosedAge}
Let $X$ be a relatively universal separable Banach space. Then $\Age(X)=\closedAge{X}$.
\end{lemma}
\begin{proof}
 This is an immediate consequence of the definition of relatively universal separability and of $\closedAge{X}$.
\end{proof}

The following was more-or-less observed also in \cite[Remark 5.16 (i)]{Kha21} (where however the complete proof is missing and additional assumptions are mentioned).

\begin{theorem}\label{thm:CatImpliesSRU}
Every infinite-dimensional $\omega$-categorical separable Banach space $X$ is relatively universal.
\end{theorem}
\begin{proof}
Let $Y$ be a Banach space which is finitely-representable in $X$. We choose a sequence $(E_n)_{n\in\Nat}$ of strictly increasing finite-dimensional subspaces of $Y$ such that $E_1:=E$ and $Y=\overline{\bigcup_{n\in\Nat} E_n}$. Pick a sequence $S_i\in \Emb_{1/i}(E_i,X)$, $i\in\Nat$. By Proposition~\ref{prop:CompacityEmb}, $\Emb_{1/k}(E_k,X)\sslash \Iso(X)$ is totally bounded for every $k\in\Nat$, so using a standard diagonal argument and passing to a subsequence we may assume that for every $k$ and $i,j\geq k$ we have \[d([S_i\upharpoonright E_k], [S_j\upharpoonright E_k])<2^{-k},\] where $d$ is the quotient metric. Then, by definition, we can find $g_2\in\Iso(X)$ such that $\|S_1-g_2\circ S_2\upharpoonright E\|<1/2$. In general, at the $k$-th step, we may find $g_k\in\Iso(X)$ so that $\|g_{k-1}\circ S_{k-1}-g_k\circ S_k\upharpoonright E_{k-1}\|<2^{-k+1}$. We claim that for every $k$ the map $T_k: E_k\to X$ defined by \[T_k(x):=\lim_{k\leq n\to\infty} g_n\circ S_n(x)\] is a linear isometric embedding, and for $i<j$ we have $T_i\subseteq T_j$. Indeed, for each $k$ and $x\in E_k$, the sequence $(g_n\circ S_n(x))_{n\geq k}$ is Cauchy, so $T_k(x)$ is well-defined. Moreover, for every $x,y\in E_k$ we have \[\|T_k(x)-T_k(y)\|=\lim_{n\to\infty}\|g_n\circ S_n(x)-g_n\circ S_n(y)\|=\|x-y\|\] since $g_n\circ S_n\upharpoonright E_k\in \Emb_{1/n}(E_k,X)$. The linearity of $T_k$ and the inclusion $T_k\subseteq T_l$, for $l\geq k$, are clear. Thus the linear isometric embeddings $(T_n)_{n\in\Nat}$ induce a linear isometric embedding $T:Y\to X$.
\end{proof}

Another interesting property shared by $\omega$-categorical spaces is the one, which we call ``the $\omega$-amalgamation property''.

\begin{definition}
We say $\KK\subseteq\BM$ has the \textit{$\kappa$-amalgamation property} if for every $E\in\KK$, $C > 0$, $\varepsilon>0$ and sequences $(F_\alpha)_{\alpha<\kappa}\in \KK^\kappa$ and $(f_\alpha)_{\alpha<\kappa}\in \big(\Emb_C(E,F_\alpha)\big)^\kappa$ there are $\alpha\neq \beta\in \kappa$, $G\in\KK$ and $g_\alpha\in \Emb(F_\alpha,G)$, $g_{\beta}\in\Emb(F_{\beta},G)$ satisfying $\|g_\alpha\circ f_\alpha - g_\beta\circ f_\beta\|<\varepsilon$.

We say $\KK\subseteq\BM$ is \emph{$\kappa$-Fra\"iss\'e class} if it is hereditary, closed, has the joint embedding property and the $\kappa$-amalgamation property.
\end{definition}

\begin{lemma}\label{lem:catImpliesCountableAmalg}
Let $X$ be an $\omega$-categorical separable Banach space. Then $\closedAge{X}$ has the $\omega$-amalgamation property.
\end{lemma}
\begin{proof}Pick $E\in\closedAge{X}$, $C > 0$,  $\varepsilon>0$ and sequences $(F_n)_{n\in\Nat}\in \big(\closedAge{X}\big)^{\Nat}$ and $(f_n)_{n\in\Nat}\in \big(\Emb_{C}(E,F_n)\big)^{\Nat}$.
Note that it follows from Theorem~\ref{thm:CatImpliesSRU} and Lemma~\ref{lem:RUSImpliesClosedAge} that we have $\Age(X)=\closedAge{X}$.
Thus, for every $n\in\Nat$ we may pick $h_n\in \Emb(F_n,X)$. Since $([h_n\circ f_n])_{n\in\Nat}$ is an infinite sequence in $\Emb_{C}(E,X)\sslash \Iso(X)$ which is a totally bounded space by Proposition~\ref{prop:CompacityEmb}, there exists $n\neq m$ and an isometry $T\in\Iso(X)$ satisfying $\|T\circ h_n\circ f_n - h_m\circ f_m\|<\varepsilon$. Put $G=\closedSpan{h_m(F_m)\cup (T\circ h_n)(F_n)}$ and consider $g_n:=T\circ h_n\in \Emb(F_n,G)$ and $g_m:=h_m\in\Emb(F_m,G)$. Then we have
\[
\|g_n\circ f_n - g_m\circ f_m\| = \|T\circ h_n\circ f_n - h_m\circ f_m\| <\varepsilon.
\]\qedhere
\end{proof}

Now, we shall prove that both properties introduced above, and shared by $\omega$-categorical spaces, give us already examples of guarded Fr\"iss\'e Banach spaces, see Theorem~\ref{thm:sufficent_general}. 
The main technical argument is based on an inductive tree-like construction summarized in the following Lemma. We shall use the set-theoretical equality $n = \{0,\ldots,n-1\}$ and for $\sigma\in 2^n$ and $k\leq n$ we put $\sigma|k:=\sigma\restriction_{\{0,\ldots,k-1\}}$ and use the the convention that $\sigma|0:=2^0=\{\emptyset\}$.

\begin{lemma}\label{lem:treeConstruction}
Suppose that $\KK\subseteq\BM$ is closed, hereditary and does not have (GAP). Then there are $A\in \KK$, $\varepsilon>0$, $\{A_\sigma\setsep \sigma\in 2^{<\omega}\}\subseteq\KK$, $\varepsilon':2^{<\omega}\to (0,1)$, $\{C_{\sigma,\tau}\in (0,1)\setsep \tau,\sigma\in 2^{<\omega}, \sigma\subseteq\tau\}$ and $\{\phi^\tau_\sigma\in\Emb_{C_{\sigma,\tau}}(A_\sigma,A_\tau)\setsep \tau,\sigma\in 2^{<\omega}, \sigma\subseteq\tau\}$ such that
\begin{enumerate}[label=(t-\arabic*)]
    \item\label{t:initial} $A_{\{\emptyset\}} = A$
    \item\label{t:commute} for $\sigma\subset\tau\subset\mu\in 2^{<\omega}$ we have $\phi^\mu_\sigma = \phi^\mu_\tau\circ \phi^\tau_\sigma$,
    \item\label{t:constants} for $\sigma\subset\tau\in 2^{<\omega}$ we have $C_{\sigma,\tau}<\tfrac{\varepsilon}{2}$. Moreover, if $i=|\sigma|\geq 1$ then also $C_{\sigma,\tau} < \tfrac{\varepsilon'(\sigma|i-1)}{2}$.
    \item\label{t:keyPart} For every $n\in\omega$, $\sigma\in 2^n$, $E\in\KK$, $\iota_1\in\Emb_{\varepsilon'(\sigma)}(A_{\sigma\smallfrown 0},E)$ and $\iota_2\in\Emb_{\varepsilon'(\sigma)}(A_{\sigma\smallfrown 1},E)$ we have
    \[
    \|\iota_1\circ\phi^{\sigma\smallfrown 0}_{\{\emptyset\}}  - \iota_2\circ\phi^{\sigma\smallfrown 1}_{\{\emptyset\}}\|>\varepsilon.
    \]
\end{enumerate}
\end{lemma}
\begin{proof}
Since $\KK$ does not have (GAP),  using Remark~\ref{rem:equivAmalgPair} and Lemma~\ref{lem:amalgableSelfImprovement}, there are $A\in\KK$ and $\varepsilon>0$ such that for every $B\in\KK$, $\phi\in\Emb_{\varepsilon}(A,B)$ and $\delta>0$, there are $C,D\in\KK$, $\psi_C\in\Emb_{\delta}(B,C)$ and $\psi_D\in\Emb_{\delta}(B,D)$, and $\varepsilon'>0$ such that for every $E\in\KK$, $\iota_C\in\Emb_{\varepsilon'}(C,E)$ and $\iota_D\in\Emb_{\varepsilon'}(D,E)$ we have $\|\iota_C\circ \psi_C\circ \phi-\iota_D\circ \psi_D\circ \phi\|>\varepsilon$.
 
 Pick some $(\varepsilon_n)_{n=1}^\infty\in (0,1)^\Nat$ with $\sum_{n=1}^\infty \varepsilon_n<\tfrac{\varepsilon}{2}$. Now, we shall inductively for every $n\in\omega$ and $\sigma\in 2^n$ construct sets $A_\sigma\in\KK$ together with mappings $\psi_\sigma$ and $\varepsilon',\delta:2^{<\omega}\to (0,1)$ such that 
\begin{enumerate}[label=(c\arabic*), series=constructList]
    \item\label{it:initialStep} $A_{\{\emptyset\}} = A $, $\psi_{\{\emptyset\}} = \Id_A$ and $\delta(\{\emptyset\})=\tfrac{\varepsilon_1}{10}$,
    \item\label{it:deltaNext} for every $n,m\in\omega$ with $1\leq n\leq m$ and $\sigma\in 2^\omega$ we have \[C_{\sigma,n,m}:=\sum_{k=n}^m \delta(\sigma|k) < \tfrac{\varepsilon'(\sigma|n-1)}{2},\]
    and $\delta(\sigma|n) < \varepsilon_{n+1}$,
    \item\label{it:boundPsi} for every $n\in\omega$ and $\sigma\in 2^{n+1}$, $\psi_\sigma\in\Emb_{\delta(\sigma|n)}(A_{\sigma|n}, A_\sigma)$,
    \item\label{it:iotaCondition} for every $n\in\omega$, $\sigma\in 2^n$, $E\in\KK$, $\iota_1\in\Emb_{\varepsilon'(\sigma)}(A_{\sigma\smallfrown 0},E)$ and $\iota_2\in\Emb_{\varepsilon'(\sigma)}(A_{\sigma\smallfrown 1},E)$ we have
    \[
    \|\iota_1\circ\psi_{\sigma\smallfrown 0}\circ\psi_{\sigma}\circ \psi_{\sigma|n-1}\circ\ldots\circ \psi_{\sigma|0}  - \iota_2\circ\psi_{\sigma\smallfrown 1}\circ\psi_{\sigma}\circ\psi_{\sigma|n-1}\circ\ldots\circ \psi_{\sigma|0}\|>\varepsilon.
    \]
\end{enumerate}
Let us describe the construction. We put $A_{\{\emptyset\}}:=A$, $\psi_{\{\emptyset\}}:=\Id_A$ and $\delta(\{\emptyset\})=\tfrac{\varepsilon}{10}$. For the inductive step assume that we have defined for every $\sigma\in \bigcup_{j\leq n} 2^j$ set $A_\sigma$, mapping $\psi_\sigma$, $\delta(\sigma|i)$ for $i=0,\ldots,n$ and $\varepsilon'(\sigma|j)$ for $j=0,\ldots,n-1$. Given $\sigma \in 2^n$, note that the mapping $\phi = \psi_\sigma\circ \psi_{\sigma|n-1}\circ\ldots\circ \psi_{\sigma|0}:A\to A_\sigma$ is isometric embedding for $n=0$ and $\exp(\sum_{i=0}^{n} \delta(\sigma|i))$-isomorphic embedding for $n\geq 1$. Thus, since $\sum_{i=0}^{n} \delta(\sigma|i)\leq \sum_{i=1}^\infty \varepsilon_i<\varepsilon$, we may apply the assumption above with $B=A_{\sigma}$, $\phi$ and $\delta = \delta(\sigma)$ in order to obtain $\varepsilon'(\sigma)>0$, $A_{\sigma\smallfrown 0}, A_{\sigma\smallfrown 1}\in\KK$, $\psi_{\sigma\smallfrown 0}\in\Emb_{\delta(\sigma)}(A_\sigma, A_{\sigma\smallfrown 0})$ and $\psi_{\sigma\smallfrown 1}\in \Emb_{\delta(\sigma)}(A_\sigma, A_{\sigma\smallfrown 1})$ such that \ref{it:iotaCondition} holds. Then \ref{it:boundPsi} is satisfied as well. Finally, we pick $\delta'\in \big(0,\min\{\varepsilon_{n+2},\tfrac{\varepsilon'(\sigma)}{2}\}\big)$ satisfying also that
\[
\forall j=1,\ldots,n:\quad \delta' + \sum_{k=j}^{n}\delta(\sigma|k) < \tfrac{\varepsilon'(\sigma|j-1)}{2},
\]
and put $\delta(\sigma\smallfrown 0) = \delta(\sigma\smallfrown 1) = \delta'$, so \ref{it:deltaNext} holds as well. This finishes the inductive step of the construction.

For $\sigma,\tau\in 2^{<\omega}$ with $\sigma\subseteq\tau$ let us denote by $\phi^\tau_\sigma$ the canonical embedding of $A_\sigma$ into $A_\tau$ given by the formula $\phi^\tau_\sigma:=\psi_{\tau}\circ \psi_{\tau|j-1}\circ\ldots\circ \psi_{\tau|i+1}$, where $i = |\sigma| < j = |\tau|$. Note that then we have $\phi^\tau_\sigma\in\Emb_{C_{\sigma,\tau}}(A_\sigma,A_\tau)$ where $C_{\sigma,\tau}:=\sum_{k=|\sigma|}^{|\tau|-1} \delta(\tau|k)$. By \ref{it:deltaNext}, we obtain that \ref{t:constants} holds and by the construction it is obvious that \ref{t:initial} and \ref{t:commute} hold as well. Finally, note that the condition \ref{it:iotaCondition} may be now reformulated as \ref{t:keyPart}
\end{proof}

\begin{proposition}\label{prop:omegaAmalg}
If $\KK\subseteq\BM$ is closed, hereditary and has $\omega$-amalgamation property, then it has (GAP). 
\end{proposition}
\begin{proof}
If $\KK$ does not have (GAP), let $A$, $\varepsilon>0$, $\varepsilon'$, $\{A_\sigma\setsep \sigma\in 2^{<\omega}\}$, $\{C_{\sigma,\tau}\setsep \sigma\subseteq\tau\in 2^{<\omega}\}$ and $\{\phi^\tau_\sigma\setsep \sigma\subseteq\tau\in 2^{<\omega}\}$ be as in Lemma~\ref{lem:treeConstruction}.

Pick $\sigma\equiv 0 \in 2^\omega$ (that is $\sigma(i) = 0$ for $i\in\omega$) and consider the mappings $\sigma_i:=\sigma|i$ and $f_i := \phi^{\sigma_i\smallfrown 1}_{\{\emptyset\}}\in \Emb_{\varepsilon}(A,A_{\sigma_i\smallfrown 1})$ for every $i\in \omega$. By the $\omega$-amalgamation property of $\KK$, we obtain $i<j$, $E\in\KK$ and $g_i\in \Emb(A_{\sigma_i\smallfrown 1},E)$, $g_j\in\Emb(A_{\sigma_j\smallfrown 1},E)$ satisfying that $\|g_i \circ f_i - g_j  \circ f_j\|<\varepsilon$. Consider now $\iota_2=g_i$ and $\iota_1=g_j\circ \phi^{\sigma_j\smallfrown 1}_{\sigma_i\smallfrown 0}\in\Emb_{C_{\sigma_i\smallfrown 0,\sigma_j\smallfrown 1}}(A_{\sigma_i\smallfrown 0}, E)$. Then, using \ref{t:commute} we obtain 
\[
\|\iota_2\circ \phi^{\sigma_i\smallfrown 1}_{\{\emptyset\}} - \iota_1 \circ \phi^{\sigma_i\smallfrown 0}_{\{\emptyset\}}\| = \|g_i \circ f_i - g_j  \circ \phi^{\sigma_j\smallfrown 1}_{\{\emptyset\}}\|<\varepsilon,
\]
which is in contradiction with \ref{t:keyPart}, because   \ref{t:constants} implies $C_{\sigma_i\smallfrown 0,\sigma_j\smallfrown 1} < \varepsilon'(\sigma_i)$, and therefore we have $\iota_2\in \Emb(A_{\sigma_i\smallfrown 1},E)$ and $\iota_1\in \Emb_{\varepsilon'(\sigma_i)}(A_{\sigma_i\smallfrown 0}, E)$.
\end{proof}

The last ingredient is the following Lemma, which is essentially well-known.

\begin{lemma}\label{lem:finitelyReprUltraPower}
Let $X$ be a Banach space, $I$ a set, $(Y_i)_{i\in I}$ Banach spaces finitely representable in $X$ and $\UU$ an ultrafilter on $I$. Then $(\prod_{i\in I} Y_i)_\UU$ is finitely representable in $X$
\end{lemma}
\begin{proof}
Each $Y_i$ isometrically embeds into $X_i$, where $X_i$ is some ultrapower of $X$ (see e.g. \cite[Theorem 6.3]{Hein80}). Then it follows that $\prod_\UU Y_i$ isometrically embeds into $\prod_\UU X_i$, which is finitely representable in $X$ (see e.g. \cite[Proposition 6.1]{Hein80}).
\end{proof}

\begin{theorem}\label{thm:gASufficient}
Let $X$ be a separable Banach space and let us assume that at least one of the following conditions is satisfied.

\begin{enumerate}
    \item $\closedAge{X}$ has the $\omega$-amalgamation property.
    \item $X$ is relatively universal.
\end{enumerate}

Then $\closedAge{X}$ is guarded Fra\" iss\' e.
\end{theorem}
\begin{proof}
Recall that, by Propositions \ref{prop:HPstableUnderClosures} and \ref{prop:ageAmalgThenItsClosureAlso}, $\closedAge{X}$ is always closed, hereditary and has (JEP), so it remains to prove that is has (GAP).

We have already proved (1) implies the guarded amalgamation property in Proposition~\ref{prop:omegaAmalg}, so we only prove here that (2) implies it as well.

We shall prove something formally stronger. We prove that if there are separable Banach spaces $(Z_n)_{n\in\omega}$ with $\Age(Z_n)\subseteq\closedAge{X}$, $n\in\omega$ such that for every $\varepsilon>0$ and every increasing sequence $(E_k)_{k\in\omega}\in \big(\closedAge{X}\big)^\omega$ there exists $n\in\omega$ and isometric embeddings $(\iota_k)_{k\in\omega}\in\Emb(E_k,Z_n)$ satisfying $\sup_k \|\iota_k\restriction_{E_0} - \iota_0\|\leq \varepsilon$, then $\closedAge{X}$ has the guarded amalgamation property.

Assume, in order to reach a contradiction, that $\closedAge{X}$ does not have the guarded amalgamation property. Let $A$, $\varepsilon\in (0,1)$, $\varepsilon'$, $\{A_\sigma\setsep \sigma\in 2^{<\omega}\}$, $\{C_{\sigma,\tau}\setsep \sigma\subseteq\tau\in 2^{<\omega}\}$ and $\{\phi^\tau_\sigma\setsep \sigma\subseteq\tau\in 2^{<\omega}\}$ be as in Lemma~\ref{lem:treeConstruction}.
 
Fix a non-principal ultrafilter $\UU$ on $\omega$ and for every $\sigma\in 2^\omega$ we let $\ZZ_\sigma$ be the ultraproduct of the sequence $(A_{\sigma|n})_{n\in\omega}$ with respect to $\UU$. Further, for $m\in\omega$ and $z\in A_{\sigma|m}$ put $z_{\sigma,m}:=[\{0\}^{m}\smallfrown (\phi^{\sigma|m+i}_{\sigma|m}(z))_{i=1}^\infty]\in \ZZ_\sigma$ and we let $E_{\sigma,m}\subseteq\ZZ_\sigma$ be a separable subspace defined as $E_{\sigma,m}:=\closedSpan{z_{\sigma, m}\setsep z\in A_{\sigma|m}}$. Note that $E_{\sigma,m}$ is a subspace of $\ZZ_\sigma$, which is finitely representable in $X$ by Lemma~\ref{lem:finitelyReprUltraPower}, so we have $E_{\sigma,m}\in\closedAge{X}$.
 
 Finally, consider mappings $\xi_{\sigma|m}:A_{\sigma|m}\to E_{\sigma|m}$, $m\in\omega$ given by
 \[
 A_{\sigma|m}\ni z\mapsto \xi_{\sigma|m}(z):=[\{0\}^{m}\smallfrown (\phi^{\sigma|m+i}_{\sigma|m}(z))_{i=1}^\infty]\in E_{\sigma,m}
 \]
 and note that those are surjections and since
 by \ref{t:constants} we have
  \[
   (\limsup_{i\to\infty} C_{\sigma|m,\sigma|m+i})\leq \frac{\varepsilon}{2}< \varepsilon,
 \]
 we obtain
 $\xi_{\sigma|m}\in\Emb_{\varepsilon}(A_{\sigma|m},E_{\sigma|m})$. Moreover, by \ref{t:commute} for every $m,i\in\omega$, $i\geq 1$ we have $\xi_{\sigma|m+i} \circ \phi^{\sigma|m+i}_{\sigma|m}= \xi_{\sigma|m}$ and in particular, we obtain
 \[E_{\sigma,m} = \xi_{\sigma|m}(A_{\sigma|m}) = (\xi_{\sigma|m+i} \circ \phi^{\sigma|m+i}_{\sigma|m}) (A_{\sigma|m}) \subseteq\xi_{\sigma|m+i} (A_{\sigma|m+i}) = E_{\sigma|m+i}.
 \]
 
 By the assumption, for every $\sigma\in 2^\omega$ there exists $n(\sigma)\in\omega$ and isometries $\iota_{\sigma,m}\in \Emb(E_{\sigma,m},Z_{n(\sigma)})$, $m\in\omega$ satisfying $\sup_m \|\iota_{\sigma,0} - \iota_{\sigma,m}\restriction_{E_{\sigma,0}}\| \leq \tfrac{\varepsilon}{10}$. Pick $N\in\omega$ such that the set $A_N:=\{\sigma\in 2^\omega\setsep n(\sigma) = N\}$ has cardinality continuum.
 
Since $A_N$ is uncountable and $\Emb_{\varepsilon}(A,Z_N)$ is separable, there must exist $\sigma\neq \sigma'\in A_N$ such that $\|\iota_{\sigma,0}\circ \xi_{\sigma|0}-\iota_{\sigma',0}\circ\xi_{\sigma'|0}\|\leq \tfrac{\varepsilon}{10}$. Let $n\in\omega$ be minimal such that $\sigma(n)\neq \sigma'(n)$ and put $\tau:=\sigma|n$, note that $\{\sigma|n+1,\sigma'|n+1\} = \{\tau^\smallfrown 0,\tau^\smallfrown 1\}$. Let $E$ be the linear hull of the set $(\iota_{\sigma,n+1}\circ \xi_{\sigma|n+1})(A_{\sigma|n+1})\cup (\iota_{\sigma',n+1}\circ \xi_{\sigma'|n+1})(A_{\sigma'|n+1})$. Then we have $E\in\Age(Z_N)\subseteq\closedAge{X}$ and using \ref{t:constants} also $\iota_{\sigma,n+1}\circ \xi_{\sigma|n+1}\in \Emb_{\varepsilon'(\tau)}(A_{\sigma|n+1},E)$ and $\iota_{\sigma',n+1}\circ \xi_{\sigma'|n+1}\in \Emb_{\varepsilon'(\tau)}(A_{\sigma'|n+1},E)$. But we also have
 \[\begin{split}
 \|\iota_{\sigma,n+1}\circ & \xi_{\sigma|n+1}\circ\phi^{\sigma,n+1}_{\{\emptyset\}} -  \iota_{\sigma',n+1}\circ \xi_{\sigma'|n+1}\circ \phi^{\sigma'|n+1}_{\{\emptyset\}}\|  = \|\iota_{\sigma,n+1}\circ \xi_{\sigma|0}-\iota_{\sigma',n+1}\circ\xi_{\sigma'|0}\|\\
 & \leq \|\iota_{\sigma,n+1}\restriction_{E_{\sigma,0}} - \iota_{\sigma,0}\|\cdot \|\xi\restriction_{\sigma,0}\| + \tfrac{\varepsilon}{10} + \|\iota_{\sigma',n+1}\restriction_{E_{\sigma,0}} - \iota_{\sigma',0}\|\cdot \|\xi\restriction_{\sigma',0}\|\\
 & \leq \tfrac{\varepsilon}{10}\cdot 3 + \tfrac{\varepsilon}{10} + \tfrac{3\varepsilon}{10} <\varepsilon,
 \end{split}\]
 which is in contradiction with \ref{t:keyPart}.
\end{proof}

As explained above, Theorem~\ref{thm:gASufficient} implies that $\closedAge{X}$ is guarded Fra\"iss\'e whenever $X$ is $\omega$-categorical. In order to prove Theorem~\ref{thm:omegaCategoricalIsFraisse}, it remains to show that $\Flim(X)$ is $\omega$-categorical.

\begin{proposition}\label{prop:flimSepCat}
Let $X$ be a separable Banach space such that $\closedAge{X}$ is $\omega$-Fra\"iss\'e class. Then $\closedAge{X}$ is guarded Fra\"iss\'e and $\Flim(X)$ is $\omega$-categorical.
\end{proposition}
\begin{proof}
By Theorem~\ref{thm:gASufficient}, we know already that $\closedAge{X}$ is guarded Fra\"iss\'e and so $\Flim(X)$ exists. In order to shorten the notation, in what follows we put $Z=\Flim(X)$. Now, pick $C>1$ and $E\in\closedAge{Z} = \closedAge{X}$. By Proposition~\ref{prop:CompacityEmb}, it suffices to show that $\Emb_{C}(E,Z)\sslash \Iso(Z)$ is totally bounded. Pick a sequence $s_n\in\big(\Emb_{C}(E,Z)\big)^\Nat$ and $\varepsilon'>0$. We need to find $n\neq m$ and an isometry $T\in\Iso(Z)$ with $\|T\circ s_n - s_m\|<\varepsilon'$.

Pick $\xi=\tfrac{\varepsilon'}{3e^C}$. Applying for every $n\in\Nat$ condition \ref{it:gFraisseFourth} to the space $s_n(E)\subseteq Z$ and to the number $\xi>0$, there are $F_n\in\Age(Z) \subseteq\closedAge{X}$, $\delta_n>0$ and $\phi_n\in\Emb_{\xi}(s_n(E),F_n)$ such that $(s_n(E),F_n,\phi_n)$ is $(Z,\xi,\delta_n)$-ultrahomogeneity triple (see Definition~\ref{def:gFpairs}). Put $f_n:=\phi_n\circ s_n\in\Emb_{C + \xi}(E,F_n)$. Applying $\omega$-amalgamation property of $\closedAge{X}$ to the sequence $(\phi_n\circ s_n)\in \big(\Emb_{C+\xi}(E,F_n)\big)^\Nat$ we obtain $n\neq m$, $G\in\closedAge{X}$ and $g_n\in\Emb(F_n,G)$, $g_m\in\Emb(F_m,G)$ satisfying $\|g_m\circ \phi_m\circ s_m - g_n\circ \phi_n\circ s_n\|<\xi$. Fix any $k\in\Emb_{\min\{C,\delta_n,\delta_m\}}(G,Z)$. Since for $i\in\{n,m\}$ we know that $(s_i(E),F_i,\phi_i)$ is $(Z,\xi,\delta_i)$-ultrahomogeneity  triple, there are onto isometries $T_i\in\Iso(Z)$ with $\|T_i\restriction_{s_i(E)} - k\circ g_i \circ \phi_i\|<\xi$. Finally, put $T:=(T_m)^{-1}\circ T_n\in\Iso(Z)$. Then we obtain that
\[\begin{aligned}
\|T\circ s_n - s_m\| & = \|T_n\circ s_n - T_m\circ s_m\| \leq \|T_n\restriction_{s_n(E)} - k\circ g_n \circ \phi_n\|\|s_n\|\\ & \quad  + \|k\circ g_n \circ \phi_n\circ s_n - k\circ g_m \circ \phi_m\circ s_m\| + \|T_m\restriction_{s_m(E)} - k\circ g_m \circ \phi_m\|\|s_m\|\\
& \leq e^C\xi + \|k\circ g_n \circ \phi_n\circ s_n - k\circ g_m \circ \phi_m\circ s_m\| + e^C\xi < 3e^c\xi = \varepsilon'.
\end{aligned}\]
\end{proof}

We finish this Section by proving Theorem~\ref{thm:sufficent_general}. The argument is based on the following lemma, which might be seen as a variant of a result by Schechtman \cite{Sch79} used in the proof that $L_p$ is a Fra\"iss\'e limit of $\{\ell_p^n\}$ from \cite{FLT} (see the proof of \cite[Proposition 3.7]{FLT} which is using the result by Schechtman denoted in \cite{FLT} as Theorem 3.8).

\begin{lemma}\label{lemma:PerturbationIsometry}
Let $X$ be an $\omega$-categorical separable Banach space, $E$ a finite-dimensional subspace of $X$, and $\varepsilon > 0$. Then there exists $\delta > 0$ satisfying the following property: for every $f \in \Emb_{\delta}(E, X)$, there exists $g \in \Emb(E, X)$ with $\|f-g\|\leqslant \varepsilon$.
\end{lemma}

\begin{proof}
Let $A:= \{f \in \Emb_2(E, X) \setsep (\forall g \in \Emb(E, X))(\|f-g\| \geqslant \varepsilon)\}$. For every $\delta \in (0, 1)$, $A \cap \Emb_{\delta}(E, X)$ is a closed subset of $\Emb_2(E, X)$, invariant under the action of $\Iso(X)$. It follows that $B_\delta := (A \cap\Emb_{\delta}(E, X))\sslash\Iso(X)$ is a well-defined closed subset of $\Emb_2(E, X)\sslash\Iso(X)$. Observe that $B_\delta$ decreases when $\delta$ decreases, and that $\bigcap_{\delta \in (0, 1)} B_\delta = (A \cap \Emb(E, X))\sslash \Iso(X) = \varnothing$. Since $\Emb_2(E, X) \sslash \Iso(X)$ is compact by Proposition \ref{prop:CompacityEmb}, it follows that for some $\delta \in (0, 1)$, $B_\delta = \varnothing$. It follows that $A \cap \Emb_{\delta}(E, X) = \varnothing$; hence $\delta$ satisfies the desired property.
\end{proof}

\begin{proof}[Proof of Theorem~\ref{thm:sufficent_general}]
Since $\mathcal{F}$ is cofinal, it suffices to prove that for any $\varepsilon>0$ and any $F\in\mathcal{F}$ there exists $\delta>0$ such that $\Iso(X)\curvearrowright \Emb_\delta(F,X)$ is $\varepsilon$-transitive. Pick $\varepsilon>0$. By Lemma~\ref{lemma:PerturbationIsometry}, there is $\delta>0$ such that for every $f\in\Emb_{\delta}(F,X)$, there exists $g\in\Emb(F,X)$ with $\|f-g\|<\tfrac{\varepsilon}{3}$. Pick $\phi,\psi\in \Emb_\delta(F,X)$ and by the choice of $\delta$, find $\phi',\psi'\in\Emb(F,X)$ with $\|\phi - \phi'\|<\tfrac{\varepsilon}{3}$ and $\|\psi - \psi'\|<\tfrac{\varepsilon}{3}$. Using that $\Iso(X)\curvearrowright \Emb(F,X)$ is $\frac{\varepsilon}{3}$-transitive, pick $T\in\Iso(X)$ satisfying $\|T\psi' - \phi'\|<\tfrac{\varepsilon}{3}$ and then we obtain
\[
\|T\phi - \psi\|\leq \|\phi-\phi'\| + \|T\psi' - \phi'\| + \|\phi-\phi'\|\leq \varepsilon.
\]
Since $\varepsilon>0$ was arbitrary, this shows that $\Iso(X)\curvearrowright \Emb_\delta(F,X)$ is $\varepsilon$-transitive, which finishes the proof.

Finally, we note that cofinally Fra\"iss\'e spaces are guarded Fra\"iss\'e by Proposition~\ref{prop:cofinallySpaces}.
 \end{proof}

\section{The example of \texorpdfstring{$L_p(L_q)$}{Lp(Lq)} spaces}\label{sec:L_pL_q}

In this section, given a Banach space $X$, we denote $L_p(X):=L_p([0,1];X)$ for $p\in[1,\infty)$. Further, by $L_p(L_q)$ we denote the Banach space $L_p(X)$ with $X = L_q([0,1])$. Recall that $\closedAge{L_p(L_q)}$ is guarded Fr\"iss\'e (see Example~\ref{ex:omegaCat} and Theorem~\ref{thm:omegaCategoricalIsFraisse}). The main result of this section is the following identification of the unique guarded Fra\"iss\'e Banach space $X$ satisfying $\closedAge{X} = \closedAge{L_p(L_q)}$.

\begin{theorem}\label{thm:flimLpLq}
Let $1\leq p,q<\infty$. Then
\[
\Flim(L_p(L_q)) = \begin{cases}L_p & \text{ if $q\in \{2,p\}$ or $q\in (p,2)$,}\\
L_p(L_q) & \text{ otherwise.}\end{cases}
\]
\end{theorem}

Note that $L_p(L_q)$ with $p\neq q$ is not isometric to any $L_r$ (this must be well-known, for the argument one can have a look at Theorem~\ref{thm:LpLqiso} below). The proof of Theorem~\ref{thm:flimLpLq} will be given later in this section. There are two interesting consequences.

\begin{corollary}\label{cor:LpLq}
Let $1\leq p\neq q<\infty$. Then
\begin{itemize}
    \item If $q = 2$ or $q\in (p,2)$, then $L_p(L_q)$ is not guarded Fra\"iss\'e.
    \item Otherwise, $L_p(L_q)$ is guarded Fra\"iss\'e Banach space which is not isometric to any $L_r$, $r\in [1,\infty)$.
\end{itemize}    
\end{corollary}
\begin{proof}If $q = 2$ or $q\in (p,2)$, then by Theorem~\ref{thm:flimLpLq} we have $\Flim(L_p(L_q)) = L_p$ and since, by Theorem~\ref{thm:LpLqiso} the space $L_p(L_q)$ is not isometric to $L_p$, using uniqueness of guarded Fra\"iss\'e Banach spaces (see Theorem~\ref{thm:GuardedFraisseCorrespondence}) we obtain that $L_p(L_q)$ is not guarded Fra\"iss\'e.

Otherwise, $L_p(L_q)$ is guarded Fra\"iss\'e by Theorem~\ref{thm:flimLpLq} and $L_p(L_q)$ is not isometric to any $L_r$ by Theorem~\ref{thm:LpLqiso}.
\end{proof}

This first item in Corollary~\ref{cor:LpLq} seems to be interesting, because it is known that $L_p(L_q)$ is $\omega$- categorical Banach lattice (see Example~\ref{ex:omegaCat} above) with quantifier elimination (see \cite{HenRay11}), still by Corollary~\ref{cor:LpLq} this does not guarantee it is a guarded Fra\"iss\'e Banach space. The second item in Corollary~\ref{cor:LpLq} is interesting, because it gives new examples of Banach spaces with a $G_\delta$ isometry class, see \cite[p. 36]{CDDK23} where this has been announced.

In the remainder of this section we prove the above mentioned Theorem~\ref{thm:flimLpLq} (and Theorem~\ref{thm:LpLqiso}, which however we strongly suspect is well-known even though we did not find an explicit reference).

The first step towards the proofs is the following observation.

\begin{lemma}\label{lem:ageOfLpX}
Let $X$ be a separable Banach space, $(E_n)_{n\in\Nat}$ sequence of finite-dimensional spaces satisfying $\closedAge{X} = \overline{\bigcup_{n\in\Nat} \Age(E_n)}$ and let $(\Omega,\AA,\mu)$ be a measure space such that there are infinitely many pairwise disjoint sets of positive measure. Then for $p\in[1,\infty)$ we have
\[\closedAge{L_p(\mu;X)} = \closedAge{\ell_p(X)} = \overline{\bigcup_{n_1,\ldots,n_k\in\Nat, k\in \Nat} \Age(E_{n_1}\oplus_p\ldots \oplus_p E_{n_k})},\]
where the closure is taken in the Banach-Mazur distance.

Moreover, if we denote by $\mathcal{F}$ the family of all the subspaces 
\[
\Span\{\chi_{A_i}(t)x\setsep i=1,\ldots,l,\;x\in S\},
\]
where $A_1,\ldots,A_l$ are pairwise disjoint sets in $(\Omega,\AA,\mu)$ and $S\in\Age(X)$, then $\mathcal{F}$ is cofinal family of finite-dimensional subspaces of $L_p(\mu; X)$ (see Definition~\ref{def:CofFraisse}). 
\end{lemma}
\begin{proof}
We start with the easy observation that whenever $N\in\Nat\cup\{\Nat\}$, $(f_n)_{n\in N}$ is a sequence in $S_{L_p(\mu)}$ such that the sets $\{x\setsep f_n(x)\neq 0\}$, $n\in N$ are pairwise disjoint, then the mapping
\[
\ell_p^N(X)\ni (x_n)_{n=1}^N \longmapsto \sum_{n=1}^N f_n(t)x_n\in L_p(\mu; X)
\]
is linear isometry. Thus, in particular $\ell_p(X)$ is isometric to a subset of $L_p(\mu;X)$ and so we have $\closedAge{L_p(\mu;X)} \supset \closedAge{\ell_p(X)}$.

For the other inclusion pick $E\in \Age(L_p(\mu; X))$, $\varepsilon>0$ and pick a basis $\{F_1,\ldots,F_k\}$ of $E$. Since simple functions are dense in $L_p(\mu; X)$, there are simple functions $\{G_1,\ldots,G_k\}$ such that the mapping $F_i\mapsto G_i$ induces $\phi\in\Emb_{\varepsilon}(E,\Span\{G_1,\ldots,G_k\})$ satisfying $\|\Id_E - \phi\|<\varepsilon$ (see Lemma~\ref{lem:perturbationArgument}). For every $i\leq k$, $G_i$ is of the form $\sum_{j=1}^{n_i} \chi_{B_j}x_j$, where $B_j$'s are sets of positive measure and $(x_j)_{j\leq n_i}\subseteq X$. Set $S(G_i):=\{x_1,\ldots,x_{n_i}\}$ and $S:=\bigcup_{i\leq k} S(G_i)$. Then we can find a finite sequence of pairwise disjoint sets of positive measure $(A_i)_{i=1}^l$  such that $\{G_1,\ldots,G_k\}\subseteq Z:=\Span\{\chi_{A_1}(t)y_1,\ldots,\chi_{A_l}(t)y_l\setsep y_1,\ldots,y_l\in\Span S\}$. Thus, $E$ is $(1+\varepsilon)$-isomorphic to a subspace of $Z$, which itself is isometric to $\ell_p^l(\Span S)$. This proves also the ``Moreover'' part.

Finally, by the paragraph above we have $Z\in\Age(\ell_p^l(X))\subseteq\Age(\ell_p(X))$ and since $\varepsilon>0$ was arbitrary, we obtain $E\in \Age(\ell_p(X))$. Thus, the equality $\closedAge{L_p(\mu;X)} = \closedAge{\ell_p(X)}$ holds.

It remains to prove the second equality. The inclusion ``$\supset$'' is obvious. In order to prove the other one, note that $\bigcup_n \Age(\ell_p^n(X))$ is dense in $\closedAge{\ell_p(X)}$, so it suffices to prove that for every $n\in\Nat$ there are $n_1,\ldots,n_k\in\Nat$ such that $\Age(\ell_p^n(X))\subseteq\closedAge{E_{n_1}\oplus_p\ldots \oplus_p E_{n_k}}$. But this easily follows by induction from the observation that for any separable Banach spaces $Y,Z$ we have $\Age(Y\oplus_p Z)\subseteq\bigcup\{\Age(E\oplus_p F)\setsep E\in \Age(Y),F\in\Age(Z)\}$.
\end{proof}

For the special case of $L_p(L_q)$ spaces we obtain the following, which we shall use later (this is most probably known, see e.g. \cite[Lemma 3.2]{HenRay10} for a similar result in the context of Banach lattices).

\begin{lemma}\label{lem:ageOfLpLq}
For every $p,q\in [1,\infty)$ we have
\[\closedAge{L_p\big(L_q([0,1])\big)} = \overline{\bigcup_{k,n\in \Nat} \Age(\ell_p^n(\ell_q^k))}.\]

Moreover, if we denote by $\mathcal{F}$ the family of all the subspaces 
\[
\Span\{\chi_{A_i}(t)x\setsep i=1,\ldots,l,\; x\in S\},
\]
where $A_1,\ldots,A_l$, $l\geq 3$ are pairwise disjoint measurable sets in $[0,1]$ and $S\subseteq L_q$, $|S|\geq 3$ is a finite set consisting of disjointly supported functions, then $\mathcal{F}$ is cofinal family of finite-dimensional subspaces of $L_p(L_q([0,1]))$.
\end{lemma}
\begin{proof}Equality $\closedAge{L_p\big(L_q([0,1])\big)} = \overline{\bigcup_{k,n\in \Nat} \Age(\ell_p^n(\ell_q^k))}$ follows easily from Lemma~\ref{lem:ageOfLpX} and the fact that $L_q$ is a $\mathcal{L}_{q,1+}$-space.

For the ``Moreover'' part, pick finite-dimensional $E\subseteq L_p(\mu;X)$ and $\varepsilon>0$. Using the ``Moreover'' part of Lemma~\ref{lem:ageOfLpX} we find pairwise disjoint measurable sets $A_1,\ldots,A_l$, $l\geq 3$ in $[0,1]$ and $S'\in\Age(L_q)$ such that for $F':=\Span\{\chi_{A_i}(t)x\setsep i\leq l,\; x\in S'\}$ we have $E\subseteq_\varepsilon F'$. Using the ``moreover'' part of Lemma~\ref{lem:ageOfLpX} once more, we obtain disjointly supported functions $S = \{y_1,\ldots,y_k\}$, $k\geq 3$ in $L_q$ such that $S'\subseteq_\varepsilon \Span S$. Finally, we put $F:=\Span\{\chi_{A_i}(t)x\setsep i\leq l,\; x\in S\}$ and observe that then $E\subseteq_{2\varepsilon} F$.
\end{proof}

Now, we are ready to prove the first part of Theorem~\ref{thm:flimLpLq}.

\begin{proposition}\label{prop:guardedLpLqTrivialCase}
Pick $p,q\in[1,\infty)$ such that $q\in \{2,p\}$ or $q\in (p,2)$. Then $\Flim(L_p(L_q)) = L_p$.
\end{proposition}
\begin{proof}From a classical result, see \cite[Theorem 6.4.18]{AlbiacKalton} we obtain that $\ell_q$ isometrically embeds into $L_p$, so by Lemma~\ref{lem:ageOfLpLq} we obtain that $\closedAge{L_p(L_q)} = \closedAge{L_p}$. Since $L_p$ has $G_\delta$ isometry class by \cite[Theorem B]{CDDK23}, our Theorem~\ref{thm:BMgameForIsometryClasses} implies that $L_p$ is guarded Fra\"iss\'e, so we have $\Flim(L_p(L_q)) = L_p$.
\end{proof}

In order to prove the remaining part of Theorem~\ref{thm:flimLpLq} we shall need some results concerning the Banach-lattice structure of $L_p(L_q)$ spaces. Let us recall that any $f\in L_p(L_q)$ can be identified with a measurable function $\tilde{f}:[0,1]^2\to \Rea$ given as $\tilde{f}(x,y):=f(x)(y)$ for $(x,y)\in [0,1]^2$. It is well-known and easy to check that in this case the Banach-lattice notions from $L_p(X)$ are pointwise, that is, for $f,g \in L_p(L_q)$ we have
\begin{itemize}
    \item $f\leq g$ iff $\tilde{f}(x,y)\leq \tilde{g}(x,y)$ for almost every $(x,y)\in [0,1]^2$;
    \item $f$ and $g$ are disjoint iff $\tilde{f}(x,y)\tilde{g}(x,y)=0$ for almost every $(x,y)\in [0,1]^2$.
\end{itemize}

Our argument is based on two important ingredients. The first one is the recent result by M. A. Tursi, see \cite[Theorem 3.6]{Tursi23}. This will be used later to satisfy the second condition from Theorem~\ref{thm:sufficent_general}.
\begin{theorem}[Tursi]\label{thm:tursi}
Let $p,q\in [1,\infty)$, $p\neq q$ and let $E\in \Age(L_p(L_q))$ be lattice isometric to $\oplus_p(\ell_q^{n_i})_1^N$ for some $N\in\Nat$ and $n_1,\ldots,n_N\in\Nat$. Then for every $\eta>0$ the group of surjective lattice isometries acts $\eta$-transitively on lattice isometries from $E$ into $L_p(L_q)$, that is, for two lattice isometries $f_i:E\to L_p(L_q)$, $i\in \{1,2\}$ there exists a lattice isometry $T\in \Iso(L_p(L_q))$ such that $\|T\circ f_1 - f_2\|<\eta$.
\end{theorem}

In order to reduce the situation to the case of lattice isometries as above, we shall need the following result by Y. Raynaud, see \cite[Proposition 6.3]{Ray18}.

\begin{theorem}[Raynaud]\label{thm:raynaud}
Suppose that $1\leq p\neq q< \infty$, $q\neq 2$ with $q\notin [p,2]$ are given. Then every linear isometry $U:\ell_p^n(\ell_q^m)\to L_p(L_q)$ with $n,m\in \Nat$, $m\geq 3$ preserves disjointness of supports.
\end{theorem}

Finally, we shall need two more observations. The first one is a straightforward fact which we shall use, the second is a consequence of Theorem~\ref{thm:raynaud}.

\begin{lemma}\label{lemma:turnToAbsoluteValue} Let $p,q\in [1,\infty)$, $k\in\Nat$ and $f_1\ldots,f_k\in L_p(L_q)$ be pairwise disjoint. Then there exists $U\in\Iso(L_p(L_q))$ with $U(f_i)\geq 0$ for every $i=1,\ldots,k$.
\end{lemma}
\begin{proof}[Sketch of the proof]
Put $A_i:=\{(x,y)\setsep f_i(x)(y)\neq 0\}$ for $i\leq k$. We may without loss of generality assume that $A_i\cap A_j=\emptyset$ for $i\neq j$. Now, we define $U:L_p(L_q)\to L_p(L_q)$ as
\[
Uf(x)(y):=\begin{cases} f(x)(y) & \text{ if }(x,y)\notin \bigcup_{i=1}^k A_i,\\
\frac{f_i(x,y)}{|f_i(x,y)|}f(x,y) & \text{ if }(x,y)\in A_i \text{ for some } i\leq k.
\end{cases}
\]
Then it is easily checked that $U\in\Iso(L_p(L_q))$ and $Uf_i = |f_i|$ for $i\leq k$.
\end{proof}

The following together with Proposition~\ref{prop:guardedLpLqTrivialCase} proves Theorem~\ref{thm:flimLpLq}.

\begin{theorem}\label{thm:LpLqIsGuarded}
Suppose that $p,q\in [1,\infty)$ are given such that $q\notin \{2,p\}$ and $q\notin (p,2)$. Then $L_p(L_q)$ is cofinally Fra\"iss\'e Banach space (and therefore also guarded Fra\"iss\'e Banach space).
\end{theorem}
\begin{proof}
Let us denote by $\mathcal{F}$ the family of all the subspaces $\Span\{\chi_{A_i}(t)x\setsep i\leq l, x\in S\}$, where $A_1,\ldots,A_l$, $l\geq 3$ are pairwise disjoint sets in $[0,1]$ of positive measure and $S\subseteq L_q$, $|S|\geq 3$ is finite set consisting of functions with pairwise disjoint supports. By Lemma~\ref{lem:ageOfLpX}, $\mathcal{F}$ is cofinal in $L_p(L_q)$, and so using Theorem~\ref{thm:sufficent_general} and the fact that $L_p(L_q)$ is $\omega$-categorical, it suffices to check that the action $\Iso(L_p(L_q))\curvearrowright\Emb(F,L_p(L_q))$ is $\eta$-transitive, for every $F\in\mathcal{F}$ and $\eta>0$. Pick $A_1,\ldots,A_l$, $l\geq 3$ pairwise disjoint sets in $[0,1]$ of positive measure, functions $g_1,\ldots,g_k\in L_q$, $k\geq 3$ with pairwise disjoint supports and consider $F = \Span\{\chi_{A_i}(t)g_j\setsep i\leq l, j\leq k\}\in\mathcal{F}$ and some $\eta>0$. Pick $\psi_\epsilon\in\Emb(F,L_p(L_q)))$, $\epsilon=1,2$. By Lemma~\ref{lemma:turnToAbsoluteValue}, there is an isometry $U\in\Iso(L_p(L_q))$ such that $U(\chi_{A_i}(t)g_j) = |\chi_{A_i}(t)g_j|$ for every $i\leq l$ and $j\leq k$. Since $F$ is isometric to $\ell_p^l(\ell_q^k)$, by Theorem~\ref{thm:raynaud} we have that $\psi_\epsilon$, $\epsilon=1,2$ preserves disjointness of supports, so using Lemma~\ref{lemma:turnToAbsoluteValue} again, for $\epsilon=1,2$ we find $V_\epsilon\in\Iso(L_p(L_q))$ satisfying $V_\epsilon(\psi_\epsilon(\chi_{A_i}(t)g_j)) = |\psi_\epsilon(\chi_{A_i}(t)g_j)|$ for every $i\leq l$ and $j\leq k$. Now it is easy to observe that the mapping $V_\epsilon\circ \psi_\epsilon\circ U^{-1}\in\Emb(G,X)$, $\epsilon=1,2$, is a lattice embedding, where $G = \Span\{|\chi_{A_i}(t)g_j|\setsep i\leq l,\; j\leq k\}$ is lattice isometric to $\ell_p^l(\ell_q^k)$. Thus, by Theorem~\ref{thm:tursi}, there exists $T\in\Iso(L_p(L_q))$ satisfying \[\eta>\|T\circ V_2\circ\psi_2\circ U^{-1}-V_1\circ\psi_1\circ U^{-1}\|=\|V_1^{-1}\circ T\circ V_2\circ\psi_2-\psi_1\|.\] Since $V_1^{-1}\circ T\circ V_2\in\Iso(L_p(L_q))$ this finishes the proof.
\end{proof}

Finally, we provide an argument why $L_p(L_q)$ is not isometric to any $L_r$ space. This should be well-known, but we did not find a satisfactory reference, so we include the argument below. First, we need to realize that both $L_p$ and $L_q$ are isometric to a 1-complemented subspace of $L_p(L_q)$. Let us recall that given $f\in L_p$ and $g\in L_q$, the mapping $f\otimes g\in L_p(L_q)$ is given by $(f\otimes g)(t):=f(t)\cdot g\in L_q$, $t\in [0,1]$.

\begin{proposition}\label{prop:LpInLpLq}
    Let $1 \leqslant p, q < \infty$. The space $L_p$ is isometric to a 1-complemented subspace of $L_p(L_q)$.

    More precisely, the mapping $L_p\ni f\mapsto f\otimes 1\in L_p(L_q)$ is isometry onto $1$-complemented subspace of $L_p(L_q)$.
\end{proposition}

\begin{proof}
    Let us denote by $J:L_p\to L_p(L_q)$ the mapping defined by $J(f):=f\otimes 1$. For every $f\in L_p$ we have
    \[\|Jf\|_{L_p(L_q)}^p = \int_0^1\|J(f)(s)\|_{L_q}^p ds = \int_0^1 |f(s)|^p\cdot \|1\|_{L_q}^p ds = \|f\|_{L_p}^p,\]
    so $J \colon L_p \to L_p(L_q)$ is isometric embedding; denote by $X$ its range. We will show that $X$ is complemented in $L_p(L_q)$.

    \smallskip

    Since we are working on $[0, 1]$, we have $L_q \subseteq L_1$ and for every $u \in L_q$, $\|u\|_{L_1} \leqslant \|u\|_{L_q}$. Hence, the integral functional $\Intg \colon L_q \to \Rea$ is well-defined and continuous. For $F \in L_p(L_q)$, let $Q(F) \coloneq \Intg \circ F$; in other words, $Q(F)$ is a function $[0, 1] \to \Rea$ such that for every $s \in [0, 1]$, we have $Q(F)(s) = \int_0^1 F(s)(t) dt$. The map $Q(F)$ is measurable as the composition of a continuous map and a measurable map. Moreover, for each $s \in [0, 1]$, one has
    $$|Q(F)(s)| = \left|\int_0^1 F(s)(t) dt\right| \leqslant \|F(s)\|_{L_1} \leqslant \|F(s)\|_{L_q},$$
    thus,
    $$\int_{0}^1|Q(F)(s)|^p ds \leqslant \int_{0}^1 \|F(s)\|_{L_q}^p ds = \|F\|_{L_p(L_q)}^p.$$
    It follows that $Q(F) \in L_p$ and that $Q \colon L_p(L_q) \to L_p$ is a continuous linear map with operator norm at most $1$.

    \smallskip

    Finally, we have $Q \circ J = \Id_{L_p}$. Indeed, for $f \in L_p$ and $s \in [0, 1]$
    $$Q \circ J (f)(s) = \int_0^1 J(f)(s)(t) dt = \int_0^1 f(s) dt = f(s).$$
    Thus, $P: = J \circ Q$ is a norm-one projection with range $X$ isometric to $L_p$.
\end{proof}

\begin{proposition}\label{prop:LqInLpLq}
    Let $1 \leqslant p, q < \infty$. The space $L_q$ is isometric to a 1-complemented subspace of $L_p(L_q)$.

    More precisely, the mapping $L_q\ni f\mapsto 1\otimes f\in L_p(L_q)$ is isometry onto $1$-complemented subspace of $L_p(L_q)$.
\end{proposition}

\begin{proof}
    Let us denote by $J:L_q\to L_p(Lq)$ the mapping defined by $J(u):=1\otimes u$. For every $u\in L_q$ we have
    \[\|J(u)\|_{L_p(L_q)}^p = \int_0^1 \|u\|_{L_q}^p ds = \|u\|_{L_q}^p,\]
    hence $J \colon L_q \to L_p(L_q)$ is an isometric embedding. Denote by $X$ its range; we will show that $X$ is complemented in $L_p(L_q)$.

    \smallskip

    Given $F \in L_p(L_q)$, the function $f \colon s \mapsto \|F(s)\|_{L_q}$ is in $L_p$; hence, since we are working on $[0, 1]$, it is integrable and satisfies $\|f\|_{L_1} \leqslant \|f\|_{L_p}$. In particular $F$ is Bochner-integrable; denote by $Q(F)$ its integral. It is an element of $L_q$, and one has
    $$\|Q(F)\|_{L_q} = \left\|\int_0^1F(s)ds\right\|_{L_q} \leqslant \int_0^1 \|F(s)\|_{L_q} ds = \|f\|_{L_1} \leqslant \|f\|_{L_p} = \|F\|_{L_p(L_q)}.$$
    Thus  $Q \colon L_p(L_q) \to L_q$ is a continuous operator with operator norm at most $1$.

    \smallskip

    Finally, we have $Q \circ J = \Id_{L_q}$. Indeed, for $u \in L_q$ one has:
    $$Q\circ J(u) = \int_0^1 J(u)(s) ds = \int_0^1 u ds = u.$$
    Thus, $P:=J\circ Q$ is a norm-one projection with range $X$ isometric to $L_q$.
\end{proof}

Now, using well-known results we are able to deduce that $L_p(L_q)$ is not isometric to any $L_r$ space. 

\begin{theorem}\label{thm:LpLqiso}
    Let $1\leqslant p, q, r < \infty$, $p\neq q$. Then $L_p(L_q)$ is not linearly isometric to $L_r$.
\end{theorem}

\begin{proof}
    It is well-known that $1$-complemented subspaces of $L_p$ spaces are isometric to an $L_p(\mu)$ space for some measure $\mu$ (see \cite{Doug} and \cite{Ando}, where this was proved for $p=1$ and $1<p<\infty$ with $p\neq 2$, respectively). However, by Propositions~\ref{prop:LpInLpLq} and ~\ref{prop:LqInLpLq}, we have that both $L_p$ and $L_q$ are isometric to $1$-complemented subspaces of $L_p(L_q)$.
\end{proof}

Given $1\leq p\neq q<\infty$ with $q\neq 2$ and $q\notin (p,2)$, we know that $L_p(L_q)$ is guarded Fra\"iss\'e (see Corollary~\ref{cor:LpLq}). However, we have the following. We recall from \cite[Definition 2.2]{FLT} that a Banach space $X$ is \emph{approximately ultrahomogeneous (AuH)} if for every finite-dimensional $E\subseteq X$, $\Iso(X)$ acts on $\Emb(E,X)$ $\varepsilon$-transitively for every $\varepsilon>0$, which is equivalent to the fact that the action $\Iso(X)\curvearrowright \Emb(E,X)$ is minimal, i.e. it has no closed invariant subspace.

\begin{proposition}\label{prop:LpLqNotAuH}
Given $1\leq p\neq q<\infty$ with $q\neq 2$ and $q\notin (p,2)$, $L_p(L_q)$ is not (AuH), and thus not weak Fra\"iss\'e.
\end{proposition}
\begin{proof}
The idea of the proof is more-or-less to show that the set
\[\DD:=\{\psi\in\Emb(\ell_2^2,L_p(L_q))\setsep \text{for a.e. }t\in[0,1]\;\forall x,y\in\ell_2^2\;\big(\psi(x)(t),\psi(y)(t)\text{ are collinear}\big)\}\]
is a proper $\Iso(L_p(L_q))$-invariant and closed subspace of $\Emb(\ell_2^2,L_p(L_q))$ which shows that $L_p(L_q)$ is not (AuH). In order to avoid certain technicalities, we shall extract what is essential and prove a technically less-complicated (but still sufficient) statement that 
for $E = \ell_2^2\in\Age(L_p(L_q))$ there exists $\varepsilon_0>0$ and $\phi,\psi\in \Emb(\ell_2^2,L_p(L_q))$ such that for any $T\in\Iso(L_p(L_q))$ we have $\|\phi - T\circ \psi\|\geq \varepsilon_0$.

Since $\ell_2$ embeds isometrically into any $L_p$ space (see e.g. \cite[Proposition 6.4.12]{AlbiacKalton}), we pick functions $f,g\in L_p$ and $u,v\in L_q$ such that both $(f,g)$ and $(u,v)$ are isometric to the $\ell_2^2$-basis. Further, we pick the isometric embeddings from $L_p$ and $L_q$ respectively into $L_p(L_q)$ as in Propositions~\ref{prop:LpInLpLq} and \ref{prop:LqInLpLq}. This way we obtain isometries $\psi,\phi\in\Emb(E,L_p(L_q))$ given for $(a,b)\in \ell_2^2$ as
\[
\psi((a,b)):= a (f\otimes 1) + b (g\otimes 1)\qquad \phi((a,b)):= a (1\otimes u) + b (1\otimes v).
\]
In order to obtain a contradiction, assume there is a sequence of isometries $T_n\in \Iso(L_p(L_q))$, $n\in\Nat$ satisfying $\|\psi - T_n\circ\psi\|\to 0$. By \cite[Theorem 5.1 and Proposition 1.6]{Ray88}, for eveery $n\in\Nat$ there are isometries $S_n\in\Iso(L_p)$ and $(T_{n,t})_{t\in[0,1]}$ where $T_{n,t}\in \Iso(L_q)$ for $t\in [0,1]$ such that for every $(a,b)\in\ell_2^2$ we have
\begin{equation}\label{eq:specialIsometries}
(T_n\circ\psi)((a,b))(t) = \big(aS_nf(t)+bS_ng(t)\big)\cdot T_{n,t}(1).
\end{equation}
We recall the well-known fact that whenever $x_n\to x$ is a Lebesgue-Bochner space $L_p(X)$ then there is a subsequence such that for a.e. $t$ we have $x_n(t)\to x(t)$ in the Banach space $X$\footnote{This can be proved e.g. following the classical proof that $L_p([0,1])$ is complete as presented in \cite[Theorem 3.11]{Rudin}, replacing all the occurances of absolute values $|\cdot|$ by norms $\|\cdot \|_X$.}. Thus, there is an increasing sequence of natural numbers $(n_k)$ such that for a.e. $t\in [0,1]$ we have $T_{n_k}(f\otimes 1)(t)\to (1\otimes u)(t) = u$ and $T_{n_k}(g\otimes 1)(t)\to (1\otimes v)(t) = v$. Further, using \eqref{eq:specialIsometries} we obtain $|S_{n_k}f(t)| = \|T_{n_k}(f\otimes 1)(t)\|$, so $(|S_{n_k}f(t)|)_k$ and similarly also $(|S_{n_k}g(t)|)_k$ are convergent sequences for a.e. $t\in [0,1]$. Thus, for a.e. $t\in [0,1]$ passing to one a further subsequence there are $a:=\lim_k S_{n_k}f(t)$ and $b:=\lim_k S_{n_k}g(t)$ with $|a| = \|u\|$ and $|b| = \|v\|$ for which we obtain
\[
0 = -S_{n_k}g(t)\cdot T_{n_k}(f\otimes 1)(t) + S_{n_k}f(t)\cdot T_{n_k}(g\otimes 1)(t)\to au + bv,
\]
which implies that $u$ and $v$ are not linearly independent, contradiction.
\end{proof}

\section{Open questions}

We present below a list of problems and open questions relative to the notions studied in the present article, together with comments.

\begin{problem}\label{prob:main}Find new examples of infinite-dimensional guarded Fra\"iss\'e (or cofinally Fra\"iss\'e) Banach spaces. In particular, are spaces of the form $L_p(L_q(L_r))$, and $L_p(E)$ where $E$ is finite-dimensional not contained in $L_p$, among such examples?
\end{problem}
The class of guarded Fra\"iss\'e Banach spaces is potentially much richer than the class of Fra\"iss\'e Banach spaces, however the verification of particular cases might be difficult. This is already witnessed on the case of $L_p(L_q)$ spaces from the previous section where the guarded Fra\"iss\'eness of such spaces depends on deep results on their isometries and lattice structures. Recently, Ferenczi and Lopez-Abad announced a proof that if the Banach space $X$ is $\omega$-categorical, then the spaces $C(2^\Nat, X)$ and $L_p(X)$ for $1 \leqslant p < \infty$ are $\omega$-categorical, too. Similarly, it is tempting to conjecture that for every guarded Fra\"iss\'e Banach space $X$ and $p\in\left[1,\infty\right)$ such that $\Age(L_p(X))$ is neither equal to $\Age(L_p)$ nor to $\Age(X)$, $L_p(X)$ is guarded Fra\"iss\'e, but with current tools and methods this is out of reach now. Note, however, that proving the existence of new guarded Fra\"iss\'e Banach spaces is certainly much easier than actually constructing them since, by Theorem \ref{thm:GuardedFraisseCorrespondence}, building new guarded Fra\"iss\'e classes is enough for that. For instance, the aforementioned result by Ferenczi and Lopez-Abad, combined with Theorem \ref{thm:omegaCategoricalIsFraisse}, might easily lead to the identification of new such classes.

\begin{question}\label{quest:equivCAPGAP}
    Does there exist a guarded Fra\"iss\'e Banach space that is not cofinally Fra\"iss\'e?
\end{question}

All currently known examples of guarded Fra\"iss\'e Banach spaces are cofinally Fra\"iss\'e. Note that, by Lemma \ref{lem:CofFraisseCorresp}, it is equivalent to Question \ref{quest:equivCAPGAP} to ask whether there exists a guarded Fra\"iss\'e class of finite-dimensional Banach spaces that is not cofinally Fra\"iss\'e. In the discrete setting, weak Fra\"iss\'e structures (i.e., structures satisfying the discrete analogue of guarded Fra\"iss\'eness) that are not cofinally Fra\"iss\'e are known; see for instance \cite{gFExamples}.
 However, even in that setting it is not easy to find really
natural hereditary classes with (JEP) and the weak amalgamation property but not the cofinal amalgamation property. Thus, Question~\ref{quest:equivCAPGAP} may turn out to be challenging and ingenious ideas may be needed to find a guarded Fra\"iss\'e Banach space that is not cofinally Fra\"iss\'e.

\smallskip

\begin{question}\label{quest:AgeNotGAP}
    Does there exist a separable Banach space $X$ such that $\closedAge{X}$ doesn't satisfy (GAP)? 
\end{question}

In the discrete setting, structures whose age doesn't satisfy the discrete analogue of the (GAP) exist, but are not easy to find. One example is the class of all finite graphs having no cycle of length $4$, see \cite{PanagiotopoulosTent}. Thus, if the answer to Question \ref{quest:AgeNotGAP} turned out to be negative, this situation would be very specific to the Banach space setting.

\smallskip

\begin{question}\label{quest:nearHilbert}
    Fix a separable Banach space $X$ that is near Hilbert (i.e. has type $p$ and cotype $q$ for every $1 \leqslant p < 2 < q \leqslant \infty$), but not isometric to a Hilbert space. Does $\closedAge{X}$ satisfy (GAP)? (CAP)? Is it a (weak) amalgamation class?
\end{question}

There are many examples of near Hilbert, non-Hilbert spaces in the literature: some Orlicz spaces \cite{Katirtzoglou}, all twisted Hilbert spaces \cite{ThreeSpaces}... Question \ref{quest:nearHilbert} can be asked for each of them separately, and the answer might depend on the space. The reason why this question is interesting is that none of the currently known examples of guarded Fra\"iss\'e classes can be the closed age of a near Hilbert, non-Hilbert space. Thus, if for such a space $X$, we manage to determine whether $\closedAge{X}$ does satisfy the (GAP) or not, this would automatically provide either a new example of a guarded Fra\"iss\'e class, or a positive answer to Question \ref{quest:AgeNotGAP}.

\smallskip

\begin{question}\label{quest:GFImpliesOmegaCat}
    Is every guarded Fra\"iss\'e (resp. cofinally Fra\"iss\'e) Banach space $\omega$-categorical?
\end{question}

All currently known examples of guarded Fra\"iss\'e Banach spaces are $\omega$-categorical (see Example \ref{ex:omegaCat}). Additionally, as shown in \cite{FeRV23}, all Fra\"iss\'e Banach spaces are $\omega$-categorical. However, in the discrete setting, there exist very natural examples of cofinally Fra\"iss\'e structures that are not $\omega$-categorical; one of them is the linear graph\footnote{The action of the group of automorphisms of $\Int$ as a linear graph has infinitely many orbits for two element subsets. Indeed, each orbit is determined by the distance between the two elements and the distance can be arbitrarily large. Thus the action is not oligomorphic, so the graph is not $\omega$-categorical. However, $\Aut(\Int)$ acts transitively on $\Emb(G, \Int)$ for every connected subgraph $G \subseteq \Int$, hence $\Int$ is cofinally Fra\"iss\'e.} (\textit{i.e.} the graph with vertex set $\mathbb{Z}$ where to integers are linked by an edge iff they are consecutive). So, once again, if the answer to question \ref{quest:GFImpliesOmegaCat} happens to be positive, this would be very specific to the Banach space setting.

\smallskip

\begin{question}\label{quest:closedAge}
    Does every guarded Fra\"iss\'e (resp. cofinally Fra\"iss\'e) Banach space have a closed age?
\end{question}

Once again, this is the case of all known examples, and every Fra\"iss\'e Banach space has a closed age, as shown in \cite[Theorem 2.12]{FLT}. Observe that, by Theorem \ref{thm:omegaCategoricalIsFraisse}, a positive answer to Question \ref{quest:GFImpliesOmegaCat} would imply a positive answer to Question \ref{quest:closedAge}. Let us also mention two additional consequences of positive answer to Question \ref{quest:closedAge}: 
\begin{itemize}
    \item by \cite[Theorem 2.12]{FLT}, every weak Fra\"iss\'e Banach space would actually be Fra\"iss\'e;
    \item every guarded Fra\"iss\'e class $\KK \subseteq \BM$ would satisfy the following property, seemingly stronger than (GAP): for every $E \in \KK$ and $\varepsilon > 0$,  there exists $F \in \KK$ and $\delta > 0$ such that $(E, F)$ is a $(\KK, \varepsilon, \delta)$-amalgamation pair. This property would follow from Remark \ref{rem:gFFromPairs} and Proposition \ref{prop:ageOfGuardedHasAmalg}, observing that $\KK = \Age(\Flim(\KK))$.
\end{itemize}

\begin{question}\label{quest:kpt}
    Fix $1 \leqslant p, q < \infty$. Is the group $\Iso(L_p(L_q))$, endowed with the strong operator topology (SOT), extremely amenable? If not, does it have a metrizable universal minimal flow?
\end{question}

The celebrated Kechris--Pestov--Todorcevic correspondence \cite{kpt} allows one to prove extreme amenability or to compute universal minimal flows of automorphism groups of ultrahomogeneous discrete structures from the Ramsey theory of their ages (a gentle introduction to the relevant dynamical notions is given at the beginning of \cite{kpt}). Extreme amenability of $\Iso(L_p)$, endowed with SOT, has been proved by Gromov--Milman for $p=2$ \cite{GromovMilman} and by Giordano--Pestov for other values of $1 \leqslant p < \infty$ \cite{GiordanoPestov}, both using measure concentration. The latter result has been reproved in \cite{FLT} using a Banach space version of the Kechris--Pestov--Todorcevic correspondence. Recently, Barto\v{s}--Bice--Dasilva Barbosa--Kubi\'s \cite{weakKPT} extended the Kechris--Pestov--Todorcevic correspondence to discrete weak (\textit{i.e.} guarded) Fra\"iss\'e structures (their setting is actually more general but still doesn't allow one to work with metric structures). We believe that extending methods from \cite{weakKPT} to the Banach space setting and studying the Ramsey properties of $\Age(L_p(L_q))$ could allow one to answer Question \ref{quest:kpt} for values of $p$ and $q$ given by Theorem \ref{thm:LpLqIsGuarded}. It would also be interesting to answer this question for other values of $p$ and $q$; it might also be possible to adapt Giordano--Pestov's method in this case.

\bigskip

\noindent\textbf{Acknowledgments.} The authors would like to thank Antonio Avil\'es, Valentin Ferenczi, Sophie Grivaux, C. Ward Henson, Wies\l aw Kubi\'s, Maciej Malicki, \'Etienne Matheron, Yves Raynaud, Christian Rosendal and Mary Angelica Tursi for insightful discussions during writing of this paper. They would also like to thank Dragan Ma\v{s}ulovi\'c for suggesting the terminology ``guarded Fra\"iss\'e''.

\bibliography{refSimpleSpaces}\bibliographystyle{acm}
\end{document}